\documentclass[11pt,letterpaper,twoside,reqno,nosumlimits,final]{amsart}

\usepackage{xcolor}
\usepackage{fancyhdr}
\usepackage{amsmath,amsfonts,amsbsy,amsgen,amscd,mathrsfs,amssymb,amsthm}
\usepackage{stackengine}

\usepackage{mathtools}

\usepackage{graphicx}
\graphicspath{{figures/}}

\usepackage{tikz}
\usetikzlibrary{arrows,chains,matrix,positioning,scopes}
\usepackage{tikz-3dplot}

\usepackage{xcolor} %

\definecolor{cit}{rgb}{0.91,0.39,0.16}	%

\definecolor{dark-gray}{gray}{0.3}

\definecolor{dkgray}{rgb}{.3,.3,.3}
\definecolor{medgray}{rgb}{.5,.5,.5}
\definecolor{ltgray}{rgb}{.7,.7,.7}
\definecolor{dkblue}{rgb}{0,0,.5}
\definecolor{medblue}{rgb}{0,0,.75}
\definecolor{ltblue}{rgb}{0.97,0.97,1}
\definecolor{rust}{rgb}{0.5,0.1,0.1}
\definecolor{ltyellow}{rgb}{1, 1, 0.9}

\usepackage{soul}
\sethlcolor{ltyellow}
\newcommand{\hilite}[1]{\hl{\textit{#1}}}

\usepackage[normalem]{ulem}

\usepackage{microtype} %
\usepackage[utf8]{inputenc} %
\usepackage[T1]{fontenc} %

\usepackage[english]{babel} %
\usepackage{setspace}
\usepackage{csquotes}

\usepackage[%
style=alphabetic,
datamodel=iso-alphabetic,
giveninits=true,
sorting=nyt, sortcites=false,
autopunct=true, autolang=hyphen,
hyperref=true,
abbreviate=false,
backref=false,
backend=bibtex
]{biblatex}
\defbibheading{bibempty}{}

\renewbibmacro{in:}{%
  \ifentrytype{article}{}{\printtext{\bibstring{in}\intitlepunct}}}

\AtEveryBibitem{%
	\clearfield{issn}%
	\clearfield{isbn}%
	\clearlist{location}%
	\ifentrytype{book}%
		{%
		\clearfield{series}%
		\clearfield{volume}%
		\clearfield{pages}}%
		{%
		\clearname{editor}}%
}

\renewbibmacro*{doi+eprint+url}{%
  \newunit\newblock
  \iftoggle{bbx:doi}
    {\printfield{doi}}
    {}%
  \iftoggle{bbx:eprint}
    {\iffieldundef{doi}{\usebibmacro{eprint}}{}}
    {}%
  \newunit\newblock
  \iftoggle{bbx:url}
    {\iffieldundef{doi}{\iffieldundef{eprint}{\usebibmacro{url}}}} %
    {}}

\AtBeginBibliography{\footnotesize}				%

\usepackage{url} %
\makeatletter
\g@addto@macro{\UrlBreaks}{\UrlOrds}
\makeatother

\usepackage{hyperref}

\hypersetup{hidelinks,colorlinks=true,breaklinks=true,urlcolor=cit,citecolor=dkblue,linkcolor=dkblue,bookmarksopen=false,hypertexnames=false}

\usepackage{bookmark}
\bookmarksetup{
open,
numbered,
addtohook={%
\ifnum\bookmarkget{level}=0 %
\bookmarksetup{bold}%
\fi
\ifnum\bookmarkget{level}=-1 %
\bookmarksetup{color=cit,bold}%
\fi
}
}

\usepackage[full]{textcomp}

\usepackage[scaled=.98,sups,osf]{XCharter}%
\usepackage[scaled=1.04,varqu,varl]{inconsolata}%
\usepackage[type1,scaled=0.98]{cabin}%
\usepackage[utopia,vvarbb,scaled=1.07]{newtxmath}
\usepackage[cal=euler]{mathalfa}
\usepackage{bm} %

\newcommand{\lang}{\textit}

\usepackage{enumitem}

\setlist{noitemsep} %

\setlist[enumerate]{font=\sffamily\bfseries\footnotesize\textcolor{dkgray},label=\arabic*.}
\setlist[itemize]{font=\textcolor{dkgray},label=\small\textcolor{dkgray}\textbullet}

\usepackage[capitalize]{cleveref}

\numberwithin{equation}{section}

\theoremstyle{plain}

\newtheorem{theorem}{Theorem}[section]

\newtheorem{proposition}[theorem]{Proposition}
\newtheorem{fact}[theorem]{Fact}

\theoremstyle{definition}

\newtheorem{definition}[theorem]{Definition}
\newtheorem{example}[theorem]{Example}

\newtheorem*{ideaT}{Theme}
\newtheorem*{problemT}{Problem}

\Crefname{proposition}{Proposition}{Propositions}

\RequirePackage[framemethod=default]{mdframed}

\newmdenv[skipabove=6pt,
skipbelow=6pt,
rightline=false,
leftline=true,
topline=false,
bottomline=false,
backgroundcolor=ltyellow,
linecolor=cit,
innerleftmargin=10pt,
innerrightmargin=10pt,
innertopmargin=0pt,
innerbottommargin=5pt,
leftmargin=0cm,
rightmargin=0cm,
linewidth=4pt]{iBox}	

\newenvironment{idea}{\begin{iBox}\begin{ideaT}}{\end{ideaT}\end{iBox}}
\newenvironment{problem}{\begin{iBox}\begin{problemT}}{\end{problemT}\end{iBox}}

\newmdenv[skipabove=0pt,
skipbelow=0pt,
backgroundcolor=ltblue,
linecolor=dkblue,
linewidth=2pt,
rightline=false,
leftline=false,
topline=false,
bottomline=false,
innerleftmargin=7pt,
innerrightmargin=10pt,
innertopmargin=6pt,
innerbottommargin=6pt,
leftmargin=0cm,
rightmargin=0cm,
innerbottommargin=5pt]{aBox}

\newenvironment{algbox}{\begin{aBox}\begin{algorithmic}}{\end{algorithmic}\end{aBox}}

\newmdenv[skipabove=10pt,
skipbelow=10pt,
backgroundcolor=white,
linecolor=dkblue,
linewidth=0.5pt,
rightline=true,
leftline=true,
topline=true,
bottomline=true,
innerleftmargin=10pt,
innerrightmargin=0.5in,
innertopmargin=5pt,
innerbottommargin=5pt,
leftmargin=0cm,
rightmargin=0cm]{lfBox}

\newenvironment{listbox}{\begin{lfBox}}{\end{lfBox}}

\usepackage{float}
\floatstyle{plaintop}

\numberwithin{figure}{section}
\numberwithin{table}{section}

\newfloat{recipe}{thp}{lor}
\floatname{recipe}{Recipe}
\numberwithin{recipe}{section}

\usepackage{caption}
\usepackage{subcaption}

\usepackage{booktabs,longtable,tabu} %
\usepackage{multirow} %

\usepackage{algorithm,algorithmicx}
\usepackage[noend]{myalgpseudocode}

\algrenewcommand\alglinenumber[1]{\sf\scriptsize\color{dkgray}{#1}}
\algrenewcommand\algorithmicrequire{\textbf{Input:}}
\algrenewcommand\algorithmicensure{\textbf{Output:}}

\algdef{SE}[SUBALG]{Indent}{EndIndent}{}{\algorithmicend\ }
\algtext*{Indent}
\algtext*{EndIndent}

\newcommand{\econst}{\mathrm{e}}

\newcommand{\eps}{\varepsilon}

\renewcommand{\phi}{\varphi}

\newcommand{\vct}[1]{\bm{#1}}
\newcommand{\mtx}[1]{\bm{#1}}
\newcommand{\set}[1]{\mathsf{#1}}

\newcommand{\term}[1]{\textit{#1}}

\newcommand{\sfbf}[1]{\textbf{\textsf{\small{#1}}}}

\newcommand{\N}{\mathbb{N}}

\newcommand{\R}{\mathbb{R}}

\newcommand{\C}{\mathbb{C}}

\newcommand{\M}{\mathbb{M}}
\newcommand{\Sym}{\mathbb{H}}

\newcommand{\range}{\operatorname{range}}

\newcommand{\rank}{\operatorname{rank}}
\newcommand{\trace}{\operatorname{tr}}
\newcommand{\diag}{\operatorname{diag}}
\newcommand{\Id}{\mathbf{I}}

\newcommand{\psdle}{\preccurlyeq}

\newcommand{\abs}[1]{\vert {#1} \vert}
\newcommand{\norm}[1]{\Vert {#1} \Vert}
\newcommand{\ip}[2]{\langle {#1}, \ {#2} \rangle}

\newcommand{\abssq}[1]{\abs{#1}^2}

\newcommand{\normsq}[1]{\norm{#1}^2}
\newcommand{\fnorm}[1]{\norm{#1}_{\mathrm{F}}}
\newcommand{\fnormsq}[1]{\norm{#1}_{\mathrm{F}}^2}

\newcommand{\lip}[2]{\left\langle {#1}, \ {#2} \right\rangle}

\newcommand{\lnorm}[1]{\left\Vert {#1} \right\Vert}

\newcommand{\Expect}{\operatorname{\mathbb{E}}}
\newcommand{\Var}{\operatorname{Var}}

\newcommand{\Probe}{\mathbb{P}}
\newcommand{\Prob}[1]{\Probe\left\{ #1 \right\}}

\newcommand{\condbar}{\, \vert \,}

\newcommand{\normal}{\textsc{normal}}
\newcommand{\uniform}{\textsc{uniform}}

\newcommand{\minimize}{{{\textsf{minimize}}}}

\newcommand{\subjto}{{{\textsf{subject\ to}}}}

\DeclareFontFamily{U}{matha}{\hyphenchar\font45}
\DeclareFontShape{U}{matha}{m}{n}{
  <-6> matha5 <6-7> matha6 <7-8> matha7
  <8-9> matha8 <9-10> matha9
  <10-12> matha10 <12-> matha12
  }{}

\DeclareSymbolFont{matha}{U}{matha}{m}{n}
\DeclareMathSymbol{\abscont}{3}{matha}{"CE}

\DeclareMathOperator*{\intdim}{int\,dim}

\newcommand{\lt}{\left}
\newcommand{\rt}{\right}

\newcommand{\err}{\mathrm{err}}
\newcommand{\dem}{\mathrm{dem}}

\makeatletter
\def\paragraph{\@startsection{paragraph}{4}%
  \z@\z@{-\fontdimen2\font}%
  {\normalfont\scshape}}
\makeatother

\newcommand{\eedit}[2]{\textcolor{red}{\st{#1}}\textcolor{blue}{#2}} %

\addglobalbib{kt23.bib} %

\usepackage[foot]{amsaddr}

\title[Randomized Matrix Computations: Themes and Variations]{Randomized Matrix Computations: \\ Themes and Variations}

\author[A.~Kireeva]{Anastasia Kireeva}
\address[AK]{Department of Mathematics, ETH, Z{\"u}rich, Switzerland.} 
\email{anastasia.kireeva@math.ethz.ch}

\author[J.~A.~Tropp]{Joel A.~Tropp}
\address[JAT]{Department of Computing and Mathematical Science, Caltech, Pasadena, CA, USA.}
\email{jtropp@caltech.edu, https://tropp.caltech.edu}

\date{Lectures: Cetraro, July 3--7, 2023.  Notes: February 20, 2024.  Revised: April 2, 2024.}


\begin{document}

\begin{abstract}
This short course offers a new perspective on randomized algorithms for matrix computations.
It explores the distinct ways in which probability can be used to design algorithms
for numerical linear algebra.
Each design template is illustrated by its application to several computational problems.
This treatment establishes conceptual foundations for randomized numerical linear algebra,
and it forges links between algorithms that may initially seem unrelated.
\end{abstract}

\maketitle

\vspace{-.33in}

\tableofcontents

\vspace{-0.33in}

\newpage

\section{Motivation}

Numerical analysis and probability theory have not always been
on the best interpersonal terms.  Although Monte Carlo methods date back to
the earliest days of numerical computation~\cite{Met87:Beginning-Monte},
researchers have often been skeptical about their performance.
For example, Monte Carlo algorithms for approximating integrals~\cite[Sect.~9.9.3]{QSS07:Numerical-Mathematics-2ed}
can only achieve low precision and their output is highly variable.
This fact about this particular procedure
inspired a general sentiment that randomized algorithms
are unreliable tools for numerical problems:

\begin{listbox}
``Our experience suggests that many practitioners of scientific computing
view randomized algorithms as a desperate and final resort.'' %
\flushright---Halko et al.~\cite[Rem.~1.1]{HMT11:Finding-Structure}.
\end{listbox}

Over the last 20 years, the numerical analysis community has come to
appreciate the value of probabilistic algorithms. %
This sea change can be attributed to the development and
popularization of randomized methods that can solve large-scale
problems efficiently, reliably, and robustly.  In particular,
the randomized SVD~\cite{HMT11:Finding-Structure,TW23:Randomized-Algorithms}
has become a workhorse algorithm for obtaining low-rank matrix approximations in scientific
computing and machine learning.  By now, the numerical analysis literature contains
a diverse collection of randomized algorithms.

This short course offers a new perspective on
\hilite{randomized algorithms for matrix computations}.
Our goal has been to identify distinct conceptual ways
that probability can be used for algorithm design
in numerical linear algebra.
We refer to these design templates as \hilite{themes}.
For each theme, we describe how it allows us to solve
a number of basic computational problems.
We refer to these examples as \hilite{variations}.
It is our hope that this treatment highlights
connections between approaches that may seem different
in spirit.

\subsection{The role of randomness}

Common problems in numerical linear algebra include linear systems,
least-squares problems, eigenvalue problems, matrix approximation, and more.
Before we discuss computational approaches to these problems,
we would like to clarify what we mean by a
\hilite{randomized algorithm} and how it
differs from \hilite{statistical randomness}.

\subsubsection{Algorithmic randomness}

Recall that a \hilite{computer algorithm} is a finite sequence
of well-defined elementary steps that provably returns a solution
to a well-specified computational problem.
In numerical analysis, we allow approximate solutions
that achieve a specified error tolerance.
A \hilite{randomized algorithm} solves each fixed problem instance
by generating random variables and using them as part
of the computational procedure.

\begin{example}[Checking matrix multiplication] \label{ex:matrix-mult-test}
There is a simple randomized algorithm for testing whether
we have correctly multiplied two matrices.  Given square
$n \times n$ matrices $(\mtx{A}, \mtx{B}, \mtx{C})$ with integer entries,
we would like to confirm whether $\mtx{C} = \mtx{AB}$ without performing
a matrix multiplication.  An inexpensive algorithm is to draw a
random vector $\vct{x} \sim \uniform\{ \pm 1 \}^n$
and form the matrix--vector products $\mtx{C}\vct{x}$ and
$\mtx{A}(\mtx{B}\vct{x})$.  If the two vectors
are unequal, then we can confidently report \sfbf{false}.
If the two vectors are equal, then we report \sfbf{true}
even though it is possible that our judgment is wrong.

Since this procedure can fail with nonzero probability,
it is called a \hilite{Monte Carlo-style}
randomized algorithm.  Most of the algorithms we discuss
are Monte Carlo-style methods where the outcome is random.
Fortunately, it is often inexpensive to check whether
they have returned the correct answer.
\end{example}

\begin{example}[Maximum eigenvalue]
The randomized power method is a familiar
algorithm for approximating the maximum eigenvalue of a
positive-semidefinite matrix $\mtx{A}$ to a specified relative error.
It begins with a random vector $\vct{x} \sim \normal(\vct{0}, \Id)$
and repeatedly applies the matrix to the vector to obtain a sequence
of maximum eigenvalue estimates:
\[
\xi_t = \vct{x}^* (\mtx{A}^{t} \vct{x}) / \vct{x}^* (\mtx{A}^{t-1} \vct{x})
\quad\text{for $t = 1,2,3,\dots$.}
\]
After a sufficient number of steps, we hope that the estimate
$\xi_t$ approximates the maximum eigenvalue within a specified
relative error tolerance.
We will see that this approach succeeds with probability one.
See \Cref{sec:rpm} for more details about the randomized
power method.

Although this method eventually succeeds, its \hilite{runtime}
is a random variable.  A procedure with a random runtime is
called a \hilite{Las Vegas-style} randomized algorithm.
We will encounter several more algorithms of this type.
\end{example}

\subsubsection{Statistical randomness}

In sharp contrast, \hilite{statistical randomness} is the
idea of using probability to model a computational problem
or to describe the behavior of an algorithm.
Aside from a short overview here, this course
does not consider the role of statistical randomness
in numerical analysis.

Statisticians often model observations as independent
samples from an underlying population.  This perspective leads
to the concept of \hilite{average-case analysis}, where we ask how
an algorithm performs for a problem instance drawn at random
from some distribution.  For examples in numerical linear
algebra, see Demmel's paper~\cite{Dem88:Probability-Numerical}.

In many scientific fields, data collection involves both measurement
errors and uncertainty.  To address these issues, numerical analysts
study the \hilite{sensitivity} of mathematical problem formulations
to perturbations of the data.
They also study the \hilite{stability} of algorithms under changes in the input.
Historically, these analyses dealt with the worst choice of perturbation
of a given magnitude.  See~\cite{Hig02:Accuracy-Stability} for discussion.

As another application of statistical randomness to numerical analysis,
we can model \hilite{perturbations of problem data} using probability.
We can try to understand how much the answer to a computational
problem is likely to change under a random perturbation.
We can also try to assess how a random perturbation of the problem data impacts the performance of
a particular algorithm, which is called a \hilite{smoothed analysis} of the algorithm~\cite{ST02:Smoothed-Analysis}.
We will see that smoothed analysis has computational implications
for linear algebra (\Cref{sec:smoothed-analysis}).

Numerical analysts have also worked hard to understand the
effects of \hilite{rounding errors} on the output of numerical algorithms.
The worst-case analysis is quite pessimistic,
which motivates us to consider statistical models for
rounding errors.
This idea was already considered in 1951 by
Goldstine \& von Neumann~\cite{GvN51:Numerical-Inverting-II}
in their work on solving linear systems.
Higham and coauthors have recently revisited the
probabilistic analysis of rounding errors~\cite{HM19:New-Approach,CH23:Probabilistic-Rounding}.
Computationally, we can also exploit the beneficial effects of random rounding errors
by using stochastic rounding procedures~\cite{CHM21:Stochastic-Rounding}.

Analysis based on statistical randomness has been a useful paradigm for understanding numerical computation.
Nevertheless, statistical randomness imposes the \hilite{modeling assumption}
that the problem data, perturbations to the problem data,
or errors made during computations can be treated as random quantities.
By contrast, algorithmic randomness is injected into the computational procedure
in a controlled manner by the algorithm designer,
allowing for numerical methods that are provably effective for arbitrary inputs.

\subsubsection{Statistical behavior of randomized algorithms}

Keep in mind that randomized algorithms for matrix computations produce statistical output.
Therefore, we can use tools from statistical theory to understand the
distribution of the output or even to develop more effective algorithms.

\begin{example}[Majority vote]
Consider a randomized algorithm for a decision problem %
that gives the correct result with probability $p > 0.5$.  %
Although we may not have a lot of confidence in the output from a single run of the algorithm,
we can run the algorithm $n$ times and take a majority vote.
(Were the trials more often \sfbf{true} or \sfbf{false}?)
To obtain failure probability below $\delta \in (0,1)$, it is sufficient to
perform $n \geq (p - 0.5)^{-2} \log(1/\delta)$ independent repetitions.
This approach can dramatically increase the reliability of the basic algorithm.
\end{example}

Statistical theory can also be used to equip the output of a randomized
algorithm with a confidence interval; see \Cref{sec:trace-post} for an application
to the problem of trace estimation.  We can also use insights from the
design of statistical estimators to develop better algorithms, e.g.,
whose output is less variable; see \Cref{sec:trace++} for some discussion.

We will not emphasize the nexus between statistics and computation,
but this link offers potential opportunities for supercharging the
methods that we discuss.

\subsection{Prerequisites}

This course is intended for advanced undergraduates, beginning graduate students,
and curious researchers in the mathematical sciences.
On the mathematical side, we assume a thorough knowledge of linear algebra,
matrix analysis, applied probability, and basic statistics.
On the computational side, you should understand the
basic principles of computer algorithms, and we expect
that you have taken a first course on numerical linear
algebra.  If you have followed everything in the introduction so far,
you're probably in good shape.

\subsection{Other resources}

As a classical reference on matrix computations,
the book of Golub \& Van Loan~\cite{GVL13:Matrix-Computations-4ed}
is a comprehensive and trusted source.

The second author has written several surveys and sets of lecture notes
on aspects of randomized matrix computations.  These resources are available
from the website \url{https://tropp.caltech.edu/}.  The current notes borrow
heavily from this body of work.

\begin{itemize}
\item	\textbf{Finding structure with randomness~\cite{HMT11:Finding-Structure}.}  This survey provides a framework for randomized SVD algorithms
for low-rank matrix approximation.  The paper includes pseudocode for reliable algorithms, a number of
applications in scientific computing and machine learning, and the first theoretical analysis that can
predict the detailed behavior of these algorithms.

\item	\textbf{An introduction to matrix concentration inequalities \cite{Tro15:Introduction-Matrix}.}
This monograph provides a comprehensive introduction to matrix concentration inequalities,
which are basic tools for studying random matrices.
It explains how these results can be used to analyze probabilistic methods
for matrix computation, including matrix Monte Carlo methods and randomized
dimension reduction schemes.  These notes support a tutorial given at NeurIPS 2012.

\item	\textbf{Matrix concentration and computational linear algebra \cite{Tro19:Matrix-Concentration-LN}.}
These notes accompany a short course taught at
{\' E}cole Normale Sup{\'e}rieure in Paris in July 2019.  They present some foundational
results on matrix concentration for independent sums and martingales, and they include detailed
applications to quantum state tomography, sparsification of graphs, and the analysis of the
near-linear-time graph Laplacian solver of Kyng \& Sachdeva~\cite{KS16:Approximate-Gaussian}.

\item	\textbf{Randomized algorithms for matrix computations \cite{Tro20:Randomized-Algorithms-LN}.}
These notes were written to support a quarter class on randomized matrix computations taught at Caltech in the Winter 2020 term.
The material includes trace estimation, eigenvalue computations, low-rank matrix approximation,
matrix concentration, matrix Monte Carlo methods, randomized dimension reduction, and more.

\item	\textbf{Randomized numerical linear algebra \cite{MT20:Randomized-Numerical}.}
This survey was invited to the 2020 volume of \textsl{Acta Numerica}.  It covers a significant
proportion of the literature on randomized matrix computations, with a focus on methods that
are reliable in practice.

\item	\textbf{Randomized algorithms for low-rank matrix approximation \cite{TW23:Randomized-Algorithms}.}
This survey updates the randomized SVD paper~\cite{HMT11:Finding-Structure}
to include more modern approaches based on Krylov subspace methods.
It also treats Nystr{\"o}m approximations of psd matrices.  In addition,
the paper showcases scientific applications in genetics and in molecular biophysics.
\end{itemize}

There are several other treatises on randomized matrix computations.
For a theoretical computer science perspective,
see the notes~\cite{Woo14:Sketching-Tool,KV17:Randomized-Algorithms,DM18:Lectures-Randomized}.
The monograph \cite{MDM+23:Randomized-Numerical} discusses the RandBLAS primitives and the RandLAPack
library that are currently under development.

\subsection{Notation}

We employ standard notation that, hopefully, is self-explanatory.
In most places, the initial appearance of a symbol is accompanied
by a reminder about what it means.  If you want to gird yourself first,
this section will acquaint you with the main pieces of notation.

We work primarily in the real field $\R$, although large parts of the discussion
extend to matrices with entries in the complex field $\C$.
All logarithms have the natural base $\econst$.  The Pascal notation
$\coloneqq$ or $\eqqcolon$ generates a definition.

We use lowercase italic letters ($a,b,c$) for scalar variables.
Lowercase bold italic letters ($\vct{a}, \vct{b}, \vct{c}$) refer to vectors,
which are always column vectors.  Uppercase bold italic letters ($\mtx{A}, \mtx{B}, \mtx{C}$)
denote matrix variables.  The symbol $\mathbf{e}_j$ is the $j$th standard
basis vector.  The letter $\Id$ or $\Id_n$ represents a square identity matrix;
the subscript specifies the dimension.
The symbol ${}^*$ is the (conjugate) transpose of a vector or matrix.
The dagger ${}^\dagger$ refers to the Moore--Penrose pseudoinverse
of a matrix.

We equip the Euclidean space $\R^n$ with the standard inner product $\ip{\cdot}{\cdot}$
and the norm $\norm{\cdot}_2$.
There are several ways to express the coordinates of a vector or a matrix.
To denote the $(j, k)$ entry of a matrix $\mtx{A}$,
we may write $a_{jk}$ or use the functional notation $\mtx{A}(j,k)$.
The symbol $\mtx{A}(i, :)$ refers to the $i$th row of the matrix,
while $\mtx{A}(:, j)$ refers to the $j$th row of the matrix.
As a warning, $\vct{a}_k$ may refer to either the $k$th row
or the $k$th column of the matrix, depending on context.
In some settings, we must extract components with respect to bases
other than the standard basis.

The linear space $\M_n(\R)$ comprises all square $n \times n$ matrices.
The real-linear subspace $\Sym_n(\R)$ comprises the real, symmetric matrices
with dimension $n$.  Each of these matrices has $n$ real eigenvalues,
arranged in decreasing order: $\lambda_1 \geq \lambda_2 \geq \dots \geq \lambda_n$.
The cone $\Sym_n^+(\R)$ contains the positive-semidefinite (psd) matrices,
those symmetric matrices whose eigenvalues are all nonnegative.
For a pair $(\mtx{A}, \mtx{B})$ of symmetric matrices, the semidefinite
relation $\mtx{A} \psdle \mtx{B}$ means that $\mtx{B} - \mtx{A}$ is a psd matrix.

For a general matrix $\mtx{B} \in \R^{m \times n}$, we list the singular
values in decreasing order:
\[
\sigma_{\max} \coloneqq \sigma_1 \geq \sigma_2 \geq \dots \geq \sigma_n \eqqcolon \sigma_{\min}.
\]
The spectral condition number is the matrix function $\kappa \coloneqq \sigma_{\max} / \sigma_{\min}$.
We frequently use the Frobenius norm $\fnorm{\cdot}$, the spectral norm $\norm{\cdot}$,
and the trace norm $\norm{\cdot}_{*}$.

Everything takes place in an (unspecified) probability space that is big enough
to support all of the random variables we consider.
We often use letters at the end of the alphabet to refer to
random variables, while maintaining our other conventions.
For instance, $X$ is a scalar random variable,
$\vct{x}$ is a vector-valued random variable,
and $\mtx{X}$ is a random matrix.
The symbol $\sim$ means ``has the distribution''.
Small capitals are reserved for named distributions (\normal, \uniform).
The abbreviation ``i.i.d.'' means independent and identically distributed.
Random variables with hats, such as $\widehat{\vct{x}}$, generally refer
to estimators.

The symbol $\Probe$ reports the probability of an event.
The operator $\Expect$ returns the expectation of a random variable,
computed entrywise for vectors or matrices.  We use subscripts
to denote partial expectations, so $\Expect_X$ is the expectation
with respect to the randomness in $X$.  The partial expectation is only used
when there is no possibility for confusion (e.g., $X$ is independent
from everything else).

We use the computer science interpretation of the $\mathcal{O}$ symbol.
For example, $\mathcal{O}(n)$ is the class of functions
that grow no more quickly than $n$ as $n \to \infty$.

\section{Monte Carlo approximation}

Already in the earliest days of numerical computation, scientists were exploring randomized algorithms for approximating
complicated physical models~\cite{Met87:Beginning-Monte}.  Enrico Fermi is credited with using statistical
sampling methods to simulate neutron diffusion in the 1930s.  Instead of tracking the full dynamical
process, he randomly constructed the next segment of an individual neutron trajectory.
The computations were often performed with a mechanical adding machine.
Independently, in 1946 and 1947, Stanislaw Ulam and John von Neumann proposed a similar model
for neutron diffusion as a trial problem for the newly constructed ENIAC computer.
In late 1947, Nick Metropolis and Klari von Neumann designed an implementation,
and they performed the first tests of the Monte Carlo method on the ENIAC.

Today, every student of numerical analysis learns about Monte Carlo methods as a technique for approximating
a complicated integral, especially of a high-dimensional function~\cite[Sect.~9.9.3]{QSS07:Numerical-Mathematics-2ed}.
In this application, we draw a random variable that has the same distribution as the integrand.
To reduce the variance of this simple estimator,
we can draw independent samples and average them together.

Although Monte Carlo methods are an ancient part of numerical analysis,
their advent in numerical linear algebra is more recent.

\begin{idea}[Monte Carlo approximation]
    Construct a simple unbiased estimator of a complicated linear-algebraic quantity.
    Reduce the variance by averaging independent copies of the estimator.
\end{idea}

\noindent
This section describes how Monte Carlo approximation can be applied to two problems
in matrix computations.
In \Cref{sec:trace-estimation}, we consider how to estimate a specific scalar quantity,
the trace of a square matrix.
In \Cref{sec:matrix-monte-carlo}, we explain how related ideas allow us to estimate a matrix quantity,
such as the product of two matrices.
We also provide citations to papers that formulate other applications of Monte Carlo
approximation in numerical linear algebra.

\subsection{Trace estimation}
\label{sec:trace-estimation}

Let us begin with an approachable example of scalar Monte Carlo approximation in linear algebra.
The goal is to estimate the trace of a square matrix using as few matrix--vector
products as possible.  We will introduce a randomized algorithm for this task
and describe some elements of the analysis.  This approach, due to Girard~\cite{Gir89:Fast-Monte-Carlo},
stands among the oldest examples of probability in matrix computations.
Our treatment follows~\cite{MT20:Randomized-Numerical,Tro20:Randomized-Algorithms-LN}.

\subsubsection{Implicit trace estimation}

Suppose that we are given a square matrix $\mtx{A} \in \M_n(\R)$.
Our goal is to compute the trace:
\[
\trace(\mtx{A}) \coloneqq \sum_{i=1}^n a_{ii}.
\]
Of course, the problem is trivial if we can access the diagonal entries of $\mtx{A}$
at low computational cost.
In this case, revealing and summing the diagonal entries takes
about $\mathcal{O}(n)$ arithmetic operations.

Instead, assume that we have access to the matrix via matrix--vector products (\hilite{matvecs}).
That is, given a vector $\vct{x} \in \R^n$, we can form the product $\vct{x} \mapsto \mtx{A} \vct{x}$
at a reasonable cost.  We will justify this model in \Cref{sec:trace-est-appl}.
For now, our task is to determine $\trace(\mtx{A})$ using as few matvecs as possible.

As before, we can complete this job by using matvecs to read off the diagonal entries of the matrix:
\[
a_{ii} = \mathbf{e}_i^* (\mtx{A} \mathbf{e}_i)
\quad\text{for each $i = 1, 2, 3, \dots, n$.}
\]
The symbol $\mathbf{e}_i$ denotes the $i$th standard basis vector in $\R^n$,
and ${}^*$ is the (conjugate) transpose.
This procedure provides an exact value for the trace, but it requires fully $n$ matvec
operations.  The cost may be prohibitive when $n$ is large.

Fortunately, in many settings, it is acceptable to produce an \hilite{approximation} for the trace
if we can reduce the number of matvecs substantially.

\begin{problem}[Implicit trace estimation]
Suppose that $\mtx{A}$ is a square matrix, accessible via matvecs:
$\vct{x} \mapsto \mtx{A}\vct{x}$.  The task is to approximate $\trace(\mtx{A})$
while limiting the number of matvecs.
\end{problem}

\subsubsection{Applications of trace estimation}
\label{sec:trace-est-appl}

Although the implicit trace estimation problem may seem unnatural, it arises in a wide range of applications.
These domains include statistical computation, network science, chemical graph theory, and quantum statistical mechanics.

\begin{itemize}
\item	\textbf{Smoothing splines \cite{Gir89:Fast-Monte-Carlo,Hut90:Stochastic-Estimator}.}
When fitting a smoothing spline to scattered data,
we need to estimate $\trace(\mtx{A}^{-1})$, the trace of a matrix inverse, 
to perform generalized cross-validation for the smoothing parameter.
In this context, the matrix $\mtx{A}$ is a tridiagonal matrix that models a second-difference
operator.  We can apply the inverse to a vector in $\mathcal{O}(n)$ arithmetic operations
using a Cholesky factorization and triangular substitution.

\item	\textbf{Social network analysis \cite{al2018triangle}.}
We can count triangles
in a graph to measure the centrality of nodes in a social network.
This task reduces to computing $\trace(\mtx{A}^3)$, where
$\mtx{A}$ is the adjacency matrix of the graph.  We can
apply the third power $\mtx{A}^3$ to a vector by repeated
multiplication.

\item	\textbf{Chemical graph theory \cite{estrada2022many}.}  The folding degree of long-chain
proteins can be measured using the \textit{Estrada index}, which is defined as
$\trace(\econst^{\mtx{A}})$.  In this setting, $\mtx{A}$ is a tridiagonal
matrix whose diagonal is related to the dihedral angles between planes
in the folded protein.

\item	\textbf{Quantum statistical mechanics \cite{CH23:Krylov-Aware,ETW24:XTrace}.}  We can use
trace estimation to compute the partition function $\trace( \econst^{-\beta \mtx{H}} )$
and the (unnormalized) average energy $\trace(\mtx{H} \econst^{-\beta \mtx{H}})$ of a quantum mechanical system
described by a Hamiltonian matrix $\mtx{H}$.  We can approximate these exponentials
efficiently using low-degree polynomials or Krylov subspace methods.
\end{itemize}

\noindent
See the survey~\cite{US17:Applications-Trace} for further discussion and examples.

\subsubsection{The role of randomness}

In 1989, Girard~\cite{Gir89:Fast-Monte-Carlo} proposed an elegant Monte Carlo approximation algorithm
for the implicit trace estimation problem.  Hutchinson~\cite{Hut90:Stochastic-Estimator} described
a variant of Girard's method that has become popular.  Both of these authors were concerned with cross-validation
for smoothing splines.  Here is the idea.

Suppose that $\vct{x} \in \R^n$ is a random test vector whose distribution is \hilite{isotropic}:
\[
\Expect[ \vct{xx}^* ] = \Id_n.
\]
Many familiar distributions are isotropic:

\begin{itemize}
\item	\textbf{Standard normal.}  A random vector $\vct{x} \sim \normal(\vct{0}, \Id_n)$ with independent
standard normal entries is isotropic.

\item	\textbf{Uniform on a sphere.}  A random vector $\vct{x} \sim \uniform(\sqrt{n} \mathbb{S}^{n-1})$, distributed
uniformly on the Euclidean sphere with radius $\sqrt{n}$, is isotropic.

\item	\textbf{Independent signs.}  A random vector $\vct{x} \sim \uniform \{ \pm 1 \}^n$ whose entries
are independent and take values $\pm 1$ with equal probability is isotropic.
\end{itemize}

\noindent
These distributions all have a lot of randomness, which helps us build estimators with low variance.

Let $\mtx{A} \in \M_n(\R)$ be a square input matrix.
Taking one matvec with the matrix $\mtx{A}$, we can form the scalar random variable
\[
Y \coloneqq \vct{x}^* (\mtx{A} \vct{x}).
\]
Observe that $Y$ is an \hilite{unbiased estimator} for the trace of the matrix:
\[
\Expect[Y] = \Expect[ \trace( \vct{x}^* \mtx{A} \vct{x}) ]
	= \Expect[ \trace( \mtx{A} \vct{xx}^* ) ]
	= \trace( \mtx{A} \cdot \Expect[ \vct{xx}^* ] )
	= \trace(\mtx{A}).
\]
We have used the cyclic property of the trace, the linearity of the trace
and the expectation, and the assumption that the test vector $\vct{x}$
has an isotropic distribution.  

Although the simple trace estimator $Y$ is correct on average, its fluctuations may be large
in comparison with $\trace(\mtx{A})$.  A standard remedy is to \hilite{average independent
copies} of the simple trace estimator.  Taking $s$ matvecs with the matrix $\mtx{A}$,
we can form the Monte Carlo trace estimator:
\[
\widehat{\trace}_s \coloneqq \frac{1}{s} \sum_{i=1}^s Y_i
\quad\text{where $Y_i \sim Y$ i.i.d.}
\]
By linearity, the trace estimator $\widehat{\trace}_s$ remains unbiased.
By independence of the summands, we have reduced the variance by a
factor of $s$ over the simple estimator.  Altogether,
\begin{equation} \label{eqn:trace-var-redux}
\Expect[ \widehat{\trace}_s ] = \trace(\mtx{A})
\quad\text{and}\quad
\Var[ \widehat{\trace}_s ] = \frac{1}{s} \Var[ Y ].
\end{equation}
Of course, we still need to control $\Var[Y]$, the variance of the simple
estimator, to confirm that the trace estimator $\widehat{\trace}_s$
is precise enough to employ in practice.
\Cref{alg:trace-est} presents pseudocode for the basic method.

\begin{algorithm}[t]
\begin{algbox}[1]
\caption{\textit{Trace estimation: Simplest Monte Carlo method.}}
\label{alg:trace-est}

\Require	Input matrix $\mtx{A} \in \M_n(\R)$, number $s$ of samples
\Ensure		Trace estimate $\widehat{\trace}_s$, an estimate $\widehat{v}_s$ for the variance of $\widehat{\trace}_s$

\vspace{5pt}
\hrule
\vspace{5pt}

\Function{TraceEstimate}{$\mtx{A}$, $s$}

\For{$i = 1, \dots, s$}
\State	Sample isotropic random vector $\vct{x}_i \in \R^n$
\Comment{For example, $\vct{x}_i \sim \uniform \{ \pm 1 \}^n$}
\State	$Y_i = {\vct{x}_i}^* (\mtx{A} \vct{x}_i)$
\Comment{One matvec}
\EndFor

\State	$\widehat{\trace}_s = s^{-1} \sum_{i=1}^s Y_i$
\Comment{Trace estimator}
\State  $\widehat{v}_s = (s(s-1))^{-1} \sum_{i=1}^s (Y_i - \widehat{\trace}_s)^2$
\Comment{Variance estimate for trace estimator}
\EndFunction
\end{algbox}
\end{algorithm}

\subsubsection{Trace estimation: A priori bounds}

It remains to assess how many samples $s$ are sufficient for the Monte
Carlo estimator $\widehat{\trace}_s$ to approximate $\trace(\mtx{A})$
within a specified tolerance.  This analysis hinges on the variance
of the simple trace estimator.

For ease of exposition, assume that the input matrix $\mtx{A} \in \Sym_n^+(\R)$
is a real symmetric \hilite{positive-semidefinite (psd)} matrix that is not zero.
We also assume that the test vector follows the
\hilite{independent sign distribution}: $\vct{x} \sim \uniform\{ \pm 1 \}^n$.

With these assumptions, the simple trace estimator $Y = \vct{x}^* (\mtx{A} \vct{x})$
remains unbiased: $\Expect[ Y ] = \trace(\mtx{A})$.
We can also obtain an exact expression for the variance:
\[
\Var[ Y ] = 2 \sum_{i \neq j} a_{ij}^2 \leq 2 \fnorm{ \mtx{A} }^2.
\]
Here and elsewhere, $\fnorm{\cdot}$ is the Frobenius norm.
Since $\mtx{A}$ is psd, its eigenvalues $\lambda_1, \dots, \lambda_n$ are nonnegative.
Therefore, we have the upper bound
\[
\Var[Y] = 2 \fnormsq{\mtx{A}} = 2 \sum_{i=1}^n \lambda_i^2
	\leq 2 \big(\max\nolimits_{j = 1, \dots, n} \lambda_j \big) \sum_{i=1}^n \lambda_i
		= 2 \norm{\mtx{A}} \trace(\mtx{A}).
\]
As usual, $\norm{\cdot}$ is the spectral norm. %
This inequality compares the variance of the estimator with the trace that we
are trying to estimate.

The variance is a quadratic functional, so we would like to place it on the scale
of the \hilite{squared} trace.  This step yields the bound 
\[
\frac{\Var[Y]}{\trace(\mtx{A})^2} \leq \frac{2 \norm{\mtx{A}}}{\trace(\mtx{A})}.
\]
To interpret this result, let us make a definition.

\begin{definition}[Intrinsic dimension]
For a nonzero psd matrix $\mtx{A} \in \Sym_n^+(\R)$,
the \hilite{intrinsic dimension} is the quantity
\[
\intdim (\mtx A) \coloneqq \frac{\trace( \mtx A )}{\norm{\mtx A}}.
\]
We define $\intdim(\mtx{0}) \coloneqq 0$.
\end{definition}

\noindent
If $\mtx{A} \in \Sym_n^+(\R)$ is nonzero, then $1 \leq \intdim(\mtx{A}) \leq n$.
The lower bound is attained if and only if $\mtx{A}$ has rank one.
The upper bound is attained if and only if $\mtx{A}$ is a scalar matrix: $\mtx{A} = \alpha \Id_n$
for $\alpha > 0$.
More generally, we think about the intrinsic dimension
as a continuous proxy for the rank of the matrix.

It follows that the variance of the simple trace estimator $Y$
is inversely proportional to the intrinsic dimension of the target matrix:
\[
\frac{\Var[Y]}{\trace(\mtx{A})^2} \leq \frac{2}{\intdim(\mtx{A})}.
\]
As a consequence of~\eqref{eqn:trace-var-redux}, the Monte Carlo estimator $\widehat{\trace}_s$
with $s$ independent sign test vectors satisfies the variance bound
\begin{equation} \label{eqn:hutch-var-bd}
\frac{\Var[ \widehat{\trace}_s ]}{\trace(\mtx{A})^2} \leq \frac{2}{s \cdot \intdim(\mtx{A}) }.
\end{equation}
Let us emphasize that this calculation depends on the
distribution of the test vector $\vct{x}$.
We arrive at the following statement.

\begin{proposition}[Monte Carlo trace estimation]
\label{prop:trace-est}
Assume that $\mtx{A} \in \Sym_n^+(\R)$ is a nonzero psd matrix.
Form the Monte Carlo trace estimator $\widehat{\trace}_s$
with $s$ independent sign test vectors:
\[
\widehat{\trace}_s \coloneqq \frac{1}{s} \sum_{i=1}^s \vct{x}_i^* (\mtx{A} \vct{x}_i)
\quad\text{where $\vct{x}_i \sim {\textup \uniform}\{ \pm 1 \}^n$ i.i.d.}
\]
For any tolerance $\eps > 0$, the estimator satisfies the probability bound
\begin{equation} \label{eqn:trace-est-chebyshev}
\Prob{\abs{\widehat{\trace}_s - \trace(\mtx A)} \geq \eps \trace(\mtx{A})}
	\leq \frac{2}{\eps^2 s \cdot \intdim(\mtx{A})}.
\end{equation}
In particular, the probability bound is nontrivial as soon as the number $s$
of samples satisfies
\begin{equation} \label{eqn:trace-est-samples}
s \geq 2 \eps^{-2} (\intdim(\mtx{A}))^{-1}.
\end{equation}
\end{proposition}

\begin{proof}
This result follows from an easy application of Chebyshev's inequality:
\[
\Prob{\abs{\widehat{\trace}_s - \trace(\mtx A)} \geq \eps \trace(\mtx{A})}
	\leq \frac{ \Expect \abs{ \widehat{\trace}_s - \trace(\mtx{A}) }^2 }{\eps^2 \cdot \trace(\mtx{A})^2}
	= \frac{\Var[\widehat{\trace}_s]}{\eps^2 \cdot \trace(\mtx A)^2}
	\leq \frac{2}{\eps^2 s \cdot \intdim(\mtx{A})}.
\]
The equality holds because $\Expect[ \widehat{\trace}_s ] = \trace(\mtx{A})$,
and the last inequality is~\eqref{eqn:hutch-var-bd}.
\end{proof}

\subsubsection{Discussion}

Let us unpack the statement of Proposition~\ref{prop:trace-est}.
A surprising feature of the bound~\eqref{eqn:trace-est-samples} is
that we can reliably estimate the trace of a psd matrix within a factor
of two using just a \hilite{constant} number of samples, say, $s = 16$.
This holds  true regardless of the dimension of the matrix!
For example, if we can perform a matvec with $\mathcal{O}(n)$ arithmetic operations,
then the cost of coarse trace estimation is also $\mathcal{O}(n)$
operations.

Second, observe that the quality of the trace estimator \hilite{increases}
with the intrinsic dimension, the number of ``energetic'' dimensions
in the range of the matrix.  Our intuition suggests
that a random sample provides more information about a high-rank
matrix because it is more likely to be aligned with an energetic
direction.  Thus, it is easier to approximate the trace of high-rank matrix.  

Third, the expression~\eqref{eqn:trace-est-samples} for the number $s$
of samples scales with $\eps^{-2}$.
Thus, it may require an exorbitant number of samples
to achieve a small relative error.  For instance, $\eps = 1/10$
already introduces a factor of $100$ into the bound.  The scaling
is an inexorable consequence of the central limit theorem, rather
than an artifact of the analysis.  This is the curse of
Monte Carlo estimation (but see \Cref{sec:trace++}!).

We can obtain probability bounds much stronger than~\eqref{eqn:trace-est-chebyshev}
by using exponential concentration inequalities in place of Chebyshev's inequality;
for example, see~\cite{AT11:Randomized-Algorithms,GT18:Improved-Bounds}.
While this analysis can tighten the bound on the failure probability,
the basic scaling for the number $s$ of samples
in~\eqref{eqn:trace-est-samples} remains the same.

Finally, let us mention that the best distribution for the test vector
depends on the application.  The independent sign distribution and the
spherical distribution are both excellent candidates. The Gaussian distribution leads to trace estimates with higher variance than both the independent sign distribution and the spherical distribution, so it should usually be avoided.
See~\cite[Sec.~4]{MT20:Randomized-Numerical} for discussion
of some alternative distributions.

\subsubsection{Trace estimation: A posteriori bounds}
\label{sec:trace-post}

We have obtained a prior theoretical bound~\eqref{eqn:trace-est-samples} for the number of
samples $s$ that suffice to achieve relative error $\eps$.  Unfortunately, this bound depends
on the intrinsic dimension, $\intdim(\mtx{A})$, of the target matrix, which is usually not
available to the user of an algorithm.  Lacking this information, how can we determine the number $s$ of samples we need?

Luckily, the Monte Carlo trace estimator is a statistical method, so we can bring statistical
tools to bear.  In particular, the sample variance $\widehat{v}_s$ provides an unbiased estimator for the
variance of the trace estimator.  Recall the definition:
\[
\widehat{v}_s \coloneqq \frac{1}{s(s-1)} \sum_{i=1}^s (Y_i - \widehat{\trace}_s)^2.
\]
If the distribution of the test vector $\vct{x}$ has four finite moments ($\Expect \norm{\vct{x}}_{2}^4 < + \infty$),
then we have the relation
\[
\Expect[ \widehat{v}_s ] = \Var[ \widehat{\trace}_s ].
\]
Therefore, we can employ the sample variance to construct a stopping rule for the
trace estimator.  For a psd input matrix $\mtx{A}$,
to attain a relative error tolerance $\eps > 0$, we can continue drawing samples until
\[
\widehat{v}_s \leq (\eps \cdot \widehat{\trace}_s)^2.
\]
We can also construct confidence intervals for the trace using classical techniques
(e.g., Student's $t$ confidence intervals) or simulation-based methods (e.g., the bootstrap).
See the survey~\cite[Sec.~4]{MT20:Randomized-Numerical} or the
paper~\cite{ET24:Efficient-Error} for some discussion.

\subsubsection{Monte Carlo trace estimation: Improvements}
\label{sec:trace++}

We can refine the simplest Monte Carlo methods by exploiting other insights
from the theory of statistical estimation.  Collectively, these ideas
are called \hilite{variance reduction} techniques~\cite{Liu04:Monte-Carlo}.

At present, the most accurate implicit trace estimators layer
several variance reduction strategies~\cite{MMMW21:Hutch++,ETW24:XTrace}.
The main idea is to use random matvecs to build a low-rank approximation of the input matrix.
The trace of this approximation serves as an initial estimate for the trace
of the input matrix (called a ``control variate'').
We estimate the trace of the residual by forming
additional random matvecs.  A second key idea is to design an exchangeable
estimator that employs the random matvecs in a symmetrical fashion.

\subsubsection{Scalar Monte Carlo: Further examples}

Several other core problems in numerical linear algebra reduce
to trace estimation~\cite[Secs.~4, 5]{MT20:Randomized-Numerical}.
First, we can estimate the squared Frobenius norm of a rectangular
matrix $\mtx{B} \in \R^{m \times n}$ by applying a Monte Carlo trace
estimator to the Gram matrix:
\[
\fnormsq{ \mtx{B} } = \trace( \mtx{B}^* \mtx{B} ).
\]
Similarly, we can estimate the Schatten-4 norm $\norm{\mtx{B}}_{S_4}^4$ by
calculating the sample variance of the trace estimator
for $\mtx{B}^*\mtx{B}$ based on standard normal test vectors.

Estimating higher-order Schatten norms is significantly harder,
but the literature~\cite{LNW14:Sketching-Matrix,KV17:Spectrum-Estimation}
describes some elegant strategies based on Monte Carlo approximation.
The paper~\cite{KV17:Spectrum-Estimation} extends these techniques
to estimate the spectral density of a large matrix.

We can also use Monte Carlo methods to approximate the trace of a matrix
function $\trace (f(\mtx{A}))$ where $\mtx{A}$ is a self-adjoint matrix
and $f$ is a spectral function.
In particular, there is a remarkable technique called stochastic Lanczos
quadrature~\cite{GM10:Matrices-Moments,UCS17:Fast-Estimation} that
addresses this problem by combining the Lanczos tridiagonalization method
with stochastic trace esimation.  See~\cite[Sec.~6]{MT20:Randomized-Numerical}
for an overview.  The paper \cite{CH23:Krylov-Aware} proposes a method
that combines the virtues of stochastic Lanczos quadrature with the variance
reduction techniques mentioned in \Cref{sec:trace++}.

\subsection{Matrix approximation by sampling}
\label{sec:matrix-monte-carlo}

Monte Carlo approximation is often used to estimate scalar quantities,
such as the value of an integral or the trace of a matrix.  The same
strategy is valid in any linear space, including a space of matrices.
In this section, we investigate how to use Monte Carlo methods to
approximate a ``complicated'' matrix.  This problem can arise in
settings where we want to reduce the cost of computing the matrix exactly.
It also allows us to construct a matrix approximation that has
more regularity than the target matrix.

We begin with an abstract treatment of matrix Monte Carlo
before turning to some more concrete examples, including
approximate matrix multiplication and matrix sparsification.
Our presentation follows~\cite{Tro15:Introduction-Matrix}.

\subsubsection{Matrix approximation by sampling}
\label{sec:matrix-approx-setup}

Consider a rectangular target matrix $\mtx{B} \in \R^{m\times n}$.
Suppose that the target matrix can be represented as a sum:
\begin{equation} \label{eqn:matrix-sampling-decomp}
\mtx{B} = \sum_{j=1}^J \mtx{B}_j
\quad\text{where each $\mtx{B}_j \in \R^{m \times n}$.}
\end{equation}
We can imagine that %
each summand $\mtx{B}_j$ has a simpler structure than the target matrix $\mtx{B}$.
For example, we might assume that each $\mtx{B}_j$ is sparse
or that each $\mtx{B}_j$ has low rank.  Of course, we can decompose any matrix
into a sum of sparse matrices or into a sum of low-rank matrices.

We can exploit the decomposition~\eqref{eqn:matrix-sampling-decomp}
to design a Monte Carlo estimator for the target matrix $\mtx{B}$.
The key idea is to sample one of the simple matrices $\mtx{B}_j$ according to
a carefully chosen probability distribution $(p_j : j = 1, \dots, J)$ on the indices.
More precisely, we construct a random matrix $\mtx{Y} \in \R^{m \times n}$
with the distribution 
\begin{equation} \label{eqn:matrix-mc-Y}
\Prob{\mtx{Y} = \smash{p_j^{-1}}  \mtx{B}_j} = p_j
\quad\text{for each $j = 1, \dots, J$.}
\end{equation}
It is easy to check that $\mtx{Y}$ is an unbiased estimator for the target matrix $\mtx{B}$.  Indeed,
\[
\Expect[ \mtx{Y} ] = \sum_{j=1}^J \big(p_j^{-1} \mtx{B}_j \big) \cdot p_j
	= \sum_{j=1}^J \mtx{B}_j
	= \mtx{B}.
\]
Furthermore, the random matrix $\mtx{Y}$ inherits properties shared by the summands,
such as sparsity or low rank. 

As with scalar Monte Carlo methods, the simple estimator $\mtx{Y}$ can be a poor approximation of the target matrix $\mtx{B}$.
To address this concern, we apply the same Monte Carlo approximation strategy:
sample independent copies of $\mtx{Y}$, and average to reduce the variance.
For a given number $s$ of samples, the matrix Monte Carlo estimator takes the form
\begin{equation} \label{eqn:Bhat-s}
\widehat{\mtx B}_s \coloneqq \frac{1}{s} \sum_{i=1}^s \mtx{Y}_i
\quad\text{where $\mtx{Y}_i \sim \mtx{Y}$ i.i.d.}
\end{equation}
When the number $s$ of samples is small, %
we anticipate that the Monte Carlo approximation %
can be computed inexpensively.
Moreover, the approximation may enjoy a simpler structure than the target matrix.
For example, if each $\mtx{B}_j$ has rank one,
then the approximation $\widehat{\mtx{B}}_s$ has rank $s$.

\subsubsection{Matrix sampling: A priori bounds}

We are interested in determining the number $s$ of samples that suffice
to achieve a precise approximation of the target matrix.
For a tolerance $\eps > 0$, we would like to ensure that the expected relative error satisfies
\begin{equation}\label{eqn:Bhat-rel-error}
\Expect \left[ \frac{\norm{\widehat{\mtx{B}}_s - \mtx{B}}}{\norm{\mtx{B}}} \right] \leq \eps. 
\end{equation}
The desideratum~\eqref{eqn:Bhat-rel-error} is analogous
with the variance bound~\eqref{eqn:hutch-var-bd}
for the trace estimator.
The spectral-norm error provides very strong control on the quality
of the matrix approximation.  It allows us to compare linear functionals
of the two matrices, their singular values, their singular vectors,
and more~\cite[Sec.~4.2]{TW23:Randomized-Algorithms}.

To analyze the dependence of the relative error~\eqref{eqn:Bhat-rel-error} on the number $s$ of samples,
we rely on the matrix Monte Carlo theorem~\cite[Cor.~6.2.1]{Tro15:Introduction-Matrix}.

\begin{theorem}[Matrix Monte Carlo; Tropp 2015]
\label{thm:matrix-mc}
    Define $\mtx{B} \in \R^{m\times n}$ and $\mtx{Y} \in \R^{m\times n}$
    as in~\Cref{sec:matrix-approx-setup}.
    For the random matrix $\mtx{Y}$,
    define the per-sample second moment $v(\mtx{Y})$
    and the uniform norm bound $L(\mtx{Y})$:
    \[
    v(\mtx{Y}) \coloneqq \max\{ \norm{\Expect[ \mtx{Y} \mtx{Y}^* ]}, \norm{\Expect[ \mtx{ Y }^* \mtx{ Y } ]}\}
    \quad\text{and}\quad
    L(\mtx{Y}) \coloneqq \sup \norm{\mtx{Y}}.
    \]
    Then the randomized matrix approximation $\widehat{\mtx{B}}_s$ based on $s$ samples~\eqref{eqn:Bhat-s}
    satisfies the error bound
    \[
    \Expect \norm{\widehat{\mtx{B}}_s - \mtx{B}}
    	\leq \sqrt{ \frac{2 v(\mtx Y) \log(m + n)}{s} } + \frac{L(\mtx{Y}) \log(m+n)}{s}.
    \]
    For a tolerance $\eps \in (0,1]$, the
    relative error bound~\eqref{eqn:Bhat-rel-error} holds
    when the sample complexity $s$ satisfies
	\begin{equation} \label{eqn:mtx-approx-samples}
	s \geq 4 \max\left\{ \frac{\eps^{-2} v(\mtx Y)}{\normsq{\mtx B}}, \frac{\eps^{-1} L(\mtx{Y})}{\norm{\mtx B}} \right\}
		\cdot \log(m+n). 
	\end{equation}
	See~\cite[Cor.~6.2.1]{Tro15:Introduction-Matrix} for probability bounds.
\end{theorem}

\Cref{thm:matrix-mc} is a corollary of the matrix Bernstein inequality~\cite[Thm.~6.1.1]{Tro15:Introduction-Matrix},
a fundamental tool from nonasymptotic random matrix theory.

To understand this result, observe that the sample complexity $s$
depends on the two statistics $v(\mtx{Y}) / \normsq{\mtx{B}}$ and $L(\mtx{Y}) / \norm{\mtx{B}}$.
These scale-free quantities reflect the fluctuations of the simple approximation
$\mtx{Y}$ relative to the spectral norm $\norm{\mtx{B}}$ of the target matrix.  To obtain
an accurate approximation, we must tune the probability distribution
$(p_j : j = 1, \dots, J)$ on the indices to control both quantities.

Next, compare the sample complexity~\eqref{eqn:mtx-approx-samples}
for matrix approximation with the sample complexity~\eqref{eqn:trace-est-samples}
for trace estimation.  As $\eps \to 0$,
the first term in the matrix case exhibits the same quadratic growth $\eps^{-2}$ as we see in the scalar case.
Once again, this phenomenon reflects the central limit theorem.
On the other hand, when $\eps \geq v(\mtx{Y}) / L(\mtx{Y})$,
the linear term $\eps^{-1}$ dominates.  This feature of the bound
results from the possibility that a few samples with large norm can ruin the approximation.

We also pay a small factor $\log(m + n)$ that is logarithmic in the total
dimension of the target matrix.  This term arises because we are
bounding the spectral-norm error, and it cannot be removed in general.
The factor can often be replaced with a quantity that
reflects the intrinsic dimension of the target
matrix $\mtx{B}$; see~\cite[Chap.~7]{Tro15:Introduction-Matrix}.

Owing to concentration of measure phenomena, the expectation bound reported in~\cref{thm:matrix-mc}
captures the essential behavior of the matrix Monte Carlo approximation.
Probability bounds for Monte Carlo approximation are more complicated,
so they tend to obscure rather than enlighten.  For related reasons,
this short course focuses on expectations instead of probabilities.

\begin{algorithm}[t]
\begin{algbox}[1]
\caption{\textit{Approximate matrix multiplication: Monte Carlo method.}}
\label{alg:approx-matrix-mult}

\Require	Input matrix $\mtx{A} \in \R^{n\times d}$, number $s$ of samples
\Ensure		Approximation $\widehat{\mtx{B}}_s$ for the product $\mtx{B} = \mtx{AA}^*$

\vspace{5pt}
\hrule
\vspace{5pt}

\Function{ApproxMatrixMult}{$\mtx{A}$, $s$}

\State	$c_j = \norm{\vct{a}_{j}}_2^2$ for $j = 1,\dots,d$
\Comment	$\vct{a}_j$ is $j$th \hilite{column} of matrix %
\State	$p_j = c_j / \sum_{i=1}^d c_i$ for $j = 1,\dots,d$
\Comment	Sampling probabilities

\For{$i = 1, \dots, s$}
\State	Sample $k_i \in \{1,\dots,d\}$ from distribution $(p_j)$ %
\State	$\mtx{Y}_i = p_{k_i}^{-1} \vct{a}_{k_i} \smash{\vct{a}_{k_i}^*}$
\Comment	Sample random matrix
\EndFor

\State	$\widehat{\mtx{B}}_s = s^{-1} \sum_{i=1}^s \mtx{Y}_i$
\Comment	Estimate for matrix product

\EndFunction
\end{algbox}
\end{algorithm}

\subsubsection{Example: Approximate matrix multiplication}
\label{sec:approx-mtx-mult}

As an illustrative example of matrix Monte Carlo approximation,
we study a randomized algorithm for approximate matrix multiplication (\Cref{alg:approx-matrix-mult}).
This approach was initially proposed by Drineas \& Kannan~\cite{DK01:Fast-Monte-Carlo},
who were inspired by~\cite{FKV98:Fast-Monte-Carlo}.
The paper~\cite{HI15:Randomized-Approximation} contains a more detailed investigation
of the special case that we present here.

Consider a rectangular matrix $\mtx A \in \R^{n \times d}$.
Heuristically, we think of the case where the number $d$ of columns is much larger than number $n$ of rows.
Suppose we want to approximate the outer product $\mtx{B} \coloneqq \mtx{A} \mtx{A}^*$.  A direct computation expends $\mathcal{O}(n^2 d)$ arithmetic operations.  By \hilite{approximating} the product, we can try to reduce the cost of the computation.

\begin{problem}[Approximate matrix multiplication]
Let $\mtx{A}$ be a (rectangular) matrix.  The task is to approximate
the outer product $\mtx{B} = \mtx{AA}^*$ while limiting the number of arithmetic
operations.
\end{problem}

\noindent
We will see that the approximate matrix multiplication problem fits
into the framework for matrix approximation by sampling
that we introduced in \Cref{sec:matrix-approx-setup}.

Observe that we can express the product as a sum of rank-one matrices:
\begin{equation} \label{eqn:approx-mtx-mult-B}
\mtx{B} = \mtx{AA}^* = \sum_{j=1}^d \vct{a}_j \vct{a}_j^* =: \sum_{j=1}^d \mtx{B}_j.
\end{equation}
We have written $\vct{a}_j$ for the $j$th column of the matrix $\mtx{A}$.
Explicitly, $\mtx{B}_j = \vct{a}_j \smash{\vct{a}_j^*}$.
Using $s$ samples, the Monte Carlo approximation %
of the product $\mtx{B}$ takes the form
\[
\widehat{\mtx{B}}_s \coloneqq \frac{1}{s} \sum_{i=1}^s c_{k_i} \vct{a}_{k_i} \vct{a}_{k_i}^*
\]
where $k_1, \dots, k_s$ are chosen independently at random
from the column indices $\{1, \dots, d\}$.
The scale factor $c_{k_i}$ depends on the
sampling distribution.

To design the Monte Carlo approximation algorithm, we must construct a sampling
distribution $(p_j : j = 1, \dots, d)$ on the indices.  It is natural to
sample the matrices $\mtx{B}_j$ in proportion to their spectral norms,
as bigger summands contribute more to the sum.  Thus, we select
\begin{equation} \label{eqn:approx-mtx-mult-pj}
p_j \coloneqq \frac{\norm{\mtx{B}_j}}{ \sum_{j=1}^d \norm{\mtx{B}_j}}
	= \frac{ \norm{\vct{a}_j}_2^2 }{ \fnormsq{\mtx{A}} }.
\end{equation}
These probabilities can be computed with one pass over the matrix $\mtx{A}$,
at a cost of $\mathcal{O}(nd)$ arithmetic operations.

With the sampling probabilities in place,  define the random matrix $\mtx{Y}$ as in~\eqref{eqn:matrix-mc-Y}.
We can easily find a uniform upper bound for the norm of $\mtx{Y}$.  Indeed,
\[
L \coloneqq \sup \norm{ \mtx{Y} } = \max\nolimits_j \norm{ p_j^{-1} \mtx{B}_j }
	= \fnormsq{\mtx{A}}.
\]
To compute the per-sample second moment, we need to work a little harder.
Since $\mtx{Y}$ is a symmetric matrix,
\[
\begin{aligned}
v(\mtx{Y}) &\coloneqq \norm{ \Expect[ \mtx{Y}^2 ] }
	= \lnorm{ \sum_{j=1}^d \big(p_j^{-2} \mtx{B}_j^2\big) \cdot p_j }
	= \lnorm{ \sum_{j=1}^d p_j^{-1} \normsq{ \vct{a}_j } \cdot \vct{a}_j \vct{a}_j^* } \\
	&= \fnormsq{\mtx{A}} \cdot \lnorm{ \sum_{j=1}^d \vct{a}_j \vct{a}_j^* }
	= \fnormsq{\mtx{A}} \cdot \norm{ \mtx{B} }. %
\end{aligned}
\]
In the first line, we have used the definition~\eqref{eqn:approx-mtx-mult-B} of the matrices $\mtx{B}_j$.
To reach the second line, we introduced the sampling probabilities~\eqref{eqn:approx-mtx-mult-pj}.
Note that the two statistics satisfy the identities
\[
\frac{v(\mtx{Y})}{\norm{\mtx{B}}^2} = \frac{L(\mtx{Y})}{\norm{\mtx{B}}}
	= \frac{\fnormsq{\mtx{A}}}{\norm{\mtx{B}}}
	= \frac{\trace( \mtx{B} )}{\norm{\mtx{B}}} = \intdim(\mtx{B}).
\]
In other words, the statistics that control the quality of the Monte Carlo
approximation both coincide with the intrinsic dimension of the product.
This is a consequence of the thoughtful choice of sampling probabilities.

Applying the matrix Monte Carlo theorem (\Cref{thm:matrix-mc}), we obtain a bound
for the number $s$ of terms that we need to sample to approximate the matrix product.
Let $\eps \in (0, 1)$ be an error tolerance.  Suppose that
\[
s \geq 4 \eps^{-2} \cdot \intdim(\mtx{B}) \cdot \log(2n)\eedit{}{.}
\]
Then the approximation $\widehat{\mtx{B}}_s$ of the product $\mtx{B}$ using $s$ samples
satisfies the relative error bound~\eqref{eqn:Bhat-rel-error}.
In other words, we obtain an accurate relative approximation when
the number $s$ of samples is roughly proportional to $\intdim(\mtx{B})$,
the intrinsic dimension of the product.
After forming the sampling probabilities with $\mathcal{O}(nd)$ operations,
the cost of approximating the product is $\mathcal{O}(n^2 s)$ operations.
For contrast, forming the full product costs $\mathcal{O}(n^2 d)$ operations.
When $\intdim(\mtx{B}) \ll d$, we obtain a significant reduction in the
computational effort.

It is interesting to investigate the performance of approximate matrix multiplication
with alternative sampling probabilities.  The adaptive sampling probabilities~\eqref{eqn:approx-mtx-mult-pj}
are the most natural, but some applications force us to use uniform sampling probabilities: $p_j = 1/d$
for each $j$.  As an exercise, you should work out sample complexity of approximate matrix multiplication
with uniform sampling.

To avoid complications, we only considered multiplication of a matrix with its own transpose.
Nevertheless, a similar algorithm is valid for approximating the product of two different
matrices.  For more details, see \cite[Secs.~6.4 and 7.3.3]{Tro15:Introduction-Matrix}.

\subsubsection{Example: Matrix sparsification} 

As another simple example of matrix Monte Carlo approximation,
consider the problem of replacing a matrix by a sparse proxy.
This is called ``matrix sparsification''; its nominal
applications include compression or acceleration of downstream
computations.  Randomized sparsification methods were first studied by
Achlioptas \& McSherry~\cite{AM01:Fast-Computation}.

We can apply the same framework as before.  Observe that every matrix
$\mtx{B} \in \R^{m \times n}$ admits the entrywise decomposition
\[
\mtx{B} = \sum_{j=1}^m \sum_{k=1}^n b_{jk} \mathbf{E}_{jk},
\]
where $b_{jk}$ denotes the $(j,k)$ entry of matrix $\mtx{B}$
and the $\mathbf{E}_{jk}$ are the elements of the standard basis
for matrices.  It requires some cleverness to identify an effective
choice for the sampling probabilities.
Kundu \& Drineas~\cite{KD14:Note-Randomized} proposed the distribution
\[
p_{jk} = \frac{1}{2} \left( \frac{\abs{b_{jk}}^2}{\fnormsq{\mtx{B}}} + \frac{\abs{b_{jk}}}{\norm{\mtx{B}}_{\ell_1}} \right)
\quad\text{for $j=1,\dots,m$ and $k = 1, \dots, n$.}
\]
In this expression, $\norm{\cdot}_{\ell_1}$ reports the sum of the absolute values
of the entries of a matrix.  With these probabilities, we can obtain
an approximation with spectral-norm relative error $\eps > 0$ with
sample complexity
\[
s \geq 8 \eps^{-2} \cdot \intdim(\mtx{BB}^*) \cdot \max\{m,n\} \log(m+n)
\]
Thus, the sparsity of the approximation is roughly
the rank of $\mtx{B}$ times the maximum number of entries
in a row or column.
We leave the detailed analysis of this scheme as
an exercise for the reader.

\subsubsection{Matrix Monte Carlo: Improvements}

In practice, the sampling complexity of a matrix Monte Carlo approximation
often depends on information that is not easily accessible.  Fortunately,
we can obtain \lang{a posteriori} bounds on the quality of the approximation using
simulation methods, such as the jackknife~\cite{ET24:Efficient-Error}
or the bootstrap~\cite{Lop19:Estimating-Algorithmic}.

As with scalar Monte Carlo, the sampling complexity of matrix Monte
Carlo approximation scales as $s \approx \eps^{-2}$.
To address this problem, we may need to incorporate variance reduction techniques~\cite[Sec.~6.2.3]{Tro15:Introduction-Matrix}.
A simple idea is to adjust the sampling strategy so that ``significant'' terms always appear in the approximation.
At the cost of introducing a bias into the approximation, we can also avoid sampling ``insignificant'' terms
that generate a lot of variability.  Beyond these first steps, the literature does not
really discuss more sophisticated approaches for variance reduction in the matrix setting.

\subsubsection{Matrix Monte Carlo: Further examples}

Our examples of matrix Monte Carlo approximation are intended primarily for illustration,
rather than as serious numerical methods.  Nevertheless, these techniques can serve as %
ingredients in more sophisticated algorithms.  The matrix Monte Carlo method can also be
used to address other challenges in numerical linear algebra.

\begin{itemize}
\item	\textbf{Graph sparsification \cite{SS11:Graph-Sparsification}.}  We can approximate
the Laplacian of an undirected graph by sampling edges in proportion to their effective resistances.
By this procedure, every graph on $n$ vertices admits a spectrally accurate approximation
with $\mathcal{O}(n \log n)$ edges.

\item	\textbf{Sparse Cholesky factorization \cite{KS16:Approximate-Gaussian}.}  Inside their
algorithm for solving Laplacian linear systems, Kyng \& Sachdeva rely on a method for replacing
a clique in a combinatorial graph with a sparser graph.  This amounts to a matrix Monte Carlo
approximation of the subgraph.  See \Cref{sec:kyng} for more discussion,
or see~\cite{Tro19:Matrix-Concentration-LN} for a full exposition of this algorithm and its analysis.

\item	\textbf{Random features \cite{RR08:Random-Features}.}  We can approximate certain
types of kernel matrices efficiently using a sum of random rank-one matrices.  This method
is called a \textit{random features} approximation.  The benefit is that the random features
for each data point can be computed independently of the other data points.
The analysis of random features as a matrix Monte Carlo method appears
in~\cite{LSS+14:Randomized-Nonlinear,Tro15:Introduction-Matrix,Tro19:Matrix-Concentration-LN}.

\item	\textbf{Quantum state tomography \cite{GKKT20:Fast-State-Tomography}.}  By the laws
of quantum mechanics, measuring (copies of) a quantum state produces a sequence of
random numbers.  We can use these random measurement values and the matrix elements of the measurement
ensemble to construct a matrix Monte Carlo approximation of the quantum state.
\end{itemize}

\noindent
The literature contains many further examples of matrix approximation by sampling.

\section{Random initialization}

As you know, the performance of an iterative algorithm often depends on the initialization.
By starting an iterative algorithm at a randomly chosen point,
we can encourage the algorithm to converge more rapidly to a solution.
Alternatively, we can avoid bad starting points that may result in slow convergence
or outright failure.

\begin{idea}[Random initialization]
    Initialize an (iterative) algorithm at a random point to make a favorable trajectory likely.
\end{idea}

This technique has played a role in numerical linear algebra and optimization for decades.
But researchers did not succeed in \hilite{quantifying}
the benefits of the random initialization until much later.
We now recognize that the random start plays a fundamental role
in the performance of certain numerical methods.
Let us tell part of this story.

This section begins with a classical randomized algorithm
for computing the maximum eigenvalue of a psd matrix.
Afterward, we study a more recent randomized algorithm
for producing a low-rank approximation of a rectangular matrix.
Both of these algorithms depend on a random initialization
to support their performance.

\subsection{The randomized power method}
\label{sec:rpm}

A basic problem in spectral computation is to find the maximum eigenvalue of a psd matrix. %
In many applications, we have access to the matrix via the matvec operation.
For example, consider the case where the matrix describes the discretization of an
integral operator. %
The challenge is to determine the maximum eigenvalue of the matrix.

\begin{problem}[Maximum eigenvalue]
Suppose that $\mtx{A}$ is a psd matrix, accessible via matvecs $\vct{x} \mapsto \mtx{A} \vct{x}$.
The task is to compute the maximum eigenvalue $\lambda_{1}(\mtx{A})$ while limiting the number of matvecs.
\end{problem}

By Abel's theorem, we cannot solve this problem using a finite number of arithmetic operations;
we must resort to approximation.  Consequently, we study iterative methods
that repeatedly refine an estimate for the maximum eigenvalue.

\subsubsection{The power method}

To approximate the maximum eigenvalue of a psd matrix using matvecs,
we can employ the \hilite{power method}.  This venerable algorithm
was developed in the early 20th century
by M{\"u}ntz and by von Mises.
See Golub \& van der Vorst~\cite{GV00:Eigenvalue-Computation}
for some history and citations.

Let $\mtx{A} \in \Sym_n^+(\R)$ be a real symmetric \hilite{psd} matrix
that is \hilite{nonzero}.
Our goal is to approximate the maximum eigenvalue
$\lambda_1 \coloneqq \lambda_1(\mtx{A}) > 0$ using matvecs with $\mtx{A}$.

The power method requires a nonzero starting vector $\vct{\omega} \in \R^n$.
We set the initial iterate $\vct{x}_0 \coloneqq \vct{\omega}$ equal to the starting vector.
At each iteration, the power method updates the iterate $\vct{x}_t$ using
one matvec with the matrix $\mtx{A}$:
\begin{equation} \label{eqn:pm-vectors}
\vct{q}_t \coloneqq \vct{x}_t / \norm{ \vct{x}_t }_2
\quad\text{and}\quad
\vct{x}_{t+1} \coloneqq \mtx{A} \vct{q}_t
\quad\text{for $t = 0, 1, 2, \dots$.}
\end{equation}
Each iteration provides a new estimate $\xi_t$ for the maximum eigenvalue:
\begin{equation} \label{eqn:pm-maxeig}
\xi_t \coloneqq \vct{q}_t^* \vct{x}_{t+1} = \vct{q}_t^* (\mtx{A} \vct{q}_t)
\quad\text{for $t = 0,1,2, \dots$.}
\end{equation}
We continue this process indefinitely, or until a stopping criterion is triggered.

To measure how well the estimate $\xi_t$ approximates
the maximum eigenvalue $\lambda_1$, we use relative error:
\begin{equation} \label{eqn:maxeig-rel-err}
\err(\xi_t) \coloneqq \frac{\lambda_1 - \xi_t}{\lambda_1}. %
\end{equation}
It is easy to check that $\err(\xi_t)\in[0,1]$ since $\xi_t \in [0, \lambda_1]$.
The bounds on $\xi_t$ hold because it is a Rayleigh quotient of the psd matrix $\mtx{A}$.

\subsubsection{Power method: Intuition}

To understand why the power method procedure might work,
we introduce an eigenvalue decomposition of the input matrix:
\[
\mtx{A} = \sum_{i=1}^n \lambda_i \, \vct{v}_i \vct{v}_i^*
\quad\text{where}\quad
\lambda_1 \geq \lambda_2 \geq \dots \geq \lambda_n \geq 0.
\]
The family $(\vct{v}_i : i = 1, \dots, n)$ of eigenvectors
composes an \hilite{orthonormal} basis for $\R^n$.

Unroll the iteration~\eqref{eqn:pm-vectors} to confirm that
\begin{equation} \label{eqn:pm-vector-power}
\vct{q}_t = \frac{\mtx{A}^t \vct{\omega}}{\norm{\mtx{A}^t \vct{\omega}}_2}
\quad\text{for $t = 0,1,2,\dots$.}
\end{equation}
In terms of the spectral decomposition of the matrix $\mtx{A}$,
it is clear that the power amplifies the largest eigenvalue:
\[
\mtx{A}^t = \sum_{i=1}^n \lambda_i^t \, \vct{v}_i \vct{v}_i^*.
\]
Eventually, as $t$ increases, the power damps the eigenvalues strictly smaller
than $\lambda_1$, and $\mtx{A}^t / \norm{\mtx{A}^t}$ converges to an orthogonal
projector onto the eigenspace associated with the eigenvalue $\lambda_1$.
Provided that the starting vector $\vct{\omega}$ has a nonzero component
in this subspace, then the iterate $\vct{q}_t$ will converge toward
a unit vector in this subspace.  As a consequence, the associated
eigenvalue estimate $\xi_t$ will converge toward $\lambda_1$.

\subsubsection{Power method: Convergence}
\label{sec:pm-classical}

The classical analysis of the power method allows us to sharpen this intuition. 
To simplify the presentation, we rescale the nonzero matrix
$\mtx{A}$ so that the maximum eigenvalue \hilite{$\lambda_1 = 1$}.
We will express the starting vector $\vct{\omega}$
in the orthonormal basis of eigenvectors:
\[
\vct{\omega} = \sum_{i=1}^n \omega_i \, \vct{v}_i
\quad\text{where}\quad
\omega_i \coloneqq \vct{v}_i^* \vct{\omega}.
\]
This notation for the components for the starting vector remains
in force throughout the section.  If you prefer, you could instead
change basis so that $\mtx{A}$ is diagonal.

Let us develop an expression for the relative error~\eqref{eqn:maxeig-rel-err}.
Since $\lambda_1 = 1$, a short calculation yields
\begin{equation}
\begin{aligned}
    \err(\xi_t) &= 1 - \vct{q}_t^* (\mtx{A} \vct{q}_t)
    = \frac{(\vct{\omega}^* \mtx{A}^t)(\mtx{A}^{t} \vct{\omega})}{\normsq{\mtx{A}^t \vct{\omega}}_2} %
    	- \frac{(\vct{\omega}^* \mtx{A}^{t})(\mtx{A}^{t+1} \vct{\omega})}{\normsq{\mtx{A}^{t} \vct{\omega}}_2} \\
    &= \frac{\vct{\omega}^* \mtx{A}^{2t}(\Id_n - \mtx{A})\vct{\omega}}{\normsq{\mtx{A}^{t} \vct{\omega}}_2}
    = \frac{\sum_{i=1}^n \omega_i^2 \lambda_i^{2t} (1 - \lambda_i)}{\sum_{i=1}^n \omega_i^2 \lambda_i^{2t}}.
\end{aligned}
\label{eq:PM_relative_error}
\end{equation}
The first line relies on the definition~\eqref{eqn:pm-maxeig} of the eigenvalue estimate $\xi_t$,
along with the representation~\eqref{eqn:pm-vector-power} of the iterate $\vct{q}_t$.
In the second line, we calculate in the eigenvector basis.

The expression~\eqref{eq:PM_relative_error} for the relative error exposes the influence
of the eigenvalues and the components of the starting vector in the eigenvector basis.
From here, we can easily confirm the asymptotic convergence of the power method under minimal conditions.

\begin{proposition}[Power method: Convergence] \label{prop:pm-convergence}
Assume that the psd matrix $\mtx{A}$ is nonzero.
If the starting vector $\vct{\omega}$ satisfies $\omega_1 \neq 0$,
then the power method produces maximum eigenvalue estimates $\xi_t$
whose relative errors~\eqref{eqn:maxeig-rel-err}
satisfy $\err(\xi_t) \to 0$ as $t \to \infty$.
In particular, $\xi_t \to \lambda_1$.
\end{proposition}

\begin{proof}
Impose the normalization $\lambda_1 = 1$.
For the scalar matrix $\mtx{A} = \Id$, we can easily check that $\err(\xi_t) = 0$ for all $t = 0,1,2, \dots$.

Otherwise, let $m < \dim(\mtx{A})$ be the multiplicity of the largest eigenvalue:
$1 = \lambda_1 = \dots = \lambda_m > \lambda_{m+1}$.
Then the identity~\eqref{eq:PM_relative_error} for the error ensures that
\[
\err(\xi_t) = \frac{\sum_{i} \omega_i^2 \lambda_i^{2t} (1 - \lambda_i)}{\sum_{i} \omega_i^2 \lambda_i^{2t}}
	= \frac{\sum_{i > m} \omega_i^2 \lambda_i^{2t}(1- \lambda_i)}{(\omega_1^2 + \dots + \omega_m^2) + \sum_{i > m} \omega_i^2 \lambda_i^{2t}}.
\]
\hilite{Since $\omega_1^2 > 0$}, the parenthesis in the denominator is strictly positive.  
For each $i > m$, we have the limit $\lambda_i^{2t} \to 0$ as $t \to \infty$.
We conclude that $\err(\xi_t) \to 0$.
\end{proof}

The proof of \Cref{prop:pm-convergence} emphasizes that
the estimates for the maximum eigenvalue converge to the correct value
($\xi_t \to \lambda_1$) \hilite{if and only if} the
starting vector $\vct{\omega}$ has a nonzero component
in the eigenspace associated with the maximum eigenvalue.

\subsubsection{Power method: Convergence rate}

As an aside, we can also exploit~\eqref{eq:PM_relative_error} to
derive asymptotic convergence rates for the power method.
These results enforce a \hilite{spectral gap}: $\lambda_1 > \lambda_2$.

\begin{proposition}[Power method: Asymptotic convergence rate] \label{prop:power-method-asymptotic}
	Assume that the psd matrix $\mtx{A}$ satisfies
	$\lambda_1 > \lambda_2 > \lambda_3$.  Assume that the first two components
	of the starting vector $\vct{\omega}$ are nonzero: $\omega_1, \omega_2 \ne 0$.
	Then the power method produces eigenvalue estimates $\xi_t$
	whose relative errors~\eqref{eqn:maxeig-rel-err} satisfy 
    \[
    \frac{\err(\xi_{t+1})}{\err(\xi_t)} \to \lt (\frac{\lambda_2}{\lambda_1}\rt)^{2}
    \quad\text{as $t \to \infty$.}
    \]
\end{proposition}

This result states that $\xi_t \to \lambda_1$ exponentially fast
with a rate that depends on the spectral gap.  Let us stress that this is
an \hilite{asymptotic} rate of convergence that only prevails in the limit.
As with \Cref{prop:pm-convergence}, it is both necessary and sufficient that
$\omega_1, \omega_2 \neq 0$ to achieve this rate.

\begin{proof}
As before, we may rescale so that $\lambda_1 = 1$.  Using the identity~\eqref{eq:PM_relative_error},
we can form the ratio of the relative errors:
\[
\begin{aligned}
\frac{\err(\xi_{t+1})}{\err(\xi_t)}
	&= \frac{\sum_{i} \omega_i^2 \lambda_i^{2t+2} (1 - \lambda_i)}{\sum_{i} \omega_i^2 \lambda_i^{2t} (1 - \lambda_i)}
	\cdot \frac{\sum_{i} \omega_i^2 \lambda_i^{2t}}{\sum_{i} \omega_i^2 \lambda_i^{2t+2}} \\
		&= \frac{\lambda_2^{2} \cdot \omega_2^2 \lambda_2^{2t} (1 - \lambda_2) + \sum_{i>2} \omega_i^2 \lambda_i^{2t+2} (1 - \lambda_i)}
		{\omega_2^2 \lambda_2^{2t} (1 - \lambda_2) + \sum_{i>2} \omega_i^2 \lambda_i^{2t} (1 - \lambda_i)}
	\cdot \frac{\omega_1^2 + \sum_{i>1} \omega_i^2 \lambda_i^{2t}}{\omega_1^2 + \sum_{i>1} \omega_i^2 \lambda_i^{2t+2}}
	\to \lambda_2^2 \cdot 1.
\end{aligned}
\]
In the second line, we have extracted the dominant term from each sum.  As $t \to \infty$,
each sum converges to zero faster than this dominant term because $\lambda_1 > \lambda_2 > \lambda_3$.
We can cancel the effects of $\omega_1, \omega_2$ because they are both nonzero.
\end{proof}

\subsubsection{The role of randomness}

\Cref{prop:pm-convergence,prop:power-method-asymptotic} demand that
the starting vector $\vct{\omega}$ for the power method
has nonzero components in the direction of the leading eigenspaces of the matrix $\mtx{A}$.
How can we secure this property if we do not know anything about the eigenvectors?
The classical prescription is to select the starting vector \hilite{at random}.

Using the structure of the maximum eigenvalue problem,
we can design a {distribution} for the
random starting vector in the power method.
Since we do not have access to the eigenvector basis of the input matrix,
we should select a distribution for the starting vector that is indifferent
to the eigenvector basis.
The natural choice is a \hilite{rotationally invariant} distribution.
The power method is invariant to the scale of the starting vector,
so the distribution of the radial component is at our option.

Recognizing these principles, we will draw the starting vector from the %
\hilite{standard normal} distribution: $\vct{\omega} \sim \normal(\vct{0}, \Id_n)$.
In particular, the components $(\omega_1, \dots, \omega_n)$ of the starting vector
in the eigenvector basis %
compose an \hilite{independent} family of real standard normal variables:
$\omega_i \sim \normal(0, 1)$.
\Cref{alg:power-method} lists pseudocode for this %
version of the \hilite{power method with a random initialization}.

This random starting vector has a valuable feature that
has long been appreciated by numerical analysts~\cite[p.~155]{Dem97:Applied-Numerical}.
Almost surely, each component of the standard normal starting vector
is nonzero: $\Prob{ \omega_i \neq 0 } = 1$.
Thus, we can activate \Cref{prop:pm-convergence,prop:power-method-asymptotic}
to deduce that the randomized power method converges to the maximum eigenvalue
\hilite{with probability one}!

In fact, the random starting vector provides additional benefits
that were first recognized~\cite{Dix83:Estimating-Extremal,KW92:Estimating-Largest} in the 1980s and 1990s.
With a careful probabilistic analysis, %
we can derive several beautiful results about the randomized power method.
These theorems appear in the upcoming subsections.

For intuition, a closer examination of the relative error~\eqref{eq:PM_relative_error}
suggests that the power method converges most quickly when the first component $\omega_1$ of the
starting vector is large, while the remaining components $\omega_2, \dots, \omega_n$
are small.  %
Since the components are independent random variables with $\Expect[ \omega_i^2 ] = 1$,
the typical size of each coordinate $\omega_i$ is about one.
In other words, the standard normal distribution does a reasonable job at achieving
our goals for the components of the starting vector.
This is the best we can hope for, absent knowledge about the leading eigenspace.

\begin{algorithm}[t]
\begin{algbox}[1]
\caption{\textit{Randomized power method.}}
\label{alg:power-method}

\Require	Psd input matrix $\mtx{A} \in \Sym_n^+(\R)$, number $T$ of iterations
\Ensure		Maximum eigenvalue estimate $\xi_T$

\vspace{5pt}
\hrule
\vspace{5pt}

\Function{RandPowerMethod}{$\mtx{A}$; $T$}
\State	Sample $\vct{x}_0 \sim \normal(\vct{0}, \Id_n)$
\Comment	Standard normal starting vector

\For{$t = 0, \dots, T$}
\State	$\vct{q}_t = \vct{x}_t / \norm{\vct{x}_t}_2$
\Comment	Normalize the iterate
\State	$\vct{x}_{t+1} = \mtx{A} \vct{q}_t$
\Comment	One matvec

\State	$\xi_t = \vct{q}_t^* \vct{x}_{t+1}$
\Comment	Eigenvalue estimate
\EndFor

\EndFunction
\end{algbox}
\end{algorithm}

\subsubsection{Randomized power method: Spectral gap}

The next step is to quantify how quickly the randomized power method
converges to the maximum eigenvalue.
This theory adapted from a remarkable paper of
Kuczy{\'n}ski \& Wo{\'z}niakowski~\cite{KW92:Estimating-Largest}.
As observed in~\cite[Lec.~2]{Tro20:Randomized-Algorithms-LN},
we can dramatically simplify their calculations using the
properties of the standard normal distribution.

Echoing the classical analysis, we begin with a result that imposes an assumption
on the spectral gap of the input matrix.
The next section will show how to remove this condition.

\begin{theorem}[Randomized power method: Spectral gap] \label{thm:RPM_with_gap}
Assume that $\mtx{A} \in \Sym_n^+(\R)$ is a psd matrix with a spectral gap:
$\lambda_1 > \lambda_2$. Then the randomized power method (\Cref{alg:power-method}) produces
eigenvalue estimates $\xi_t$ whose relative errors~\eqref{eqn:maxeig-rel-err}
satisfy
\begin{equation} \label{eqn:rpm-gap}
\Expect \err (\xi_t) \le \sqrt{2n} \cdot \lt ( \frac{\lambda_2}{\lambda_1} \rt)^t 
\quad\text{for $t = 0,1,2,\dots$.}
\end{equation}
See~\cite[Thm.~4.1]{KW92:Estimating-Largest} for probability bounds.
\end{theorem}

In contrast with the classical convergence result~\Cref{prop:power-method-asymptotic},
\Cref{thm:RPM_with_gap} offers a \hilite{nonasymptotic} error bound
that is valid in each iteration $t$.
It may be disappointing that the expected error~\eqref{eqn:rpm-gap} only improves by
a factor of $(\lambda_2 / \lambda_1)$ at each iteration,
a rate that is substantially worse than the asymptotic rate $(\lambda_2 / \lambda_1)^2$.
The bound~\eqref{eqn:rpm-gap} reflects the possibility
that the first component $\omega_1$ of the starting vector is unusually small.
For the randomized power method, \Cref{thm:RPM_with_gap}
is essentially sharp~\cite[Thm.~3.1(b)]{KW92:Estimating-Largest},
but see~\cite{Tro22:Randomized-Block} for more texture.

What else does the bound~\eqref{eqn:rpm-gap} tell us about
the performance of the randomized power method?
Let us rewrite the error in terms
of the \hilite{relative spectral gap} $\gamma$, defined as %
\[
\gamma \coloneqq \frac{\lambda_1 - \lambda_2}{\lambda_1} \in [0,1].
\]
\Cref{thm:RPM_with_gap} yields the exponential convergence rate
\[
\Expect \err(\xi_t) \le \sqrt{2n} \cdot \econst^{-\gamma t}
\quad\text{for $t = 0,1,2,\dots$.}
\]
For any tolerance $\eps \in (0,1)$, we can achieve an expected relative
error $\eps$ after a predictable number $t$ of iterations:
\[
t \ge \frac{\log \sqrt{ 2n } + \log \lt( 1/\eps \rt)}{\gamma}
\quad\text{implies}\quad
\Expect \err(\xi_t) \leq \eps.
\]
This formula reveals that the randomized power method exhibits a \hilite{burn-in period}.
It may not make any progress for the first $\log(2n) / (2 \gamma)$ iterations. 
Afterward, it begins to converge at an exponential rate
that depends on the relative spectral gap $\gamma$.
The logarithmic dependence $\log(1/\eps)$
on the tolerance $\eps$ indicates that we can achieve very small relative errors after a
modest number of iterations, relative to the spectral gap $\gamma$.
For comparison, recall that a simple Monte Carlo approximation
suffers an $\eps^{-2}$ dependence on the tolerance.

\begin{proof}[Proof of \Cref{thm:RPM_with_gap}]
	As always, we may normalize the matrix so $\lambda_1 = 1$.
    Start with the expectation of the bound~\eqref{eq:PM_relative_error} on the relative error:
    \begin{equation} \label{eqn:rpm-error-decomp}
    \Expect \err(\xi_t) = \Expect \left[ \frac{\sum_{i} \omega_i^2 \lambda_i^{2t} (1 - \lambda_i)}{ \sum_i \omega_i^2 \lambda^{2t} } \right] %
    	\leq \Expect_{\omega_2, \dots, \omega_n} \Expect_{\omega_1} \left[ \frac{\sum_{i > 1} \omega_i^2 \lambda_i^{2t}}{ \omega_1^2 + \sum_{i > 1} \omega_i^2 \lambda^{2t} } \right].
	\end{equation}
	To bound the numerator, note that $0 \leq 1 - \lambda_i \leq 1$.
	Since the components $\omega_1, \dots, \omega_n$ are \hilite{independent} standard normal variables,
	we can take the expectation $\Expect_{\omega_1}$ with respect to the first component
	before averaging over the remaining components.

To continue, we rely on an expectation bound for a standard normal variable~\cite[Fact 2.8]{Tro20:Randomized-Algorithms-LN}.
\begin{fact}[A Gaussian integral] \label{fact:gaussian_integral}
	Fix $c \geq 0$. For a real standard normal variable $g \sim {\textup \normal}(0, 1)$,
    \[
    \Expect_g \lt [ \frac{c }{g^2 + c} \rt ] 
    	\leq \sqrt{\frac{\pi}{2} \cdot c}. %
    \]
\end{fact}

\noindent
	Employing Fact~\ref{fact:gaussian_integral} with $c = \sum_{i > 1} \omega_i^2 \lambda_i^{2t}$
	and $g = \omega_1$,	we can bound the error as
	\begin{align*}
	\Expect \err(\xi_t)
		\leq \Expect_{\omega_2, \dots, \omega_n} \sqrt{\frac{\pi}{2}\sum_{i>1} \omega_i^2 \lambda_i^{2t}} %
		\leq \sqrt{\frac{\pi}{2} \sum_{i > 1} \Expect[ \omega_i^2 ] \cdot \lambda_i^{2t}} 
		\leq \sqrt{\frac{\pi}{2} (n-1)} \cdot \lambda_2^{t}
	\end{align*}
	The second inequality is Jensen's.  Third, we recall that $\Expect[\omega_i^2] = 1$
	and introduce the bound $\lambda_i \leq \lambda_2$ for $i \geq 2$.
	Finally, invoke the numerical inequality $\pi/2 < 2$ and remove the normalization.
\end{proof}

\subsubsection{Randomized power method: No spectral gap}

The probabilistic analysis of the randomized power method also points
to new phenomena that are invisible to the classical analysis.
\hilite{Without assuming a spectral gap}, we can still obtain
a convergence rate for the randomized power method.  This result is also
due to Kuczy{\'n}ski \& Wo{\'z}niakowski~\cite{KW92:Estimating-Largest},
with roots in the work of Dixon~\cite{Dix83:Estimating-Extremal}.
The proof here is adapted from~\cite[Lec.~2]{Tro20:Randomized-Algorithms-LN}.

\begin{theorem}[Randomized power method: No spectral gap]
\label{thm:rpm-gapless}
Assume that $\mtx{A} \in \Sym_n^+(\R)$ is a nonzero psd matrix.
Then the randomized power method (\Cref{alg:power-method}) produces
eigenvalue estimates $\xi_t$ whose relative errors~\eqref{eqn:maxeig-rel-err}
satisfy
\[
\Expect \err(\xi_t) \le \frac{1 + \log \sqrt{2n} + \log t}{t}
\quad\text{for $t = 1,2,3,\dots$.}
\]
In particular, for any tolerance $\eps \in (0,1)$,
\[
t \ge \frac{1 + \log \sqrt{2n} + \log t }{\eps}
\quad\text{implies}\quad
\Expect \err(\xi_t) \le \eps.
\]
See~\cite[Thm.~4.1]{KW92:Estimating-Largest} for probability bounds.
\end{theorem}

Before turning to the proof, a few remarks.
\Cref{thm:rpm-gapless} is a \hilite{nonasymptotic} result that holds in each iteration.
It reveals a \hilite{burn-in period} of about $\log \smash{\sqrt{2n}}$ steps
before the randomized power method starts to make progress.
Afterward, the error decays at the \hilite{polynomial rate} of $1/t$.
This is far worse than the exponential rate attained
when we have a spectral gap, but it requires \hilite{no assumptions}.
Keep in mind that \Cref{thm:RPM_with_gap,thm:rpm-gapless} are both valid,
so we can employ whichever bound is strongest.

\begin{proof}[Proof of Theorem~\ref{thm:rpm-gapless}]
Normalize the matrix so that $\lambda_1 = 1$.
Nominate a parameter $\beta \in [0,1]$, which will be determined later.
As compared with the proof of \Cref{thm:RPM_with_gap},
the key new insight is to split the sum in the numerator
into terms where $\lambda_i \geq \beta$ and where $\lambda_i < \beta$.

Proceeding in this way,
the expected relative error~\eqref{eq:PM_relative_error} satisfies the bound
\begin{align*}
\Expect \err(\xi_t) &= \Expect \left[ \frac{\sum_{i} \omega_i^2 \lambda_i^{2t} (1 - \lambda_i)}{\sum_{i} \omega_i^2 \lambda_i^{2t}}\right]
	\leq \Expect \left[ \frac{(1-\beta) \sum_{\lambda_i \geq \beta} \omega_i^2 \lambda_i^{2t} + \sum_{\lambda_i < \beta} \omega_i^2 \lambda_i^{2t}}{\sum_{i} \omega_i^2 \lambda_i^{2t}} \right] \\
	&\leq (1 - \beta) + \Expect_{\omega_2,\dots,\omega_n} \Expect_{\omega_1} \left[\frac{\sum_{\lambda_i < \beta} \omega_i^2 \lambda_i^{2t}}{\omega_1^2 + \sum_{\lambda_i < \beta} \omega_i^2 \lambda_i^{2t}} \right] \\
\intertext{Invoke \Cref{fact:gaussian_integral} with $c = \sum_{\lambda_i < \beta} \omega_i^2 \lambda_i^{2t}$
and $g=\omega_1$ to obtain}
\Expect \err(\xi_t) &\leq (1 - \beta) + \Expect_{\omega_2, \dots, \omega_n} \sqrt{ \frac{\pi}{2} \sum_{\lambda_i < \beta} \omega_i^2 \lambda_i^{2t} }
	\leq (1 - \beta) + \sqrt{\frac{\pi}{2} \sum_{\lambda_i < \beta} \lambda_i^{2t}} \\
	&\leq (1 - \beta) + \sqrt{\frac{\pi}{2} (n-1)} \cdot \beta^{t}
	\leq (1 - \beta) + \sqrt{2n} \cdot \econst^{-(1-\beta)t}.
\end{align*}
The second inequality is Jensen's.  Then we made the coarse bound $\lambda_i \leq \beta$
for each eigenvalue that participates in the sum.  The last inequality is numerical.
Finally, minimize the right-hand side over $\beta \in [0,1]$ to arrive at
advertised result.
\end{proof}

\subsubsection{Discussion}

By exploiting the properties of the standard normal starting vector,
we can systematically establish a body of results on the behavior of the
randomized power method.
We can reproduce classical facts~\cite[Thm.~4.2.1]{Par98:Symmetric-Eigenvalue}
about the performance in the presence of a spectral gap,
including the burn-in period and nonasymptotic exponential
convergence rates.
We can also identify the surprising phenomenon that the
randomized power method makes progress
\hilite{even without a spectral gap}.
To the best of our knowledge, this result first emerged
from the probabilistic analysis of the algorithm.

This section focuses on the randomized power method because
it is both familiar and simple.  In practice, however,
we typically rely on more powerful algorithms
for estimating maximum eigenvalues.
Kuczy{\'n}ski \& Wo{\'z}niakowski~\cite{KW92:Estimating-Largest}
also proposed a probabilistic analysis of the Lanczos method
with a random starting vector.  Their work demonstrates that
the randomized Lanczos method has a burn-in period,
nonasymptotic exponential convergence in the presence of a spectral gap,
and nonasymptotic polynomial convergence even without a spectral gap.
Related results~\cite{Tro22:Randomized-Block} are also available for the randomized block
power method and the randomized block Lanczos method;
these variants use multiple starting vectors to enhance convergence.
In all of these cases, the probabilistic analysis reveals new insights
about the performance of these fundamental numerical algorithms.

\subsection{The randomized SVD}

We have seen that random initialization can play a role in the computation
of scalar quantities, such as the maximum eigenvalue of a psd matrix.
In this section, we develop a second example where random initialization
supports an algorithm for a matrix computation: the randomized SVD
for low-rank matrix approximation.

Consider a rectangular matrix %
that we can access with via matvec operations.
A basic task in numerical linear algebra is to produce a low-rank
approximation of the matrix. %
For example, principal component analysis reduces to low-rank approximation.
We frame the following problem:

\begin{problem}[Low-rank matrix approximation]
Suppose $\mtx{B}$ is a rectangular matrix, accessible via matvecs
$\vct{x} \mapsto \mtx{B}\vct{x}$ and $\vct{y} \mapsto \mtx{B}^* \vct{y}$.
The task is to produce a low-rank approximation of $\mtx{B}$ %
that is competitive with a best approximation of similar rank.
\end{problem}

\subsubsection{Singular-value decompositions and low-rank approximations}

We begin with some deep background on low-rank approximation of matrices.
Consider a rectangular matrix $\mtx{B} \in \R^{m \times n}$, and recall
that $m \wedge n \coloneqq \min\{m,n\}$.
We introduce a singular-value decomposition (SVD) of the matrix: %
\begin{equation} \label{eqn:B-svd}
\mtx{B} = %
\sum_{i=1}^{m \wedge n} \sigma_i \, \vct{u}_i \vct{v}_i^*
\quad\text{where}\quad
\sigma_1 \geq \sigma_2 \geq \dots \geq \sigma_{m \wedge n} \geq 0.
\end{equation}
The right singular vectors $(\vct{v}_i : i = 1, \dots, m \wedge n)$
compose an orthonormal system in $\R^n$, %
while the left singular vectors $(\vct{u}_i : i = 1, \dots, m \wedge n)$
compose an orthonormal system in $\R^m$.  As usual, the singular values $\sigma_i$
are listed in decreasing order.  Geometrically, the SVD tells us that the matrix
$\mtx{B}$ maps the unit sphere in $\R^n$ to an ellipsoid in $\R^m$;
see \Cref{fig:Ng1}.

\begin{figure}[th]
\centering
\begin{subfigure}[b]{0.9\textwidth}
   \includegraphics[width=\textwidth]{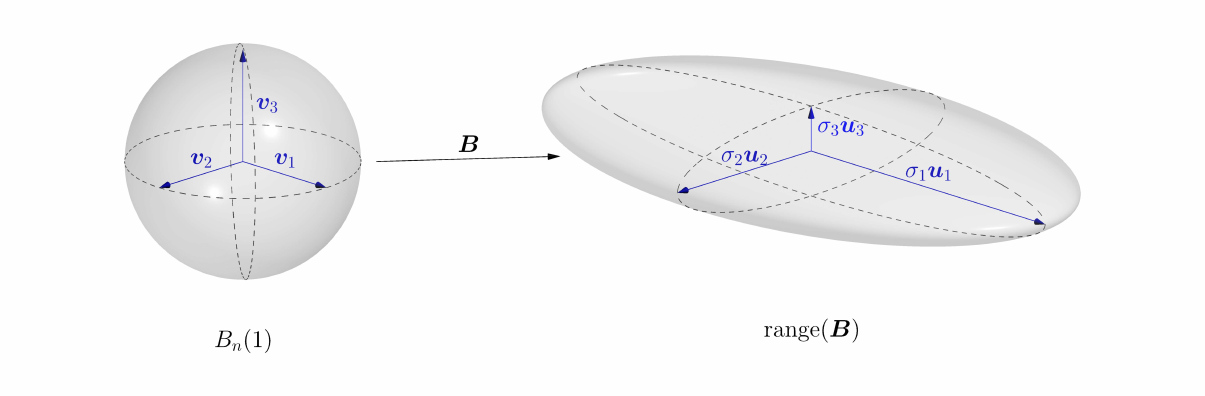}
   \caption{\textbf{Geometry of the SVD.}  The linear map $\mtx B \in \R^{n\times m}$ transforms the unit sphere in $\R^n$ to an ellipsoid in $\R^m$ with semiaxes determined by the singular values and the singular vectors of $\mtx{B}$. %
   }
   \label{fig:Ng1} 
\end{subfigure}
\begin{subfigure}[b]{0.9\textwidth}
   \includegraphics[width=\textwidth]{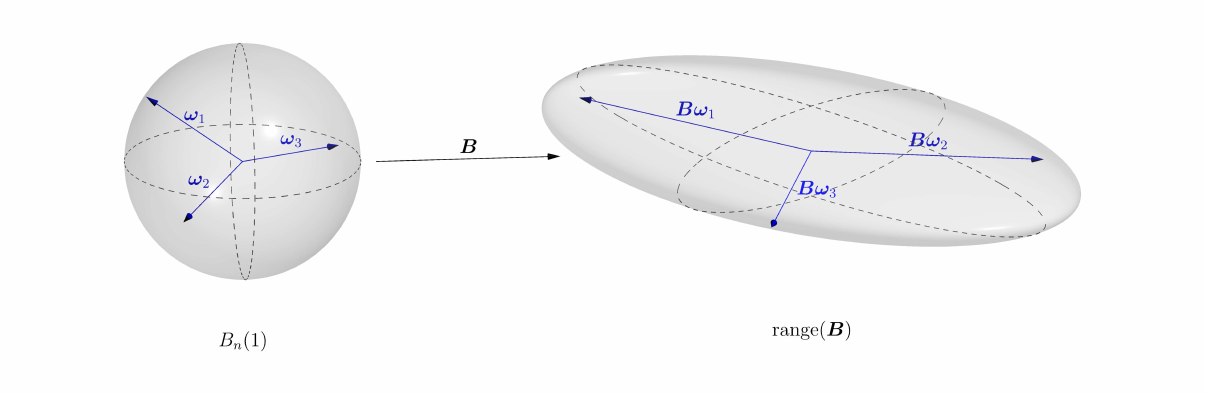}
   \caption{\textbf{Images of random vectors.}  A normalized Gaussian vector $\vct{\omega} / \norm{\vct{\omega}}_2$ is distributed uniformly on the unit sphere.  The image $\mtx{B} \vct{\omega} / \norm{\vct{\omega}}_2$ tends to align with the major semiaxes of the ellipsoid defined by $\mtx B$.
   To cover the leading $r$ directions, we need to use slightly more than $r$ random test vectors.
   }
   \label{fig:Ng2}
\end{subfigure}
\caption{\textbf{Randomized SVD: Intuition.}  These schematics illustrate the geometric ideas behind the SVD and the randomized SVD.}
\end{figure}

What can we say about the optimal low-rank approximation of the matrix $\mtx{B}$?
Given a rank parameter $r \leq m \wedge n$, define the orthogonal projector $\mtx{P}_r$
onto the span of the leading $r$ left singular vectors:
\[
\mtx{P}_r \coloneqq \sum_{i=1}^r \vct{u}_i \vct{u}_i^*.
\]
The Eckart--Young theorem~\cite{EY36:Approximation-One}
identifies a best rank-$r$ approximation with respect
to the Frobenius norm:
\begin{equation} \label{eqn:eckart-young}
\min\nolimits_{\rank(\mtx{M}) \leq r} \fnormsq{ \mtx{B} - \mtx{M} }
	= \fnormsq{ \mtx{B} - \mtx{P}_r \mtx{B} }
	= \sum_{i > r} \sigma_i^2.
\end{equation}
In words, a best approximation is achieved by projecting the
target matrix onto the $r$-dimensional subspace that captures
most of the action of the matrix.  We remark that the best
approximation is unique if and only if $\sigma_{r} > \sigma_{r+1}$.
Throughout this discussion, we focus on the Frobenius-norm error
to obtain more transparent results.

In some applications, it is unnecessary to compute the best approximation
to high accuracy.
Instead, we can relax our requirements.
Let $s = r + \ell$ for a small natural number $\ell$.
For a tolerance $\eps > 0$, we seek a \hilite{rank-$s$} approximation $\widehat{\mtx{B}}_s$
that competes with the best \hilite{rank-$r$} approximation~\eqref{eqn:eckart-young}:
\begin{equation} \label{eqn:low-rank-error-goal}
\fnormsq{ \mtx{B} - \widehat{\mtx{B}}_s }
	\leq (1 + \eps) \cdot \fnormsq{\mtx{B} - \mtx{P}_r \mtx{B}}
	= (1 + \eps) \cdot \sum_{ i > r } \sigma_i^2.
\end{equation}
When the best rank-$r$ approximation error is tiny,
it is not a big deal to pay a modest factor $(1+\eps)$ more.
This situation occurs in scientific computing applications
where the singular values of the matrix decay
exponentially fast~\cite[Sec.~7.1]{HMT11:Finding-Structure}.

\subsubsection{Randomized SVD: Intuition}

Our goal is to produce an estimate for the best rank-$r$ approximation $\mtx{P}_r \mtx{B}$
of the input matrix.
Instead of computing the projector $\mtx{P}_r$ onto the $r$ leading left singular vectors,
we will use \hilite{randomness} to estimate these directions.
We can trace this insight to work of Frieze et al.~\cite{FKV98:Fast-Monte-Carlo}.
Martinsson et al.~\cite{MRT06:Randomized-Algorithm} developed this idea into a practical algorithm,
which was crystallized in the paper~\cite{HMT11:Finding-Structure}.

The method is to multiply \hilite{random test vectors} into the matrix $\mtx{B}$
to identify salient directions in the range.  We have the intuition that the image
of a random vector tends to be aligned with the leading left singular
vectors.  By repetition, we can cover most of %
the range of the projector $\mtx{P}_r$.  Since we do not know the orthonormal basis
of right singular vectors, it is natural to draw the random vectors
from a rotationally invariant distribution.
See \Cref{fig:Ng2} for an illustration.

We can describe this approach mathematically~\cite[Sec.~2.1]{TW23:Randomized-Algorithms}.
Draw a standard normal test vector $\vct{\omega} \in \R^n$.  The image of the random vector satisfies
\[
\mtx{B} \vct{\omega} = \sum_{i=1}^{m \wedge n} \sigma_i \, \vct{u}_i (\vct{v}_i^* \vct{\omega}) 
	\eqqcolon \sum_{i=1}^{m \wedge n}  \omega_i \sigma_i \, \vct{u}_i
\]
As before, the component $\omega_i \coloneqq \vct{v}_i^* \vct{\omega}$ of the random vector
along the $i$th right singular vector follows a standard normal distribution,
and the components $(\omega_i : i = 1, \dots, m \wedge n)$ compose an independent family.
On average, $\Expect[ \omega_i^2 ] = 1$.  Therefore, the image $\mtx{B} \vct{\omega}$
tends to align with the left singular vectors associated with large singular values.

By repeating this process with a statistically independent family
$(\vct{\omega}^{(j)} : j = 1, \dots, s)$ of random test vectors,
we can obtain a family $(\mtx{B} \vct{\omega}^{(j)} : j = 1, \dots, s)$
of vectors whose span contains most of $\range(\mtx{P}_r)$.
The number $s = r + \ell$ of test vectors needs to be a bit
larger than the target rank $r$ to obtain coverage of the
subspace with high probability.

\subsubsection{Randomized SVD: Algorithm}

Let us forge this intuition into an algorithm~\cite[p.~227]{HMT11:Finding-Structure},
called the \hilite{randomized SVD}.
For a rank parameter $s$, we draw a random test matrix:
\[
\mtx{\Omega} = \begin{bmatrix} \vct{\omega}^{(1)} & \dots & \vct{\omega}^{(s)} \end{bmatrix} \in \R^{n \times s}
\quad\text{where $\vct{\omega}^{(j)} \sim \normal(\vct{0}, \Id_n)$ i.i.d.}
\]
We obtain the images of the test vectors using a matrix--matrix product
with the input matrix:
\[
\mtx{Y}  \coloneqq \mtx{B\Omega}.
\]
This step requires $s$ matvecs with $\mtx{B}$.
The orthogonal projector $\mtx{P}_{\mtx{Y}}$ onto the range of $\mtx{Y}$
serves as a proxy for the ideal projector $\mtx{P}_r$.  Computationally,
\[
\mtx{P}_{\mtx{Y}} \coloneqq \mtx{QQ}^*
\quad\text{where}\quad
\mtx{Q}  \coloneqq \texttt{orth}(\mtx{Y}).
\]
The function $\texttt{orth}$ returns an orthonormal basis for the range
of a matrix, and it costs $\mathcal{O}(s^2 m)$ arithmetic operations.
Finally, we report the approximation $\widehat{\mtx{B}}_s$ in factored form:
\[
\widehat{\mtx{B}}_s \coloneqq \mtx{P}_{\mtx{Y}} \mtx{B} = \mtx{Q} (\mtx{Q}^* \mtx{B})
\]
This step requires $s$ matvecs with the transpose $\mtx{B}^*$.
If desired, we can report the SVD of the approximation after
a small amount of additional work:
\[
\widehat{\mtx{B}}_s = (\mtx{Q} \widehat{\mtx{U}}_0) \widehat{\mtx{\Sigma}} \widehat{\mtx{V}}^*
\quad\text{where}\quad
(\widehat{\mtx{U}}_0, \widehat{\mtx{\Sigma}}, \widehat{\mtx{V}}) = \texttt{svd}(\mtx{Q}^* \mtx{B}). 
\]
This step costs another $\mathcal{O}(s^2 n)$ arithmetic operations.

See \cref{alg:rSVD} for pseudocode for the randomized SVD algorithm.
For a dense unstructured matrix $\mtx{B}$, the total cost is
$\mathcal{O}(smn)$. To compare, recall that a dense, economy-size SVD
costs $\mathcal{O}(mn (m \wedge n))$. Classic Krylov subspace methods
are competitive in cost with the randomized SVD, but they often fail
for challenging problem instances~\cite{LLS+17:Algorithm-971}.
In contrast, the randomized SVD and its relatives are bulletproof.

\begin{algorithm}[t]
\begin{algbox}[1]
\caption{\textit{Randomized SVD.}}
\label{alg:rSVD}

\Require Input matrix $\mtx{B} \in \R^{m\times n}$, number $s$ of samples
\Ensure	Factors $\mtx{Q} \in \R^{m \times s}$ and $\mtx{C} \in \R^{s \times n}$
of the approximation $\widehat{\mtx{B}}_s = \mtx{QC}$

\vspace{5pt}
\hrule
\vspace{5pt}

\Function{RandSVD}{$\mtx{B}$}
\State	$\mtx{\Omega} = \texttt{randn}(n,s) \in \R^{n \times s}$ %
	\Comment{Standard normal test matrix}
\State	$\mtx Y = \mtx{B} \mtx{\Omega}$
	\Comment{$s$ matvecs with $\mtx{B}$}
\State $\mtx{Q} = \texttt{orth}(\mtx Y)$
	\Comment{Orthonormal basis for range of $\mtx{Y}$}
\State	$\mtx{C} = \mtx{Q}^* \mtx{B}$
	\Comment{$s$ matvecs with $\mtx{B}^*$}
\State	$(\widehat{\mtx{U}}_0, \widehat{\mtx{\Sigma}}, \widehat{\mtx{V}}) = \texttt{svd}(\mtx{C})$
	\Comment{\textbf{\textcolor{dkgray}{Optional:}} Report SVD of approximation}
\State	$\widehat{\mtx{U}} = \mtx{Q}\widehat{\mtx{U}}_0$
	\Comment{$\widehat{\mtx{B}}_s = \widehat{\mtx{U}}\widehat{\mtx{\Sigma}} \widehat{\mtx{V}}^*$}
\EndFunction
\end{algbox}
\end{algorithm}

\subsubsection{Randomized SVD: Analysis}

Concerning the randomized SVD, the main question is how many test vectors $s$ are needed to obtain a rank-$r$ approximation
that satisfies~\eqref{eqn:low-rank-error-goal}.
We present a result from Halko et al.~\cite[Thm.~10.5]{HMT11:Finding-Structure}
that accurately predicts the performance of the randomized SVD algorithm.

\begin{theorem}[Randomized SVD; Halko et al.~2011] \label{thm:rsvd}
	Consider a matrix $\mtx{B} \in \R^{m \times n}$,
	and fix the target rank $r \leq m \wedge n$.
	When $s \geq r + 2$,
	the randomized SVD method (\Cref{alg:rSVD}) produces a random rank-$s$
	approximation \smash{$\widehat{\mtx{B}}_s$} that satisfies 
    \begin{equation} \label{eqn:rsvd-err-bd}
    \Expect %
    \fnormsq{\mtx{B} - \widehat{\mtx{B}}_s}
    	\leq \left(1 + \frac{r}{s - r - 1}\right) \cdot \sum_{i>r} \sigma^2_i (\mtx B).
    \end{equation}
    On average, to obtain a tolerance $\eps > 0$ in the error bound~\eqref{eqn:low-rank-error-goal},
    it suffices that $s \geq 1 + r + r/\eps$.
\end{theorem}

This theorem places no restrictions on the input matrix (such as a spectral gap).
To unpack the result, recall that the sum $\sum_{i > r} \sigma_i^2$
on the right-hand side of~\eqref{eqn:rsvd-err-bd}
is the Eckart--Young error~\eqref{eqn:eckart-young} in the best rank-$r$ approximation.
The factor $\eps = r/(s-r-1)$ reflects the loss that we suffer
from constructing the approximation $\widehat{\mtx{B}}_s$
using the random projector $\mtx{P}_{\mtx{Y}}$
instead of the ideal projector $\mtx{P}_r$.
By increasing the number $s$ of samples, we can drive the error tolerance $\eps$ down
as far as we like.  In particular, $s = 2r + 1$ yields a tolerance $\eps = 1$.
This type of approximation might seem too weak to be useful.  But keep in mind
that the main use case for the randomized SVD algorithm
is when the optimal error is tiny because the singular values decay
exponentially~\cite{HMT11:Finding-Structure,TW23:Randomized-Algorithms}.

We will sketch the proof of \Cref{thm:rsvd} to show how the properties of the
random test matrix $\mtx{\Omega}$ enter into the bound.  As with the randomized
power method, the key insight is that the random test matrix $\mtx{\Omega}$
aligns reasonably well with the leading right singular vectors of the
input matrix $\mtx{B}$.  At the same time, the test matrix does not align
too strongly with the trailing right singular vectors.  To quantify these
properties, we depend on some basic facts from random matrix theory.

As an aside, results like \Cref{thm:rsvd} also hold for the
spectral-norm error~\cite[Thm.~10.6]{HMT11:Finding-Structure},
but the proof requires some technical arguments that do not contribute
new insight about the mechanism that drives the randomized SVD algorithm.
Likewise, we can obtain probability bounds for the randomized SVD
after some additional effort; see \cite[Thms.~10.7 and 10.8]{HMT11:Finding-Structure}.

\begin{proof}[Proof of \Cref{thm:rsvd}]
Fix the target rank $r \leq m \wedge n$ and the number $s$ of random test vectors.
We can express the SVD~\eqref{eqn:B-svd} of the input matrix $\mtx{B}$
in block matrix form to isolate the leading $r$ singular directions:
\[
\mtx{B} =
	\begin{bmatrix} \mtx{U}_r & \mtx{U}_{\perp} \end{bmatrix}
	\begin{bmatrix} \mtx{\Sigma}_r & \mtx{0} \\ \mtx{0} & \mtx{\Sigma}_{\perp} \end{bmatrix}
	\begin{bmatrix} \mtx{V}_r & \mtx{V}_{\perp} \end{bmatrix}^*.
\]
Explicitly, $\mtx{U}_r \in \R^{m \times r}$, and $\mtx{\Sigma}_r = \diag(\sigma_1, \dots, \sigma_r)$,
and $\mtx{V}_r \in \R^{n \times r}$.

Next, extract the components of the test matrix $\mtx{\Omega}$ in the basis of right singular vectors:
\[
\mtx{\Omega}_r \coloneqq \mtx{V}_r^* \mtx{\Omega}
\quad\text{and}\quad
\mtx{\Omega}_{\perp} \coloneqq \mtx{V}_{\perp}^* \mtx{\Omega}.
\]
This decomposition is similar to the one we used to analyze the power method.
The matrix $\mtx{\Omega}_r$ describes the alignment of the test matrix with
the leading right singular vectors, while $\mtx{\Omega}_{\perp}$ describes
the alignment with the trailing right singular vectors.

After some linear algebra (omitted), %
we can develop the following (deterministic) bound for the
approximation error~\cite[Thm.~9.1]{HMT11:Finding-Structure}.
For a streamlined proof, see~\cite[Prop.~8.5]{TW23:Randomized-Algorithms}.

\begin{proposition}[Randomized SVD: Approximation error] \label{prop:rsvd-determ}
With notation as above, assume that $\rank(\mtx{\Omega}_r) = r$.
Then the rank-$s$ approximation \smash{$\widehat{\mtx{B}}_s$} produced by \Cref{alg:rSVD}
satisfies
\begin{equation} \label{eqn:rsvd-determ}
\fnormsq{ \mtx{B} - \widehat{\mtx{B}}_s }
	\leq \fnormsq{ \mtx{\Sigma}_{\perp} } + \fnormsq{ \mtx{\Sigma}_{\perp} \mtx{\Omega}_{\perp} \mtx{\Omega}_{r}^\dagger }
\end{equation}
\end{proposition}

The first term on the right-hand side of~\eqref{eqn:rsvd-determ} reflects the trailing singular values of the matrix.
Every rank-$r$ approximation~\eqref{eqn:eckart-young} suffers this loss.
The second term is the deficit from using the test matrix $\mtx{\Omega}$.
We want the component $\mtx{\Omega}_r$ in the leading directions to be
well conditioned (ideally, an identity matrix).  We want the component
$\mtx{\Omega}_{\perp}$ in the trailing directions to be as small as possible.

It is instructive to compare \Cref{prop:rsvd-determ} with the error bound~\eqref{eqn:rpm-error-decomp}
for the power method when $t = 1$.  Indeed, the quantity $\fnormsq{\mtx{\Omega}_r^\dagger}$
is analogous to the term $\omega_1^2$ in the \hilite{denominator} of~\eqref{eqn:rpm-error-decomp}.
Meanwhile, the quantity $\fnormsq{\mtx{\Sigma}_{\perp} \mtx{\Omega}_{\perp}}$
is analogous with the sum $\sum_{i > 1} \omega_i^2 \lambda_i^{2}$ in the numerator of~\eqref{eqn:rpm-error-decomp}.

Now, we can exploit the rotational invariance of the standard normal test matrix $\mtx{\Omega}$
to evaluate the expectation of the error bound~\eqref{eqn:rsvd-determ}.
Note that the columns of the leading singular vectors $\mtx{V}_r$
and the trailing singular vectors $\mtx{V}_{\perp}$ are mutually orthogonal.
Therefore, the associated components $\mtx{\Omega}_r \in \R^{r \times s}$ and $\mtx{\Omega}_{\perp} \in \R^{(n-r) \times s}$
of the test matrix are \hilite{standard normal matrices} that are statistically \hilite{independent}.
In particular, if $s \geq r$, then $\rank(\mtx{\Omega}_r) = r$ with probability one.
This holds true regardless of the right singular vectors of $\mtx{B}$.

To evaluate the expected deficit in the error bound~\eqref{eqn:rsvd-determ},
we use independence to introduce conditional expectations.  Thus,
\[
\Expect \fnormsq{ \mtx{\Sigma}_{\perp} \mtx{\Omega}_{\perp} \mtx{\Omega}_{r}^\dagger }
	= \Expect_{\mtx{\Omega}_r} \big[ \Expect_{\mtx{\Omega}_{\perp}} \fnormsq{ \mtx{\Sigma}_{\perp} \mtx{\Omega}_{\perp} \mtx{\Omega}_{r}^\dagger } \big]
	= \Expect_{\mtx{\Omega}_r} \big[  \fnormsq{ \mtx{\Sigma}_{\perp} } \cdot \fnormsq{ \mtx{\Omega}_{r}^\dagger } \big].
\]
The second identity follows when we write out the Frobenius norm in coordinates
and find the expectation over the entries of $\mtx{\Omega}_{\perp}$ by direct calculation.
The Frobenius norm of the remaining random matrix can be rewritten as a trace:
\[
\Expect_{\mtx{\Omega}_r} \fnormsq{ \mtx{\Omega}_{r}^\dagger }
	= \Expect_{\mtx{\Omega}_r} \trace\big[ (\mtx{\Omega}_r \mtx{\Omega}_r^*)^{-1} \big]
	= \frac{r}{s-r-1}.
\]
Indeed, since $\mtx{\Omega}_r \in \R^{r \times s}$ is standard normal,
the product $\mtx{W} = \mtx{\Omega}_r\mtx{\Omega}_r^*$
follows the standard $\textsc{wishart}(s, \Id_r)$ distribution with $s$ degrees of freedom.
The calculation of $\Expect \trace[ \mtx{W}^{-1} ]$ is a classical result
from statistics~\cite[Prop.~A.6]{HMT11:Finding-Structure}.

To conclude, take the expectation of \eqref{eqn:rsvd-determ}, and introduce the last two displays:
\[
\Expect \fnormsq{ \mtx{B} - \widehat{\mtx{B}}_s }
	\leq \fnormsq{ \mtx{\Sigma}_{\perp} } + \Expect \fnormsq{ \mtx{\Sigma}_{\perp} \mtx{\Omega}_{\perp} \mtx{\Omega}_{r}^\dagger }
	\leq \left( 1 + \frac{r}{s-r-1} \right) \cdot \fnormsq{ \mtx{\Sigma}_{\perp} }.
\]
Last, we recognize $\fnormsq{ \mtx{\Sigma}_{\perp} } = \sum_{i > r} \sigma_i^2(\mtx{B})$ as the Eckart--Young error.
\end{proof}

\subsubsection{Discussion}

We have shown that the random choice of test matrix plays an integral role
in the randomized SVD algorithm for low-rank matrix approximation.  This example
may seem different in spirit from the randomized initialization for the
power method, but the algorithms are actually close relatives.

Indeed, we can extend the randomized SVD to an algorithm called \hilite{randomized subspace
iteration} by repeated multiplication with the test matrix:

\begin{listbox}
\begin{enumerate}
\item	Sample $\mtx{\Omega} = \begin{bmatrix} \vct{\omega}^{(1)} & \dots & \vct{\omega}^{(s)} \end{bmatrix} \in \R^{n \times s}$
with i.i.d.~standard normal columns.

\item	Set the initial iterate $\mtx{X}_0 = \mtx{\Omega}$.

\item	For $t = 1, 2, 3, \dots, T$, compute
\[
\mtx{Q}_t \coloneqq \texttt{orth}(\mtx{B}\mtx{X}_{t-1})
\quad\text{and}\quad
\mtx{X}_t \coloneqq \mtx{B}^* \mtx{Q}_{t}
\]

\item	Return the approximation $\widehat{\mtx{B}}_s = \mtx{Q}_T \mtx{X}_T^*$.
\end{enumerate}
\end{listbox}

\noindent
The randomized SVD (\Cref{alg:rSVD}) is the special case of this procedure
with $T = 1$.

Unrolling the recurrence, we find that randomized subspace iteration
produces approximations %
\[
\widehat{\mtx{B}}_s = \mtx{Q}_t (\mtx{Q}_t^* \mtx{B})
\quad\text{where}\quad
\mtx{Q}_t = \texttt{orth}( (\mtx{BB}^*)^{t-1} \mtx{B} \mtx{\Omega} )
\quad\text{for $t = 1, 2, 3, \dots$.}
\]
Much as the power method drives its iterates toward the leading eigenspace,
subspace iteration drives $\range(\mtx{Q}_t)$ so that it aligns with $\range(\mtx{U}_r)$,
the leading left singular subspace of $\mtx{B}$.
The range of the random starting matrix $\mtx{\Omega}$ has sufficient alignment with the leading
right singular vectors $\mtx{V}_r$ to begin this process.  Meanwhile, the range of the starting
matrix $\mtx{\Omega}$ is oblique with the trailing right singular vectors $\mtx{V}_{\perp}$,
so they do not interfere too much with the convergence.
See~\cite{HMT11:Finding-Structure,TW23:Randomized-Algorithms}
for the analysis of randomized subspace iteration.

We can also study randomized block Lanczos methods for low-rank matrix
approximation~\cite{RST09:Randomized-Algorithm,HMST11:Algorithm-Principal,MM15:Randomized-Block}.
These algorithms exhibit much faster convergence than randomized subspace iteration,
so they are especially valuable for matrices whose singular values decay very slowly.
For these methods too, the random starting matrix is a critical ingredient.
See~\cite{TW23:Randomized-Algorithms} for a recent survey.

\section{Progress on average}\label{sec:progress_avg}

In the last section, we considered algorithms that start from a random initialization
and drive this point toward a solution to the computational problem.  Of course, there
is no reason to limit the role of probability to the initialization.
This section introduces another template where we make random choices \hilite{at each step}
of an iterative algorithm.  The methods we study share the feature that
each step reduces the error (or some other measure of progress) \hilite{on average}.

\begin{idea}[Progress on average]
	At each step of an (iterative) algorithm, make a random choice so that
	the expected error decreases.
\end{idea}

In optimization, a familiar example of this template is stochastic gradient %
iteration~\cite{Bot10:Large-Scale-Machine}.
This is a randomized variant of gradient descent in which the gradient
is replaced by an unbiased random estimate that is cheaper to compute.
Just as a small step in the direction of the negative gradient reduces the value
of the objective function, a step in the direction of the random gradient
estimate reduces the objective value on average.  Even if the individual
gradient estimates are very poor, we can still guarantee convergence of
the randomized algorithm.

We will study two problems in numerical linear algebra that we can
tackle using randomized algorithms that make progress on average.
First, we study the randomized Kaczmarz iteration for solving
an overdetermined least-squares problem.
Second, we describe the randomly pivoted partial Cholesky
method for computing a low-rank approximation of a psd matrix.

\subsection{Randomized Kaczmarz}
\label{sec:rand-kacz}

Another classic linear algebra problem, dating to the time of Gauss,
is the linear least-squares problem, often used for fitting linear
models to data.  We will focus on the overdetermined
setting where the number of observations (dramatically) exceeds the number
of variables in the model.  This problem arises in a range of applications,
such as statistical estimation, signal processing, machine learning, and even
as a subroutine in other linear algebra and optimization computations.

\begin{problem}[Overdetermined least-squares]
Given a tall matrix $\mtx{A} \in \R^{n \times d}$ with $n \gg d$
and a response vector $\vct{b} \in \R^n$, find a solution to
the overdetermined least-squares problem:
\[
\minimize_{\vct{x} \in \R^d}\quad \tfrac{1}{2} \normsq{\mtx A \vct x - \vct b}_2.
\]
\end{problem}

For large-scale instances, it may not be desirable or even feasible to
tackle the least-squares problem using direct methods, which cost
$\mathcal{O}(d^2 n)$ arithmetic operations.
Instead, we can seek iterative algorithms that approximate the solution
at a lower computational cost.  In particular, we will study methods that 
enforce individual equations. %
We will see that enforcing a \hilite{randomly chosen}
equation leads to a simple, elegant algorithm with
nice convergence guarantees.

\subsubsection{The Kaczmarz iteration}

There are many iterative methods for solving least-squares problems.
Let us introduce a classic algorithm, attributed to Kaczmarz~\cite{Kac37:Angenaherte}.
This method has historically been popular for computed tomography and digital signal
processing applications~\cite{SV09:Randomized-Kaczmarz}.

Consider a tall matrix $\mtx{A} \in \R^{n \times d}$ with $n \gg d$,
and write $\vct{a}_i^*$ for the $i$th \hilite{row} of the matrix.
Choose a response vector $\vct{b} = (b_1, \dots, b_n)^* \in \R^n$.
We can frame the overdetermined least-squares problem as
\begin{equation} \label{eqn:overdetermined-ls-eqns}
\minimize_{\vct{x} \in \R^d}\quad \frac{1}{2} \sum_{i=1}^n \big( \ip{\vct{a}_i}{ \vct{x} } - b_i \big)^2.
\end{equation}
As usual, $\ip{\cdot}{\cdot}$ denotes the standard inner product on $\R^d$.
Each term $i = 1, \dots, n$ in the sum corresponds to the squared residual in a single
equation, and we can develop algorithms that work with the individual equations.

The Kaczmarz iteration begins with an initial iterate $\vct{x}_0 \in \R^d$,
often taken to be the zero vector.  At each iteration $t = 1, 2, 3,\dots$,
we select an equation $j(t)$ and update the previous iterate $\vct{x}_{t-1}$ so this equation holds exactly:
\begin{equation} \label{eqn:kacz-update}
\vct{x}_{t} = \vct{x}_{t-1} - \frac{\ip{\vct{a}_{j(t)}}{ \vct{x}_{t-1} } - b_{j(t)}}{\norm{\vct{a}_{j(t)}}_2^2} \cdot \vct{a}_{j(t)}
\quad\text{for $t = 1, 2, 3, \dots$.}
\end{equation}
A short calculation confirms that $\ip{\vct{a}_{j(t)}}{\vct{x}_t} = b_{j(t)}$.
Each iteration of the Kaczmarz method costs just $\mathcal{O}(d)$ operations,
irrespective of the number $n$ of equations.
As an aside, note that the update amounts to a step in the direction
of the gradient of the $i$th squared residual in~\eqref{eqn:overdetermined-ls-eqns},
so we can interpret the Kaczmarz iteration as a type of gradient descent method.

To design a Kaczmarz algorithm, we need to specify a control rule $j(t)$
for picking the equation at iteration $t$.
In the classical literature, it was most common to repeatedly cycle through
the equations in the order given.
Another alternative is to arrange the equations in random order
and cycle through them. 
A third possibility is to enforce the equation that has the
largest residual, although this rule can be expensive to implement.
These strategies all lead to convergent algorithms,
but it is challenging to understand the \hilite{rate of convergence}
because it depends in a complicated way on the problem data
and the control rule.

\subsubsection{The role of randomness}

To address the limitations of the classic Kaczmarz method,
Strohmer \& Vershynin~\cite{SV09:Randomized-Kaczmarz} proposed a variant
that uses a \hilite{randomized control rule}.
Their idea is that we can make \hilite{progress on average} toward
the solution if we select an equation at random at each step.
To control the variability of the randomized algorithm,
the sampling probabilities should also reflect the ``importance'' of each equation.

With this preface in mind, we describe a random control rule for the Kaczmarz method.
At each iteration $t$, the random row $J(t)$ is drawn \hilite{independently}
from all previous random choices $(J(1), \dots, J(t-1))$.
The distribution of the random row $J(t)$ takes the form
\begin{equation} \label{eqn:rk-control}
\Prob{ J(t) = i } = \frac{\normsq{\vct{a}_i}_2}{\fnormsq{ \mtx{A} }}
\quad\text{for $i = 1, \dots, n$.}
\end{equation}
In other words, the ``importance'' of an equation to the least-squares problem~\eqref{eqn:overdetermined-ls-eqns}
is proportional to its squared norm, for reasons that will become clear in the analysis.
We already encountered a similar distribution in our discussion of approximate
matrix multiplication (\Cref{sec:approx-mtx-mult}).
This is an instance of \hilite{importance sampling}, a variance reduction
technique that was developed in the Monte Carlo approximation literature.

\Cref{alg:Kaczmarz} lists the pseudocode for
this \hilite{randomized Kaczmarz} method.
Note that it takes $\mathcal{O}(dn)$ operations to compute
the sampling probabilities exactly.
Each iteration requires $\mathcal{O}(d)$ operations.

\begin{algorithm}[t]
\begin{algbox}[1]
\caption{\textit{Randomized Kaczmarz.}}
\label{alg:Kaczmarz}

\Require	Matrix $\mtx{A} \in \R^{n \times d}$, response $\vct{b} \in \R^{n}$, initial iterate $\vct{x}_0 \in \R^d$, number $T$ of iterations
\Ensure		Solution estimate $\vct{x}_T \in \R^d$

\vspace{5pt}
\hrule
\vspace{5pt}

\Function{RandKaczmarz}{$\mtx{A}$, $\vct{b}$; $\vct{x}_0$}

\State	$r_i = \norm{\vct{a}_{i}}_2^2$ for $i = 1, \dots, n$
	\Comment $\vct{a}_i^*$ is $i$th \hilite{row} of matrix

\State	$p_i = r_i / \sum_{j=1}^n r_j$
	\Comment Sampling probabilities

\For{$t = 1, 2, 3, \dots, T$}
\State	Sample row index $j(t) \in \{1, \dots, n\}$ from the distribution $(p_i)$
\State	$\vct{x}_t = \vct{x}_{t-1} - (\ip{\vct{a}_{j(t)} }{\vct{x}_{t-1}} - b_i ) \vct{a}_{j(t)} / \normsq{\vct{a}_{j(t)}}_2$
	\Comment Update iterate
\EndFor

\EndFunction
\end{algbox}
\end{algorithm}

\subsubsection{Randomized Kaczmarz: Analysis}

In this section, we give an elegant analysis of the randomized Kaczmarz iteration
due to Strohmer \& Vershynin~\cite{SV09:Randomized-Kaczmarz}.
Following their presentation, we assume that $\mtx{A}$ is injective
and that the least-squares problem is consistent: $\vct{b} \in \range(\mtx{A})$.
Under these extra hypotheses,
the least-squares problem admits a unique solution $\vct{x}_{\star} \in \R^d$ that is exact: $\mtx{A} \vct{x}_{\star} = \vct{b}$.
Thus, we can measure progress of the iterates toward the solution:
\[
\normsq{ \vct{x}_t - \vct{x}_{\star} }_2
\quad\text{for $t = 0,1,2,\dots$.}
\]
We will prove that, on average, the randomized control rule~\eqref{eqn:rk-control}
reduces the squared norm of the error by a constant factor at each iteration.

The analysis shows that the convergence factor of the randomized Kaczmarz method
depends on the \hilite{Demmel condition number}~\cite{Dem88:Probability-Numerical}
of the matrix:
\begin{equation} \label{eqn:kappa-dem}
\kappa_{\dem} \coloneqq \kappa_{\dem}(\mtx{A}) \coloneqq \frac{ \fnorm{ \mtx{A} } }{ \sigma_{\min}(\mtx{A}) }.
\end{equation}
Here, $\sigma_{\min}(\mtx{A})$ is the $d$th largest singular value of the tall matrix $\mtx{A} \in \R^{n \times d}$.
For comparison, recall the standard condition number $\kappa \coloneqq \sigma_{\max} / \sigma_{\min}$
for linear systems.
We have the inequalities $\kappa \leq \kappa_{\dem} \leq \sqrt{d} \kappa$ for every matrix with $d$ columns.
Much as the intrinsic dimension refines the notion of the rank of a psd matrix,
the Demmel condition number $\kappa_{\dem}$ accounts for the number of energetic directions in the
range of the matrix, whereas the standard condition number $\kappa$ does not.

\begin{theorem}[Randomized Kaczmarz; Strohmer \& Vershynin 2009] \label{thm:randkacz}
Consider a consistent system of linear equations $\mtx{A} \vct{x} = \vct{b}$
where the matrix $\mtx{A}$ is injective, and let $\vct{x}_{\star}$ be the (unique) solution.
For each $t = 1, 2, 3, \dots$, the randomized Kaczmarz iteration (\Cref{alg:Kaczmarz})
yields the error bound
\[
\Expect \normsq{ \vct{x}_t - \vct{x}_{\star} }_2
	\leq \big(1 - \kappa_{\dem}^{-2} \big)^t \cdot \normsq{ \vct{x}_0 - \vct{x}_{\star} }_2.
\]
\end{theorem}

Before turning to the proof, let us discuss what this result means~\cite[Sec.~2.1]{SV09:Randomized-Kaczmarz}.
For a tolerance $\eps > 0$, we can bound the expected number of iterations
that suffice to drive the error below the tolerance.
\[
T \geq 2 \kappa_{\dem}^2 \log( 1 /\eps)
\quad\text{implies}\quad
\Expect \normsq{ \vct{x}_T - \vct{x}_{\star} }_2
	\leq \eps^2 \norm{ \vct{x}_0 - \vct{x}_{\star} }_2.
\]
Suppose that the standard condition number $\kappa = \kappa(\mtx{A})$ is a constant
that does not depend on the matrix dimensions.
Then $\kappa_{\dem}^2 = \mathcal{O}(d)$.
We determine that $T = \mathcal{O}(d \log(1/\eps))$ iterations suffice to
obtain $\eps$-error, relative to the initial guess.
If the sampling probabilities in the control rule are given, then
the cost of the whole algorithm is just $\mathcal{O}(d^2 \log(1/\eps))$ operations,
in contrast to the $\mathcal{O}(n d^2)$ cost of a direct method.
Indeed, we may not even need to enforce all of the equations!

\begin{proof}[Proof of \Cref{thm:randkacz}]
We begin with a short linear algebra calculation
that describes the evolution of the error in the
deterministic Kaczmarz iteration.
Since $\vct{x}_{\star}$ satisfies the linear equations,
we have $\ip{\vct{a}_{i}}{\vct{x}_{\star}} = b_i$ for each row $i$.
For each iteration $t = 0, 1, 2, \dots$,
define the error vector $\vct{e}_t \coloneqq \vct{x}_t - \vct{x}_{\star}$.
The Kaczmarz update~\eqref{eqn:kacz-update} ensures that the error
vector satisfies
\[
\begin{aligned}
\vct{e}_t \coloneqq \vct{x}_t - \vct{x}_{\star}
	&= (\vct{x}_{t-1} - \vct{x}_{\star})
	- \frac{\ip{\vct{a}_{j(t)}}{ \vct{x}_{t-1} - \vct{x}_{\star}}}{\normsq{ \vct{a}_{j(t)}}_2} \cdot \vct{a}_{j(t)} \\
	&= \vct{e}_{t-1} - \frac{\vct{a}_{j(t)} \vct{a}_{j(t)}^*}{\normsq{\vct{a}_j}_2} \cdot \vct{e}_{t-1}
	\eqqcolon \big(\Id - \mtx{P}_{j(t)} \big) \vct{e}_{t-1}.
\end{aligned}
\]
For each row index $i$, we define the orthogonal projector $\mtx{P}_{i} \coloneqq \vct{a}_i \vct{a}_i^* / \normsq{\vct{a}_i}_2$
onto the span of $\vct{a}_i$.
Taking the squared norm of the error evolution equation, %
\begin{equation} \label{eqn:kacz-error-redux}
\normsq{ \vct{e}_t }_2
	= \normsq{ \big(\Id - \mtx{P}_{j(t)} \big) \vct{e}_{t-1} }_2
	= \lip{ \vct{e}_{t-1} }{ \big(\Id - \mtx{P}_{j(t)} \big) \vct{e}_{t-1} }.
\end{equation}
We can bound the reduction in the error by understanding how much the projection onto the row $j(t)$ cuts down the residual.
When $j(t)$ is deterministic, this argument is challenging.

Instead, we implement the randomized control rule~\eqref{eqn:rk-control}.
In this case, the projector $\mtx{P}_{J(t)}$ is a random matrix.
Averaging with respect to the randomness in $J(t)$,
\[
\Expect_{J(t)}[ \mtx{P}_{J(t)} ] = \sum_{i = 1}^n \mtx{P}_i \cdot \Prob{ J(t) = i }
	= \sum_{i = 1}^n \frac{\vct{a}_i \vct{a}_i^*}{\normsq{ \vct{a}_i }_2 } \cdot \frac{\normsq{ \vct{a}_i }_2 }{\fnormsq{\mtx{A}}} %
	= \frac{1}{\fnormsq{\mtx{A}}} \mtx{A}^* \mtx{A}.
\]
The importance sampling distribution for the row index is designed to cancel out the squared norms.
As a consequence, the expected projector is proportional to the Gram matrix $\mtx{A}^* \mtx{A}$,
so the convergence rate of the randomized Kaczmarz method will reflect
the properties of the input matrix $\mtx{A}$.

Taking the expectation of~\eqref{eqn:kacz-error-redux} and exploiting linearity,
\[
\begin{aligned}
\Expect_{J(t)} \normsq{ \vct{e}_t }_2
	&= \lip{ \vct{e}_{t-1} }{ \big(\Id - \Expect_{J(t)}[ \mtx{P}_{J(t)}] \big) \vct{e}_{t-1} }
	= \lip{ \vct{e}_{t-1} }{ \big(\Id - \fnorm{\mtx{A}}^{-2} \mtx{A}^* \mtx{A} \big) \vct{e}_{t-1} } \\
	&\leq \lambda_1\big( \Id - \fnorm{\mtx{A}}^{-2} \mtx{A}^* \mtx{A} \big) \cdot \normsq{\vct{e}_{t-1}}_2
	= \left( 1 - \frac{\sigma_{\min}(\mtx{A})^2}{\fnormsq{\mtx{A}}} \right) \cdot \normsq{\vct{e}_{t-1}}_2.
\end{aligned}
\]
Recalling the definition~\eqref{eqn:kappa-dem} of the Demmel condition number,
we reach
\begin{equation} \label{eqn:rk-supermartingale}
\Expect_{J(t)} \normsq{ \vct{e}_t }_2 \leq \big(1 - \kappa_{\dem}^{-2} \big) \cdot \normsq{ \vct{e}_{t-1} }_2.
\end{equation}
\hilite{At each iteration, the randomized Kaczmarz method makes progress on average.}
Finally, we take the total expectation using the tower rule
and the independence of control sequence: 
\[
\Expect \normsq{ \vct{e}_t }_2 \leq \big(1 - \kappa_{\dem}^{-2} \big) \cdot \Expect \normsq{ \vct{e}_{t-1} }_2.
\]
Unrolling this recurrence, we obtain the stated result.
\end{proof}

The proof hinges on the formula~\eqref{eqn:rk-supermartingale},
which states that the error sequence $(\normsq{\vct{e}_t}_2 : t = 0, 1, 2, \dots)$
composes a positive supermartingale.
This observation has some striking implications for the long-term behavior of the
randomized Kaczmarz method~\cite{CP12:Almost-Sure-Convergence}.  In particular,
a standard argument using the Borel--Cantelli lemma provides that
\[
(1 - \varrho)^{-t} \cdot \normsq{ \vct{e}_t }_2 \to 0
\quad\text{almost surely for each $\varrho < \kappa_{\dem}^{-2}$.}
\]
We can also obtain uniform concentration inequalities for the error sequence
by an application of Ville's inequality~\cite[Thm.~26.12]{Tro23:Probability-Theory-LN}.

The argument also shows that the precise distribution~\eqref{eqn:rk-control}
of the randomly chosen equation plays an important role in the performance
of the algorithm.
Instead, suppose that we sampled an equation uniformly at random at each step.
Then we obtain an error recurrence of the form
\[
\Expect_{J(t)} \normsq{\vct{e}_t}_2
	\leq \left( 1 - \frac{\sigma_{\min}(\mtx{A})^2}{n \cdot \max\nolimits_{i} \normsq{\vct{a}_i}_2} \right) \cdot
	\normsq{\vct{e}_{t-1}}_2.
\]
Since $\fnormsq{\mtx{A}} \leq n \cdot \max_i \normsq{\vct{a}_i}_2$,
the weighted sampling probabilities always result in faster convergence
unless the rows all share the same norm.

\subsubsection{Randomized Kaczmarz: Discussion}

The paper~\cite{SV09:Randomized-Kaczmarz} of Strohmer \& Vershynin
generated a substantial body of follow-up work.
Their result, \Cref{thm:randkacz}, states that the randomized Kaczmarz
method (\Cref{alg:Kaczmarz}) converges for a consistent linear system
with an injective matrix.
Zouzias \& Freris~\cite{ZF13:Randomized-Extended} developed a variant
of the randomized Kaczmarz algorithm that can handle inconsistent systems
where the matrix is not injective.
Needell \& Tropp~\cite{NT14:Paved-Good} studied a generalization
of the randomized block Kaczmarz method that enforces several equations
at each iteration, similar to the mini-batching strategy used in
stochastic gradient methods for machine learning.
Needell et al.~\cite{NSW16:Stochastic-Gradient} also demonstrated
that the randomized Kaczmarz method is a particular instance of
stochastic gradient iteration.

The randomized Kaczmarz algorithm is closely related to other
types of randomized algorithms for least-squares problems.
Leventhal \& Lewis~\cite{LL10:Randomized-Methods} observed that we can
design and analyze \hilite{randomized coordinate descent} algorithms using
the same principles; these methods work with columns of the input matrix
rather than the rows.
Gower \& Richtarik~\cite{GR15:Randomized-Iterative}
studied a general class of iterative algorithms for solving linear systems,
called \hilite{sketch-and-project methods}.
This class includes both randomized Kaczmarz
and randomized coordinate descent as special cases.

In most applications, the randomized Kaczmarz method is not a very
competitive algorithm for solving least-squares problems.
Nevertheless, it provides a nice illustration of the progress-on-average paradigm.
The papers cited above indicate that related ideas arise in other types of
randomized iterative algorithms.
In the next section, we consider a second example in numerical linear algebra.

\subsection{Randomly pivoted Cholesky}
\label{sec:rpc}

The randomized SVD algorithm is designed to produce a low-rank approximation
of a (rectangular) matrix, accessed via matvecs.  In some applications, however,
the cost of forming the input matrix and applying it to a vector is prohibitive.
Here, we consider the problem of approximating a psd matrix where we pay
to access each entry.

\begin{problem}[Low-rank psd approximation from entries]
Consider a psd matrix $\mtx{A} \in \Sym_n^+(\R)$ that we access via entry
evaluations: $(j, k) \mapsto a_{jk}$.  The task is to produce a low-rank
psd approximation of $\mtx{A}$, while limiting the number of entry
evaluations.
\end{problem}

The key example is a \hilite{kernel matrix},
which tabulates the pairwise similarities among a family of $n$ data points.
To compute an entry of a kernel matrix, we must compare two data points,
so the cost of generating and storing the matrix is $\mathcal{O}(n^2)$,
which is \hilite{quadratic} in the number of data points.
Thus, we prefer to avoid generating the entire kernel matrix.

Kernel methods reduce data analysis problems to
linear algebra problems with the kernel matrix.
One common strategy for scaling kernel methods to large data sets
is to replace the kernel matrix by a low-rank approximation.
The goal is to produce the best approximation of the kernel matrix
that we can while evaluating as few entries of the kernel matrix as possible.

In this section, we will describe a simple but powerful randomized algorithm
for solving the low-rank psd approximation problem.
At each step, this algorithm randomly updates the approximation
so that the approximation error decreases on average.
Our treatment closely follows the paper~\cite{CETW23:Randomly-Pivoted}.

\subsubsection{Column Nystr{\"o}m approximation}

Let $\mtx{A} \in \Sym_n^+(\R)$ be a psd matrix that we access by evaluating individual entries.
One way to produce an approximation after querying a small number of entries
is to form the approximation from a small number of (adaptively chosen) {columns}.

As it happens, there is a canonical way to approximate a psd matrix
from a specified set of columns.
Given a list $\set{S} \subseteq \{1, 2, 3, \dots, n\}$ of column indices,
we can form the \hilite{column Nystr{\"o}m approximation}:
\[
\mtx{A}\langle\set{S}\rangle \coloneqq \mtx{A}(:, \set{S}) \mtx{A}(\set{S}, \set{S})^{\dagger} \mtx{A}(\set{S}, :).
\]
Since $\mtx{A}$ is psd, the rows $\mtx{A}(\set{S}, :)$ are the (conjugate) transpose
of the columns $\mtx{A}(:, \set{S})$.
We use the colon operator to indicate a row or column submatrix.

The column Nystr{\"o}m approximation has several remarkable properties:

\begin{enumerate}

\item	The range of the approximation coincides with the span of the distinguished columns of the target matrix: $\range(\mtx{A}\langle\set{S}\rangle) = \range(\mtx{A}(:,\set{S}))$.

\item	In the semidefinite order, the approximation is dominated by the target matrix: $\mtx{0} \psdle \mtx{A}\langle\set{S}\rangle \psdle \mtx{A}$.
\end{enumerate}

\noindent
Moreover, among all approximations that satisfy (1)--(2), the column Nystr{\"o}m approximation is maximal
with respect to the semidefinite order~\cite[Thm.~5.3]{And05:Schur-Complements}.

We measure the error in the column Nystr{\"o}m approximation
with respect to the trace norm $\norm{\cdot}_{*}$:
\begin{equation} \label{eqn:psd-trace-error}
\norm{ \mtx{A} - \mtx{A}\langle\set{S}\rangle }_{*}
	= \trace( \mtx{A} - \mtx{A}\langle\set{S}\rangle ).
\end{equation}
Owing to property (2), the residual is always psd, so the trace norm agrees with the trace.
Our goal is find a set $\set{S}$ of $k$ columns that make the trace-norm
error~\eqref{eqn:psd-trace-error} as small as possible.

\subsubsection{Pivoted partial Cholesky}

We can compute a column Nystr{\"o}m approximation using a classic
technique from numerical linear algebra: \hilite{the pivoted partial Cholesky algorithm}.
This procedure incrementally forms a psd approximation of a psd matrix
by eliminating one column at a time.  For example, if we choose to eliminate
the first column of the psd matrix $\mtx{A}_0$, the Cholesky update takes the form
\[
\arraycolsep=5pt\def\arraystretch{1.5}
\mtx{A}_0 = \lt[ \begin{array}{c|c}
a_{11} & \vct a_{21}^*  \\ \hline
\vct a_{21} &\mtx A_{22}
\end{array}\rt]
\qquad \xmapsto{\text{update} } \qquad
\mtx A_1 = \lt[ \begin{array}{c|c}
0 & \vct 0  \\ \hline
\vct 0 &\mtx A_{22} - \frac{\vct a_{21} \vct a_{21}^* }{a_{11}}
\end{array}\rt]
\]
By direct calculation, you can check that the updated matrix $\mtx{A}_1$
remains psd.

Here is a conceptual description of the pivoted Cholesky algorithm
for a psd input matrix $\mtx{A} \in \Sym_n^+(\R)$.
Set the initial residual equal to the input matrix: $\mtx{A}_0 = \mtx{A}$.
Set the initial approximation equal to the zero matrix: $\smash{\widehat{\mtx{A}}_0} = \mtx{0}$.
At each step $t = 1, 2, 3, \dots, k$, we select a column index $s_t \in \{1, 2, 3, \dots, n\}$,
called a \hilite{pivot}.
We update the approximation using the $s_t$ column of the residual,
and we eliminate the $s_t$ column from the residual:
\begin{align}
\widehat{\mtx{A}}_t &\coloneqq \widehat{\mtx{A}}_{t-1} + \frac{\mtx{A}_{t-1}(:, s_t) \mtx{A}_{t-1}(s_t, :)}{\mtx{A}_{t-1}(s_t, s_t)}; \label{eqn:chol-approx} \\
{\mtx{A}}_t &\coloneqq {\mtx{A}}_{t-1} - \frac{\mtx{A}_{t-1}(:, s_t) \mtx{A}_{t-1}(s_t, :)}{\mtx{A}_{t-1}(s_t, s_t)}.
\label{eqn:chol-resid}
\intertext{As we will see, the diagonal of the residual plays a key role.
We can track it with the formula}
\diag( \mtx{A}_t ) &= \diag( \mtx{A}_{t-1} ) - \frac{1}{\mtx{A}_{t-1}(s_t,s_t)} \cdot \abssq{\mtx{A}_{t-1}(:, s_t)}. \label{eqn:chol-diag}
\end{align}
The function $\abssq{\cdot} : \R^n \to \R^n_+$ returns the squared magnitudes of the entries of a vector.

\begin{fact}[Cholesky meets Nystr{\"o}m]
Suppose that we apply the pivoted partial Cholesky algorithm~\eqref{eqn:chol-approx}--\eqref{eqn:chol-resid} to a psd matrix $\mtx{A}$,
and we select the pivot set $\set{S} = \{s_1, \dots, s_k\}$ in any order.
The resulting matrix approximation $\widehat{\mtx{A}}_k$ equals the Nystr{\"o}m approximation $\mtx{A}\langle\set{S}\rangle$.
\end{fact}

\subsubsection{Cholesky: Evaluating fewer entries}

To address our computational problem, we must use an alternative formulation
of the pivoted partial Cholesky algorithm that avoids unnecessary entry evaluations.
To do so, we maintain the approximation in factored form:
\[
\widehat{\mtx{A}}_t = \mtx{F}_t \mtx{F}_t^*
\quad\text{for $t = 0, 1,2, \dots, r$.}
\]
The initial factor $\mtx{F}_0 = \mtx{0}$.  At each step $t$, we select a new column index $s_t$.
We extract the $s_t$ column from the input matrix and we subtract the $s_t$ column of the previous
approximation to obtain the $s_t$ column of the residual:
\[
\vct{c}_t \coloneqq \mtx{A}(:, s_t) - \widehat{\mtx{A}}_{t-1}(:, s_t)
	= \mtx{A}(:, s_t) - \mtx{F}_{t-1} \cdot (\mtx{F}_{t-1}(s_t, :))^*.
\]
We rescale this vector by the square-root of the diagonal entry $\mtx{A}_{t-1}(s_t, s_t) = \vct{c}_t(s_t)$,
and append the resulting vector to the factor matrix: 
\[
\mtx{F}_t \coloneqq \left[ \begin{array}{c|c} \mtx{F}_{t-1} & \vct{c}(s_t)^{-1/2} \vct{c}_t \end{array} \right].
\]
You can check that this procedure yields the same output as the standard
version of the Cholesky algorithm~\eqref{eqn:chol-approx}--\eqref{eqn:chol-resid}.

In each iteration, we only interact with the input matrix $\mtx{A}$ by reading
off the entries of the $s_t$ column.  After $k$ iterations, we have expended
$\mathcal{O}(k^2 n)$ arithmetic operations and $\mathcal{O}(kn)$ storage.

\subsubsection{Pivot selection rules}

How can we select pivots in the partial Cholesky method to achieve
the best approximation?  This problem is tricky because we only have
access to the current approximation, and we need to pay every time
we evaluate an entry of the input matrix.

One simple strategy is \hilite{uniform random pivoting}:
\[
s_t \sim \uniform\{1,2,3, \dots, n\}
\quad\text{for each $t = 1, \dots, k$.}
\]
This approach was proposed in~\cite{WS01:Using-Nystrom}, and it remains popular in the machine learning literature.
It is justified by the intuition that data points might represent
an i.i.d.~sample from a population, so one is just as good as another.

Another classic strategy is \hilite{greedy pivoting}~\cite{Hig90:Analysis-Cholesky}:
\[
s_t \in \arg\min \{ \mtx{A}_{t-1}(s, s) : s = 1, \dots, n \}
\quad\text{for each $t = 1, \dots, k$.}
\]
In other words, we find the largest diagonal of the residual,
and we select the associated column as the pivot.
To motivate this approach, recall that the entries of a psd matrix $\mtx{A}$
satisfy the inequality $\abssq{a_{ij}} \leq a_{ii} a_{jj}$.  Therefore, the
large diagonal entries of the residual point toward the columns that possibly
contain large entries (but may not). %

The uniform pivoting strategy and the greedy pivoting strategy
are both somewhat effective, but they also have serious failure modes.
The uniform random strategy is pure \hilite{exploration}
without any adaptivity to the matrix.
The greedy strategy is pure \hilite{exploitation},
governed entirely by the incomplete information in
the largest diagonal entry.

\subsubsection{The role of randomness}

Instead, we consider a random pivoting strategy that balances exploitation
and exploration.  We proceed from the intuition that the %
diagonal entries of the residual provide hints about which columns
might be significant.  Since this information is imperfect, we do
not want to place too much reliance on it.
Therefore, at each step of the iteration,
we sample a random column index
in proportion to the size of the diagonal entries.

More precisely, the \hilite{randomly pivoted Cholesky} algorithm uses the
following selection rule.  At iteration $t = 1, 2, 3, \dots, k$,
we draw the column index $s_t$ from the probability distribution
induced by the diagonal of the residual:
\begin{equation} \label{eqn:rpc-sampling}
\Prob{ s_t = j } = \frac{\mtx{A}_{t-1}(j, j)}{\trace( \mtx{A}_{t-1} )}
\quad\text{for $j = 1, 2, 3, \dots, n$.}
\end{equation}
This is another example of an importance sampling procedure.
We will see that this random pivot rule reduces the trace
of the residual \hilite{on average}.  The analysis also
highlights why this sampling distribution is a natural choice.

\Cref{alg:rpc} provides pseudocode for this approach.
Recall that we can track the diagonal entries of the residual
inexpensively using the formula~\eqref{eqn:chol-diag}.
To produce a rank-$k$ approximation, the algorithm only
reveals $(k+1) n$ entries of the input matrix:
its diagonal and the $r$ pivot columns.
We spend $\mathcal{O}(k^2 n)$ arithmetic to compute the
factors of the approximation, which require $\mathcal{O}(kn)$
storage.

\begin{algorithm}[t]
\begin{algbox}[1]
\caption{\textit{Randomly pivoted partial Cholesky.}}
\label{alg:rpc}

\Require	Psd matrix $\mtx{A} \in \Sym_n^+(\R)$; approximation rank $k$
\Ensure		Columns $\{s_1, \dots, s_k\}$; factor $\mtx{F} \in \R^{n \times k}$ of rank-$k$ approximation $\mtx{\widehat{A}}_k = \mtx{FF}^*$

\vspace{5pt}
\hrule
\vspace{5pt}

\Function{RPCholesky}{$\mtx{A}$; $k$}
\State	$\mtx{F} = \mtx{0}_{n \times k}$
\Comment	Factor of initial approximation
\State	$\vct{d} = \diag( \mtx{A} )$
\Comment	$n$ entry evals

\For{$t = 1,2, 3, \dots, k$}
\State	Sample $s_t \sim \vct{d} / \sum_{j=1}^n \vct{d}(j)$
\Comment	Sample w.r.t.~diagonal of current residual

\State	$\vct{c} = \mtx{A}(:, s_t) - \mtx{F} \cdot (\mtx{F}(s_t, :))^*$
\Comment	Column of residual via $n$ entry evals

\State	$\mtx{F}(:, t) = \vct{c} / (\vct{c}(s_t)^{1/2})$
\Comment	Update factor of approximation

\State	$\vct{d} = \vct{d} - \abssq{\mtx{F}(:,t)}$
\Comment	Update diagonal of residual

\State 	$\vct{d} = \max\{\vct{d}, \vct{0}\}$
\Comment	Improve numerical stability

\State	Stop when $\sum_{j=1}^n \vct{d}(j) < \eta \cdot \trace(\mtx{A})$.
\Comment	\textbf{\textcolor{dkgray}{Optional:}} Halt when trace is small
\EndFor

\EndFunction
\end{algbox}
\end{algorithm}

\subsubsection{Randomly pivoted Cholesky: Elements of the analysis}

The analysis of randomly pivoted Cholesky (\Cref{alg:rpc}) is significantly
more challenging than the other algorithms we have discussed.  Nevertheless,
it is easy to understand why it makes progress.
See~\cite[Thm.~5.1]{CETW23:Randomly-Pivoted} for full details.

Suppose that we initialize randomly pivoted Cholesky with the psd matrix $\mtx{A}_0 = \mtx{A}$,
and we perform one iteration to obtain a matrix $\mtx{A}_1$.  Define the
\hilite{expected residual map}:
\[
\mtx{\Phi}(\mtx{A}) \coloneqq \Expect \big[ \mtx{A}_1 \condbar \mtx{A}_0 = \mtx{A} \big].
\]
This function measures the progress that we make after one step of the algorithm.
A quick calculation yields a formula for the expected residual map:
\[
\begin{aligned}
\mtx{\Phi}(\mtx{A}) &= \sum_{j=1}^n \left[ \mtx{A} - \frac{\mtx{A}(:,j)\mtx{A}(j,:)}{\mtx{A}(j,j)} \right] \cdot \frac{\mtx{A}(j,j)}{\trace(\mtx{A})} \\
	&= \mtx{A} - \frac{1}{\trace(\mtx{A})} \sum_{j=1}^n \mtx{A}(:,j)\mtx{A}(j,:)
	= \mtx{A} - \frac{1}{\trace(\mtx{A})} \cdot \mtx{A}^2.
\end{aligned}
\]
We have used the definition~\eqref{eqn:chol-resid} of the residual and the
sampling distribution~\eqref{eqn:rpc-sampling}.

As a particular consequence, we can calculate the expected trace of the residual
after one step of the algorithm:
\[
\Expect \trace( \mtx{A}_1 ) = \left[ 1 - \frac{\trace(\mtx{A}^2)}{(\trace \mtx{A})^2} \right] \cdot \trace(\mtx{A})
	\leq \frac{n-1}{n} \cdot \trace(\mtx{A}).
\]
\hilite{In each iteration, we decrease the expected trace of the residual on average.}
The reduction is greatest when the eigenvalues of the residual are spiky,
and the reduction is least when the eigenvalues are flat.  In the latter
case, the residual does not admit a good low-rank approximation, so we
should not expect to make further progress.
Since the residual matrix $\mtx{A}_t = \mtx{A} - \widehat{\mtx{A}}_t$ is psd,
the spectrum of the residual can only be flat if $\widehat{\mtx{A}}_t$
already approximates the leading eigenspace of $\mtx{A}$.
The analysis is difficult because it depends on how the
spectrum of the matrix evolves.

The expected residual map $\mtx{\Phi}$ satisfies several remarkable features.
With respect to the semidefinite order, the function is both matrix monotone and
matrix concave.  These properties allow us to obtain bounds for
the trace of the residual after $k$ iterations:
\[
\Expect \big[ \trace(\mtx{A}_k) \condbar \mtx{A}_0 = \mtx{A} \big]
	\leq \trace\big[ \mtx{\Phi}^{\circ k}(\mtx{A}) \big].
\]
Here, $\circ k$ denotes the $k$-fold composition.
The right-hand side of this formula is a unitarily invariant function of $\mtx{A}$,
so we can assume that the input matrix $\mtx{A}$ is diagonal.
After some further argument, using the concavity of $\mtx{\Phi}$,
we recognize that the worst-case input matrix only has two distinct eigenvalues.
At this point, we can write down a dynamical system
for the eigenvalues of the residuals to understand how they evolve.
See \cite[Sec.~5]{CETW23:Randomly-Pivoted} for details. %

\subsubsection{Randomly pivoted Cholesky: Error bound}

Chen et al.~\cite{CETW23:Randomly-Pivoted} obtained a strong error
bound randomly pivoted Cholesky (\Cref{alg:rpc}).
In this section, we state their result and offer some discussion.

In the trace norm, the best rank-$r$ approximation of a psd matrix
satisfies the relation
\[
\min\nolimits_{\rank(\mtx{M}) \leq r} \norm{ \mtx{A} - \mtx{M} }_{*}
	= \sum_{j > r} \sigma_j(\mtx{A}).
\]
This point follows from a generalization of the Eckart--Young theorem~\eqref{eqn:eckart-young},
obtained by Mirsky~\cite{Mir60:Symmetric-Gauge}.
After drawing $k$ columns, randomly pivoted Cholesky
produces a rank-$k$ approximation $\widehat{\mtx{A}}_k$.
For a tolerance $\eps > 0$, we would like to know how many columns $k$
are sufficient to compete with the error in a best rank-$r$ approximation:
\begin{equation} \label{eqn:rpc-error}
\Expect \norm{ \mtx{A} - \widehat{\mtx{A}}_k }_*
	\leq (1 + \eps) \cdot \sum_{j > r} \sigma_j(\mtx{A}).
\end{equation}
We have the following statement.

\begin{theorem}[Randomly pivoted Cholesky; Chen et al.~2022] \label{thm:rpc}
Consider a psd matrix $\mtx{A} \in \Sym_n^+(\R)$.
Fix a comparison rank $r$ and a tolerance $\eps > 0$.
Randomly pivoted Cholesky (\Cref{alg:rpc}) produces an 
approximation $\widehat{\mtx{A}}_k$ that attains the error bound~\eqref{eqn:rpc-error}
after selecting $k$ columns where
\begin{equation} \label{eqn:rpc-cols}
k \geq \frac{r}{\eps} + r \cdot \log\left( \frac{1}{\eps \eta} \right)
\quad\text{and}\quad
\eta \coloneqq \frac{1}{\trace(\mtx{A})} \sum_{j > r} \sigma_j(\mtx{A}).
\end{equation}
The quantity $\eta$ is the error in a best rank-$r$ approximation of the input matrix,
relative to its trace.
\end{theorem}

For the worst-case matrix, the optimal Nystr\"om approximation requires
$k \geq r / \eps$ columns in order to achieve error
\[
\norm{ \mtx{A} - \widehat{\mtx{A}} }_{*} \leq (1 + \eps) \cdot \sum_{j > r} \sigma_j(\mtx{A}).
\]
Thus, the first term in~\eqref{eqn:rpc-cols} is necessary; see \cite{GS12:Optimal-Column} or \cite[Sect.~B.1]{CETW23:Randomly-Pivoted}.
The second term is proportional to the comparison rank $r$,
up to a logarithmic factor in the tolerance $\eps$ and
the relative approximation error $\eta$.
Practically speaking, we cannot hope to drive the relative error below
the machine precision (about $10^{-16}$ in double-precision arithmetic),
so the excess factor is at most $\log(10^{16}) \approx 37$.
Thus, randomly pivoted Cholesky not only has a favorable computational
profile, but it produces approximations that are competitive with
the best possible Nystr{\"o}m approximation.

\section{Randomized dimension reduction}

One approach for tackling a high-dimensional problem is to embed
it into a space of lower dimension, while approximately preserving
the geometry.  We may hope that the solution to the lower-dimensional
problem is close to the solution to the original problem.
At the same time, we aspire to solve the lower-dimensional
problem with fewer computational resources.  In matrix computations,
we can often construct an embedding that works for
a fixed problem instance by using randomness.

\begin{idea}[Randomized linear dimension reduction]
Use a random linear map to embed a high-dimensional linear
algebra problem into a lower-dimensional space, while approximately
preserving the geometry.
\end{idea}

Randomized linear dimension reduction is a particular example
of the sketching paradigm for computation~\cite{Mut05:Data-Streams}.
Many treatments~\cite{Woo14:Sketching-Tool,DM18:Lectures-Randomized}
of randomized matrix computation use this idea as an organizing theme.
In this section, we will introduce the method, and we will give several
applications where it plays a role.  Nevertheless, one of the purposes
of this short course is to demonstrate that there are many ways to use
probability for linear algebra computations, and it is facile to
suggest that the entire subject amounts to ``random sketching''.

\subsection{Embedding a linear subspace}

We will focus on dimension reduction for linear subspaces.
This primitive has a number of distinct computational applications,
and it is the most widely used sketching method in numerical linear algebra.
Our presentation is inspired by the paper~\cite{NT21:Fast-Accurate}.

\subsubsection{The subspace embedding property}

We begin with a definition, due to Sarl{\'o}s~\cite{Sar06:Improved-Approximation}.

\begin{definition}[Subspace embedding; Sarl{\'o}s 2006]
Let $\set{L} \subseteq \R^n$ be a \hilite{linear} subspace with dimension~$d$.
Consider a \hilite{linear} map $\mtx{\Phi} : \R^n \to \R^s$ with the property that
\begin{equation} \label{eqn:subspace-embed}
(1 - \eps) \cdot \norm{ \vct{x} }_2 \leq \norm{ \mtx{\Phi} \vct{x} }_2
	\leq (1 + \eps) \cdot \norm{ \vct{x} }_2
	\quad\text{for all $\vct{x} \in \set{L}$.}
\end{equation}
The map $\mtx{\Phi}$ is called a \term{subspace embedding} for $\set{L}$
with embedding dimension $s \in \N$ and distortion $\eps > 0$.
\end{definition}

A subspace embedding reduces the dimension of the ambient space $\R^n$,
but it preserves the norms of all vectors in the distinguished subspace $\set{L}$
up to a small distortion factor $(1 \pm \eps)$.
Typically, we assume that the subspace dimension is much smaller
than the ambient dimension: $d \ll n$.
Since the image $\mtx{\Phi}(\set{L})$ is a subspace in $\R^s$,
the embedding dimension must exceed the subspace dimension: $s \geq d$.
Still, we may hope that the embedding dimension is close to the subspace
dimension and much smaller than the ambient dimension: $s \approx d \ll n$.

Before turning to constructions of subspace embeddings,
we will explore some of the consequences of the embedding
property~\eqref{eqn:subspace-embed}.  After presenting
the constructions, we will describe a number of linear
algebra problems that we can solve efficiently using
a subspace embedding.

As an aside, you should be aware that there are other kinds
of dimension reduction primitives used in linear algebra, such as
affine embeddings~\cite{Woo20:Affine-Embeddings} and
projection-preserving sketches~\cite{MM20:Projection-Cost-Preserving-Sketches}.

\subsubsection{Metric geometry}

You may be familiar with the concept of an embedding of metric spaces,
a map from one metric space to another that approximately preserves
the distances between each pair of points.  We can view a subspace
embedding as a metric embedding of (a subspace of) a Euclidean space
into another Euclidean space.

Indeed, the subspace embedding property~\eqref{eqn:subspace-embed} guarantees that we can simultaneously
preserve the distance between each pair of points in a $d$-dimensional subspace $\set{L}$:
\[
(1 - \eps) \cdot \norm{ \vct{x} - \vct{y} }_2
	\leq \norm{ \mtx{\Phi}\vct{x} - \mtx{\Phi}\vct{y} }_2
	\leq (1 + \eps) \cdot \norm{ \vct{x} - \vct{y} }_2
	\quad\text{for all $\vct{x}, \vct{y} \in \set{L}$.}
\]
This consequence follows immediately from the linearity of the embedding $\mtx{\Phi}$.
It is also clear that we can achieve this goal with embedding dimension $s = d$, simply
by choosing a partial unitary matrix~$\mtx{\Phi}$ with $\mathrm{null}(\mtx{\Phi}) = \set{L}^{\perp}$.
By increasing the embedding dimension $s$ and allowing some distortion $\eps$,
we acquire flexibility that can facilitate the construction of subspace embeddings.

Subspace embeddings also preserve the inner produces between vectors in the subspace,
but this entails an additive error with a larger distortion.
Indeed, a messy calculation shows that
\begin{equation} \label{eqn:subspace-embed-ip}
\abs{ \ip{ \vct{x} }{ \vct{y} } - \ip{ \mtx{\Phi}\vct{x} }{ \mtx{\Phi}\vct{y} }}
	\leq \eps' \cdot \norm{ \vct{x} }_2 \norm{ \vct{y} }_2
	\quad\text{for all $\vct{x}, \vct{y} \in \set{L}$.}
\end{equation}
When $0 < \eps \leq 1$, we can take the adjusted distortion $\eps' \leq 10 \eps$.
This is not necessarily optimal.

\subsubsection{Spectral interpretation}

We can always represent the distinguished subspace as the range of a matrix.
This change of perspective allows us to develop a spectral interpretation
of the subspace embedding property.  This connection will help us construct
subspace embeddings and find applications to linear algebra problems.

\begin{proposition}[Subspace embedding: Spectral interpretation] \label{prop:subspace-embed-sval}
Suppose that $\set{L} = \range(\mtx{U})$ where $\mtx{U} \in \R^{n \times d}$
is a matrix with orthonormal columns.
For a matrix $\mtx{\Phi} \in \R^{s \times n}$,
the subspace embedding property~\eqref{eqn:subspace-embed} is equivalent
with the condition
\begin{equation} \label{eqn:subspace-embed-spectral}
1 - \eps \leq \sigma_{\min}(\mtx{\Phi} \mtx{U})
	\leq \sigma_{\max}(\mtx{\Phi} \mtx{U}) \leq 1 + \eps.
\end{equation}
\end{proposition}

\begin{proof}
By hypothesis, we can express the subspace in terms of the matrix:
\[
\set{L} = \{ \mtx{U} \vct{y} : \vct{y} \in \R^d \}.
\]
Using this parameterization, we can rewrite the embedding condition~\eqref{eqn:subspace-embed} as
\begin{equation} \label{eqn:subspace-embed-mtx}
(1 - \eps) \cdot \norm{ \mtx{U} \vct{y} }_2
	\leq \norm{ \mtx{\Phi} \mtx{U} \vct{y} }_2
	\leq (1 + \eps) \cdot \norm{ \mtx{U} \vct{y} }_2
	\quad\text{for all $\vct{y} \in \R^d$.}
\end{equation}
Since the matrix $\mtx{U}$ is orthonormal, $\norm{ \mtx{U}\vct{y} }_2 = \norm{ \vct{y} }_2$
for each vector $\vct{y} \in \R^d$.
Dividing through by $\norm{\vct{y}}$, we obtain the equivalent condition
\[
1 - \eps \leq \norm{ \mtx{\Phi} \mtx{U} \vct{z} }_2
	\leq 1 + \eps
	\quad\text{for each \hilite{unit vector} $\vct{z} \in \R^d$.}
\]
The variational definition of the maximum and minimum singular values ensures that this system
of inequalities is the same as the condition~\eqref{eqn:subspace-embed-spectral}.
\end{proof}

\subsection{Random subspace embeddings}
\label{sec:random-embed}

In many applications, it is imperative to construct a subspace embedding $\mtx{\Phi} : \R^n \to \R^s$
without using prior knowledge about the subspace $\set{L} \subseteq \R^n$.
These are called \hilite{oblivious} subspace embeddings.
We might want to construct an oblivious embedding because it is faster
to build the embedding without performing any adaptive computations on the subspace.
Another reason is that the subspace may not be available when the computation
is initiated.  For example, in streaming linear algebra problems, we may
need to maintain a summary of a subspace that evolves with time.

It requires some intelligence to construct an oblivious subspace embedding
that works for an arbitrary subspace.
Indeed, for any deterministically chosen matrix $\mtx{\Phi} \in \R^{s \times n}$
with $s < n$, we can find a pernicious subspace $\set{L}$ that defeats
the subspace embedding property~\eqref{eqn:subspace-embed}
by including a vector from the null space of $\mtx{\Phi}$.

By drawing a subspace embedding \hilite{at random},
we can ensure that the embedding property~\eqref{eqn:subspace-embed}
holds \hilite{with high probability}.  This is equivalent to using
a \hilite{random matrix} for the subspace embedding.  In view of
Proposition~\ref{prop:subspace-embed-sval}, we can verify the subspace
embedding property by studying the singular values of the
random matrix that models the interaction between the 
embedding and the subspace.

\subsubsection{Gaussian embeddings}

We begin with a demonstration that it is possible to construct an
oblivious subspace embedding for an arbitrary subspace.
To that end, fix a matrix $\mtx{U} \in \R^{n \times d}$ with orthonormal
columns.  We will describe a random matrix $\mtx{\Phi} \in \R^{s \times n}$
that serves as a subspace embedding for the range of $\mtx{U}$
with exceedingly high probability.

Let us consider the most basic random matrix model:
\[
\mtx{\Phi} \in \R^{s \times n}
\quad\text{with i.i.d.~$\normal(0, 1/s)$ entries.}
\]
The variance $1/s$ of the entries is just a normalization to ensure
that the matrix preserves the squared norm of a vector on average:
\begin{equation} \label{eqn:gauss-norm-pres}
\Expect \normsq{\mtx{\Phi} \vct{x}}_2 = \normsq{\vct{x}}_2
\quad\text{for each fixed $\vct{x} \in \R^n$.}
\end{equation}
We must show that the random matrix simultaneously preserves the
norms of all vectors in $\range(\mtx{U})$, up to a distortion factor.
For this argument, we take advantage of the orthogonal
invariance of a standard normal matrix.  Indeed,
\[
\mtx{\Phi} \mtx{U} \in \R^{s \times d}
\quad\text{has i.i.d.~$\normal(0, 1/s)$ entries.}
\]
This observation gives us access to tools for analyzing Gaussian matrices.

The subspace embedding property now follows from a classic result on the
singular values of a standard normal matrix~\cite[Thm.~II.13]{DS01:Local-Operator}:
\[
\Prob{ 1 - \sqrt{d/s} - t \leq \sigma_{\min}(\mtx{\Phi} \mtx{U})
	\leq \sigma_{\max}(\mtx{\Phi} \mtx{U})
	\leq 1 + \sqrt{d/s} + t } \geq 1 - \econst^{-st^2/2}.
\]
Moreover, according to the Bai--Yin law~\cite{BaiYin1988},
the extreme singular values converge to $1 \pm \sqrt{d/s}$
almost surely as $s, d \to \infty$ with the proportion $d/s$ fixed.

These facts allow us to obtain a nice scaling rule for the embedding dimension.
By the Bai--Yin law, we have the heuristic that
\begin{equation} \label{eqn:dim-red-scaling}
s \geq d / \eps^2
\quad\text{yields distortion $\eps$.}
\end{equation}
Alternatively, we may take  $s \geq 4d / \eps^2$
to guarantee failure probability below $\econst^{-d/2}$.
In other words, the embedding dimension $s$ only needs to be slightly
larger than the subspace dimension $d$ to achieve a modest distortion,
say, $\eps = 1/\sqrt{2}$.
On the other hand, to obtain a small distortion $\eps$, we must pay
an intolerable factor of $\eps^{-2}$ in the embedding dimension.
Thus, \hilite{you should only deploy randomized subspace embeddings in contexts
where you can accept a substantial distortion.}

The analysis for Gaussian subspace embeddings is simple and elegant.
Nevertheless, Gaussian matrices are often impractical %
because we pay $\mathcal{O}(sn)$ operations to apply the matrix to a single vector.
The cost of generating and storing the matrix is also unappealing.
To that end, we must consider subspace embeddings that are more
computationally friendly.

\subsubsection{Randomized trigonometric transforms}

We would like to design a subspace embedding that we can apply to a vector quickly.
This desideratum suggests that we should work with a linear map that
admits a fast matvec operation.  One familiar example
is an orthogonal discrete cosine transform (DCT2) matrix, which
we can apply with FFT-like algorithms.
We need to inject some additional randomness to obtain
the subspace embedding property for an arbitrary subspace.

Ailon \& Chazelle~\cite{AC09:Fast-Johnson-Lindenstrauss} proposed the following type of random matrix,
now called a \term{subsampled randomized trigonometric transform} (SRTT):
\begin{equation} \label{eqn:srtt}
\mtx{\Phi} \coloneqq \sqrt{\frac{n}{s}} \mtx{RFD} \in \R^{s \times n}
\quad\text{where}\quad
\begin{aligned}
&\text{$\mtx{R} \in \R^{s \times n}$ subsamples rows;} \\
&\text{$\mtx{F} \in \R^{n \times n}$ is a DCT2 matrix;} \\
&\text{$\mtx{D} \in \R^{n \times n}$ is random diagonal.}
\end{aligned}
\end{equation}
More precisely, $\mtx{R}$ is a uniformly random set of $s$ rows drawn from the identity matrix $\Id_n$,
and the random diagonal matrix $\mtx{D}$ has i.i.d.~$\uniform\{\pm 1\}$ entries.
It is not hard to check that the SRTT preserves the squared
norm of a vector on average, as in~\eqref{eqn:gauss-norm-pres}.

The cost of applying the SRTT to a vector is $\mathcal{O}(n \log n)$
operations using a standard fast DCT2 algorithm, and it can be reduced
to $\mathcal{O}(n \log d)$ with a more careful implementation.
Moreover, we can represent the SRTT matrix using only $s + n$ parameters. 
When $\log n \ll s$, we obtain a dramatic improvement over the $\mathcal{O}(sn)$
cost of storing and applying a Gaussian dimension reduction map.

What embedding dimension $s$ does the SRTT require?
The analysis~\cite{Tro11:Improved-Analysis} employs tools from matrix concentration
to address this question.
To obtain a subspace embedding with distortion $\eps = 0.6$,
it suffices to take the embedding dimension
\[
s \geq 4 \big[ \sqrt{d} + \sqrt{16 \log(dn)} \big]^2 \log d.
\]
The failure probability is $\mathcal{O}(1/d)$.
Note that the embedding dimension for the SRTT scales as $s \approx d \log d$,
rather than the $s \approx d$ scaling for the Gaussian map.
For the worst-case subspace, the extra logarithmic factor cannot be
avoided if we use the construction~\eqref{eqn:srtt} for the embedding map.
In practice, we often follow the prescription~\eqref{eqn:dim-red-scaling},
in spite of the possible failure modes.
Regardless, the computational benefits of the SRTT are worth the increase in the
size of the embedding.

\subsubsection{Sparse random matrices}

Another way to obtain a computationally efficient matvec
operation is to employ a \hilite{sparse} random matrix.
These matrices are simple, effective, and fast.
They are widely used in practice.
On the other hand, the analysis remains somewhat more
qualitative than the other examples.

The following construction was proposed by
Kane \& Nelson~\cite{KN14:Sparser-Johnson-Lindenstrauss}.
Consider a sparse random matrix of the form
\begin{equation} \label{eqn:sparse-map}
\mtx{\Phi} = \begin{bmatrix} \vct{\phi}_1 & \dots & \vct{\phi}_n \end{bmatrix} \in \R^{s \times n}
\quad\text{where $\vct{\phi}_i \in \R^s$ are i.i.d.~sparse vectors.}
\end{equation}
More precisely, each column $\vct{\phi}_i$ contains exactly $\zeta$ nonzero entries,
equally likely to be $\pm 1 / \sqrt{\zeta}$, in uniformly random positions.
It is not hard to check that the sparse map~\eqref{eqn:sparse-map} preserves the squared
norm of a vector on average, as in~\eqref{eqn:gauss-norm-pres}.
We can apply this matrix to a vector in $\mathcal{O}(\zeta n)$ operations.
The storage cost is at most $\zeta n$ parameters.
If $\zeta \ll s$, then we obtain a significant computational benefit.

Cohen~\cite{Coh16:Nearly-Tight-Oblivious} showed that the sparse map
serves as a subspace embedding with constant distortion, $\eps = \mathrm{const}$,
when the embedding dimension $s$ and the sparsity $\zeta$ satisfy
\[
s \geq \mathrm{Const} \cdot d \log d
\quad\text{and}\quad
\zeta \geq \mathrm{Const} \cdot \log d.
\]
The failure probability is constant.
Very recently, Chenakkod et al.~\cite{CDDR23:Optimal-Embedding}
showed that a variant of the distribution~\eqref{eqn:sparse-map}
serves as a subspace embedding with constant distortion, $\eps = \mathrm{const}$,
when
\[
s \geq \mathrm{Const} \cdot d
\quad\text{and}\quad
\zeta \geq \mathrm{Const} \cdot \log^4 d.
\]
The failure probability is again constant.  In both cases,
we can adjust the distortion and the failure probability,
but it complicates the statement of the bounds.

The existing theoretical results are not sufficiently precise
that we can use them to set algorithm parameters \lang{a priori}.
In practice~\cite{TYUC19:Streaming-Low-Rank}, we may follow the
prescription~\eqref{eqn:dim-red-scaling} and take $\zeta = 8$.

\subsubsection{Summary}

So far, we have not justified our interest in subspace embeddings.
We will remedy this shortcoming in the next several sections,
where we outline applications of these embeddings
in numerical linear algebra.

Before we continue, let us consolidate what we have learned.
For a fixed $d$-dimensional subspace $\set{L} \subseteq \R^n$,
we have seen that there are several ways to construct a \hilite{random}
matrix $\mtx{\Phi} \in \R^{s \times n}$ that serves as a subspace embedding
for $\set{L}$.  These constructions are successful with high probability.
Thus, we can draw a random matrix $\mtx{\Phi}$
and condition on the event that it is a subspace embedding for $\set{L}$.
Then the \hilite{deterministic} subspace embedding property~\eqref{eqn:subspace-embed}
is in force, and we can reason about its consequences.
\hilite{Probability only arises in the construction of the embedding.}

To embed a subspace with dimension $d$ and distortion $\eps$,
we typically choose the embedding dimension $s = d / \eps^2$.
For the SRTT map~\eqref{eqn:srtt} and the sparse map~\eqref{eqn:sparse-map},
the cost of applying the dimension reduction matrix to a vector in $\R^n$
is typically about $\mathcal{O}(n \log d)$ operations.
We will base operation counts on these heuristics.

\subsection{Approximate least-squares}

Let us return to the overdetermined least-squares problem
first posed in \Cref{sec:rand-kacz}.  We will argue that
subspace embeddings allow us to obtain a coarse solution
to the least-squares problem efficiently.  This approach
was developed by Sarl{\'o}s~\cite{Sar06:Improved-Approximation}.

\subsubsection{Sketch-and-solve for least-squares}

Consider the quadratic optimization problem
\begin{equation} \label{eqn:ls-for-sketch}
\minimize_{\vct{x} \in \R^d}\quad \tfrac{1}{2} \normsq{ \mtx{A}\vct{x} - \vct{b} }_2
\quad\text{with $\mtx{A} \in \R^{n \times d}$.}
\end{equation}
We focus on the case where $d \ll n$, while the matrix $\mtx{A}$
is dense and unstructured.
The cost of solving the problem with a direct method, such as \textsf{QR} factorization,
is $\mathcal{O}(d^2 n)$ operations.  We can obtain faster algorithms via dimension reduction.
Here is the procedure:

\begin{listbox}
\begin{enumerate}
\item	Construct a (random, fast) subspace embedding $\mtx{\Phi} \in \R^{s \times n}$
for $\range( [\mtx{A}, \vct{b}] )$.  

\item	Reduce the dimension of the problem data:
\[
\mtx{\Phi}\mtx{A} \in \R^{s \times d}
\quad\text{and}\quad
\mtx{\Phi}\vct{b} \in \R^{s}.
\]
This step is commonly referred to as \hilite{sketching}.

\item	Find a solution $\vct{x}_{\rm sk} \in \R^d$ to the sketched least-squares problem: %
\begin{equation} \label{eqn:sketched-ls}
\minimize_{\vct{x} \in \R^d}\  \tfrac{1}{2} \normsq{ \mtx{\Phi} (\mtx{A} \vct{x} - \vct{b}) }_2.
\end{equation}

\end{enumerate}
\end{listbox}

\noindent
See \Cref{alg:sketch-ls} for pseudocode that implements this procedure.
To achieve distortion $\eps$ for the $(d+1)$-dimensional subspace $\range([\mtx{A}, \vct{b}])$,
it suffices that $s = \mathcal{O}(d / \eps^2)$.
The typical cost of the sketching step is $\mathcal{O}(nd \log d)$ operations.
With a direct method, we can solve the sketched least-squares problem~\eqref{eqn:sketched-ls}
using $\mathcal{O}(d^2 s)$ operations.
We can try to use a solution $\vct{x}_{\mathrm{sk}}$ to the sketched
least-squares problem~\eqref{eqn:sketched-ls} as a \hilite{proxy}
for a solution $\vct{x}_{\star}$ to the original least-squares problem~\eqref{eqn:ls-for-sketch}.
This is called the \hilite{sketch-and-solve paradigm}.
The advantage is that we have reduced the operation
count to $\mathcal{O}(nd \log d + d^3 / \eps^2)$ operations,
as compared with the $\mathcal{O}(d^2 n)$ cost of a direct method
for the original problem.
When $\eps$ is modest and $\log d \ll d \ll n$,
we obtain a significant reduction in the arithmetic cost.

\begin{algorithm}[t]
\begin{algbox}[1]
\caption{\textit{Overdetermined least-squares: Simplest sketching method.}}
\label{alg:sketch-ls}

\Require	Matrix $\mtx{A} \in \R^{n \times d}$, response $\vct{b} \in \R^{n}$
\Ensure		Solution estimate $\vct{x}_{\rm sk} \in \R^d$

\vspace{5pt}
\hrule
\vspace{5pt}

\Function{SketchLeastSquares}{$\mtx{A}$, $\vct{b}$}

\State	Sample subspace embedding $\mtx{\Phi} \in \R^{s \times n}$ for $\range([\mtx{A}, \vct{b}])$
\Comment See~\Cref{sec:random-embed}

\State 	$\mtx{C} = \mtx{\Phi} \mtx{A}$ and $\vct{d} = \mtx{\Phi} \vct{b}$
\Comment Embed problem data

\State	$(\mtx{Q}, \mtx{R}) = \texttt{qr}(\mtx{C})$
\Comment Thin (pivoted) \textsf{QR} factorization

\State	$\vct{x}_{\rm sk} = \mtx{R}^{\dagger} (\mtx{Q}^* \vct{d})$
\Comment Solve reduced least-squares problem

\EndFunction
\end{algbox}
\end{algorithm}

\subsubsection{Sketch-and-solve: Analysis}

The analysis of \Cref{alg:sketch-ls} is straightforward.
We will compare the residual in the solution $\vct{x}_{\rm sk}$
to the sketched problem~\eqref{eqn:sketched-ls} with the optimal
residual for the original problem~\eqref{eqn:ls-for-sketch}
attained by $\vct{x}_{\star}$.

\begin{proposition}[Approximate least-squares; Sarl{\'o}s 2006]
Suppose that $\mtx{A} \in \R^{n \times d}$ is a tall matrix and $\vct{b} \in \R^n$.
Construct a subspace embedding $\mtx{\Phi} \in \R^{s \times n}$ for $\range([\mtx{A}, \vct{b}])$
with distortion $\eps$.
Let $\vct{x}_{\star} \in \R^d$ be a solution to the original least-squares problem~\eqref{eqn:ls-for-sketch},
and let $\vct{x}_{\rm sk} \in \R^d$ be a solution to the sketched problem~\eqref{eqn:sketched-ls}.  Then
\[
\norm{ \mtx{A} \vct{x}_{\rm sk} - \vct{b} }_2
		\leq \frac{1+\eps}{1-\eps} \cdot \norm{ \mtx{A} \vct{x}_{\star} - \vct{b} }_2.
\]	
\end{proposition}

\begin{proof}
We invoke the embedding property~\eqref{eqn:subspace-embed} twice:
\[
\begin{aligned}
\norm{ \mtx{A} \vct{x}_{\rm sk} - \vct{b} }_2
	&\leq (1 - \eps)^{-1} \cdot \norm{ \mtx{\Phi}( \mtx{A} \vct{x}_{\rm sk} - \vct{b} ) }_2 \\
	&\leq (1 - \eps)^{-1} \cdot \norm{ \mtx{\Phi}( \mtx{A} \vct{x}_{\star} - \vct{b} ) }_2
	\leq \frac{1+\eps}{1-\eps} \cdot \norm{ \mtx{A} \vct{x}_{\star} - \vct{b} }_2.
\end{aligned}
\]
The first inequality is the lower bound in the embedding property~\eqref{eqn:subspace-embed}.
The second inequality holds because $\vct{x}_{\rm sk}$ is the optimal solution to the
sketched least-squares problem~\eqref{eqn:sketched-ls}, so $\vct{x}_{\star}$ has a
larger residual.  Last, we invoke the upper bound in the embedding
property~\eqref{eqn:subspace-embed}.
\end{proof}

We have shown that the sketched solution $\vct{x}_{\rm sk}$ produces a residual \hilite{for the original least-squares problem}
that is competitive with the optimal residual.  In other words, reducing the dimension of the problem
only inflates the residual by a constant factor that depends on the distortion $\eps$ of
the embedding.  For instance, when $\eps = 1/2$, we may increase the residual by a factor of $3$.
As noted, it is not practical to make the distortion $\eps$ small because the computational costs
increase rapidly.

When is it acceptable to increase the residual by a factor of 3 or 5 or 10?
Suppose that the optimal residual is tiny, say, $\norm{ \mtx{A} \vct{x}_{\star} - \vct{b} }_2 \leq 10^{-8}$.
Then the sketch-and-solve method yields a residual $\norm{ \mtx{A} \vct{x}_{\rm sk} - \vct{b} }_2$ that is also tiny,
so the solution $\vct{x}_{\rm sk}$ to the sketched problem may be valuable to us.
In the next section, we will outline an application of this idea.
Nevertheless, it is important to keep in mind that a small residual \hilite{does not imply}
that $\vct{x}_{\rm sk}$ approximates an optimal point $\vct{x}_{\star}$.
See~\cite[Sec.~5.3]{GVL13:Matrix-Computations-4ed} for a discussion of
the sensitivity of least-squares problems to perturbations of the data.

For contrast, \hilite{sketch-and-solve is not an appropriate strategy for data fitting problems}.
In these situations, the residual in the original least-squares problem is usually large
because of noise in the observations.
Since the sketched solution increases the residual by a constant
factor, the sketched solution may not provide any useful information about the
original least-squares problem.

\subsubsection{Application: Subspace projection methods}

One attractive application of the sketch-and-solve paradigm is to accelerate subspace projection
methods for solving linear systems and eigenvalue problems~\cite{NT21:Fast-Accurate}.

Suppose that we wish to solve a nonsingular linear system $\mtx{A} \vct{x} = \vct{b}$
with $\mtx{A} \in \R^{n \times n}$.  We first construct a tall matrix $\mtx{K} \in \R^{n \times d}$ with $d \ll n$
whose range contains a good approximation to the solution of the linear system.
For example, $\range(\mtx{K})$ might be the Krylov subspace spanned by vectors
$(\mtx{A}^t \vct{v} : t = 0, 1, 2, \dots, d-1)$.
The subspace projection method passes from the original linear system to the
overdetermined least-squares problem
\begin{equation} \label{eqn:subspace-proj}
\minimize_{\vct{y} \in \R^{d}} \quad \tfrac{1}{2} \norm{ \mtx{AK} \vct{y} - \vct{b} }_2^2.
\end{equation}
The solution $\vct{y}_{\star}$ of the least-squares problem~\eqref{eqn:subspace-proj} induces an approximate solution
$\vct{x}_{\mtx{K}} = \mtx{K} \vct{y}_{\star}$ to the original linear system.

Since $\range(\mtx{K})$ captures approximate solutions to the linear system,
the least-squares problem~\eqref{eqn:subspace-proj} has a tiny residual.
Therefore, the subspace projection method %
is a good candidate for the sketch-and-solve approach.
For some problems, this technique can speed up the subspace projection method
by a factor of 100$\times$ or more.

\subsubsection{Iterative sketching}

We can improve the sketch-and-solve method by deploying it as
part of an iterative algorithm~\cite{GR15:Randomized-Iterative,PW15:Randomized-Sketches}.
Epperly has shown that this approach is forward stable~\cite{Epp23:Fast-Forward-Stable}.

Suppose that we wish to solve the overdetermined least-squares problem~\eqref{eqn:ls-for-sketch}.
First, construct a dimension reduction map $\mtx{\Phi} \in \R^{s \times n}$ for $\range([\mtx{A}, \vct{b}])$
with constant distortion, say, $\eps = 1/\sqrt{2}$.  Given an initial iterate $\vct{x}_0$,
we repeatedly solve the quadratic optimization problem
\[
\vct{x}_{t+1} \in \arg\min\big\{ \norm{ (\mtx{\Phi}\mtx{A}) \vct{y} }_2^2 + \vct{y}^* \mtx{A}^* (\mtx{A} \vct{x}_t - \vct{b}) : \vct{y} \in \R^d \big\}.
\]
This approach yields a sequence of iterates $(\vct{x}_t : t = 0,1,2,3,\dots)$ whose residuals
converge exponentially fast to the optimal value.  Using the sketch-and-solve solution as an initial iterate
($\vct{x}_0 = \vct{x}_{\rm sk}$), with about $\log(1/\alpha)$ iterations, we can achieve the error bound
\[
\norm{ \mtx{A} \vct{x}_{t} - \vct{b} }_2 \leq (1 + \alpha) \cdot \norm{ \mtx{A} \vct{x}_{\star} - \vct{b} }_2.
\]
The total cost is $\mathcal{O}(nd \log(d / \alpha) + d^3 \log d)$ arithmetic.

\subsection{Approximate orthogonalization}

As a second application of subspace embeddings, we consider the problem
of approximately orthogonalizing a system of high-dimensional vectors.

\begin{problem}[Approximate orthogonalization]
Consider a matrix $\mtx{A} \in \R^{n \times d}$ with full column rank.  The task is to find a
well-conditioned matrix $\mtx{B} \in \R^{n \times d}$ with $\range(\mtx{B}) = \range(\mtx{A})$.
\end{problem}

\noindent
A direct method for orthogonalizing the columns of the matrix $\mtx{A}$
requires $\mathcal{O}(nd^2)$ arithmetic.  In fact, a more serious problem
is that these methods require random access to the columns of the matrix,
which can be prohibitive in a parallel or distributed computing environment.

\subsubsection{Randomized Gram--Schmidt}

Balabanov \& Grigori~\cite{BG22:Randomized-Gram-Schmidt} observed that
we can address the approximate orthogonalization problem using a subspace embedding.
They were inspired by the fact~\eqref{eqn:subspace-embed-ip} that subspace embeddings
preserve angles.
Their approach can be used in conjunction with GMRES to help solve linear systems
or with Rayleigh--Ritz to help solve eigenvalue problems.

Here is the procedure, which is called  \hilite{randomized Gram--Schmidt}.

\begin{listbox}
\begin{enumerate}
\item	Construct a (random, fast) subspace embedding $\mtx{\Phi} \in \R^{s \times n}$
for $\range(\mtx{A})$.

\item	Sketch the problem data: $\mtx{\Phi}\mtx{A} \in \R^{s \times d}$.

\item	Compute a (thin, pivoted) \textsf{QR} factorization of the sketched data: $\mtx{\Phi} \mtx{A} = \mtx{QR}$.

\item	(Implicitly) define well-conditioned $\mtx{B} = \mtx{A} \mtx{R}^{-1}$ with $\range(\mtx{B}) = \range(\mtx{A})$.
\end{enumerate}
\end{listbox}

\noindent
\Cref{alg:rgs} provides pseudocode for this method. %
To achieve distortion $\eps$, it suffices to take $s = \mathcal{O}(d/\eps^2)$.
The sketching step typically requires $\mathcal{O}(nd \log d)$ operations.
To compute the factorization, we need $\mathcal{O}(sd^2)$ arithmetic.
For example, we can perform Gram--Schmidt on the columns;
more stable alternatives include Householder \textsf{QR}~\cite[]{GVL13:Matrix-Computations-4ed}.
If we wish to form the matrix $\mtx{B}$ explicitly, we must spend $\mathcal{O}(nd^2)$ operations.

The advantage of this method is that the sketched matrix $\mtx{\Phi}\mtx{A}$
can be computed one column at a time, and it may be small enough to store in
working memory.  In this case, we can also orthogonalize the columns of the sketched
matrix in working memory.  The coefficient matrix $\mtx{R}$
obtained from the sketched data still serves to \hilite{approximately}
orthogonalize the columns of the original data matrix.
In contrast, if we orthogonalize the columns
of the original matrix directly, then we may incur high communication
costs when rearranging and combining the columns. %

\begin{algorithm}[t]
\begin{algbox}[1]
\caption{\textit{Approximate orthogonalization.}}
\label{alg:rgs}

\Require	Matrix $\mtx{A} \in \R^{n \times d}$ with full column rank
\Ensure		Well-conditioned matrix $\mtx{B} \in \R^{n \times d}$ with $\range(\mtx{B}) = \range(\mtx{A})$

\vspace{5pt}
\hrule
\vspace{5pt}

\Function{RandomGramSchmidt}{$\mtx{A}$, $\vct{b}$}

\State	Sample subspace embedding $\mtx{\Phi} \in \R^{s \times n}$ for $\range(\mtx{A})$
\Comment See~\Cref{sec:random-embed}

\State 	$\mtx{C} = \mtx{\Phi} \mtx{A}$ %
\Comment Embed problem data

\State	$(\mtx{Q}, \mtx{R}) = \texttt{qr}(\mtx{C})$
\Comment Thin (pivoted) \textsf{QR} factorization

\State	$\mtx{B} = \mtx{A}\mtx{R}^{-1}$
\Comment Implicit representation for solution
\EndFunction
\end{algbox}
\end{algorithm}

\subsubsection{Approximate orthogonalization: Analysis}

The analysis of the approximate orthogonalization method
depends on a result of Rokhlin \& Tygert~\cite{RT08:Fast-Randomized}.
See \Cref{sec:rand-precond-rt} for the original context. %

\begin{proposition}[Whitening; Rokhlin \& Tygert 2008] \label{prop:whiten}
Let $\mtx{A} \in \R^{n \times d}$ be a tall matrix with full column rank.
Construct a subspace embedding $\mtx{\Phi} \in \R^{s \times d}$
for $\range(\mtx{A})$ with distortion $\eps$.  Form
a \textsf{QR} factorization of the sketched matrix:
\[
\mtx{\Phi}\mtx{A} = \mtx{QR}
\quad\text{with $\mtx{R} \in \R^{d \times d}$.}
\]
Then $\mtx{R}$ has full rank, and the whitened matrix $\mtx{B} = \mtx{A}\mtx{R}^{-1}$ satisfies
\[
\frac{1}{1+\eps} \leq \sigma_{\min}(\mtx{B}) \leq \sigma_{\max}(\mtx{B}) \leq \frac{1}{1-\eps}.
\]
\end{proposition}

\Cref{prop:whiten} states that the whitened matrix $\mtx{B}$ has
the same range as the input matrix $\mtx{A}$, but the spectral condition number
$\kappa(\mtx{B}) \leq (1 + \eps) / (1 - \eps)$.  In other words, the
columns of $\mtx{B}$ are roughly orthogonal.

\begin{proof}
Since $\mtx{\Phi}$ is a subspace embedding for the $d$-dimensional subspace $\range(\mtx{A})$,
the range of the sketched matrix $\mtx{\Phi} \mtx{A}$ also has dimension $d$.
Thus, the $d \times d$ triangular factor $\mtx{R}$ must have full rank.

Since $\mtx{Q}$ has orthonormal columns, by the Pythagorean theorem, we have the identities
\[
\norm{ \mtx{R} \vct{y} }_2 = \norm{ \mtx{QR} \vct{y} }_2 = \norm{ \mtx{\Phi}\mtx{A} \vct{y} }_2
\quad\text{for all $\vct{y} \in \R^d$.}
\]
By the subspace embedding property~\eqref{eqn:subspace-embed},
\[
(1 - \eps) \cdot \norm{\mtx{A} \vct{y} } \leq \norm{ \mtx{R}\vct{y} }_2 \leq (1+\eps) \cdot \norm{\mtx{A} \vct{y} }_2
\quad\text{for all $\vct{y} \in \R^d$.}
\]
Make the change of variables $\vct{y} = \mtx{R}^{-1} \vct{x}$ to obtain
\[
(1 - \eps) \cdot \norm{ \mtx{A} \mtx{R}^{-1} \vct{x} }_2 \leq \norm{ \vct{x} }_2 \leq (1 + \eps) \cdot \norm{ \mtx{A} \mtx{R}^{-1} \vct{x} }_2
\quad\text{for all $\vct{x} \in \R^d$.}
\]
Divide by $\norm{ \vct{x} }_2$ and use the variational definition
of singular values to reach the stated bounds.
\end{proof}

\subsection{Approximate null space}

As a third application, we show that subspace embeddings allow us to compute
a trailing right singular subspace of a tall matrix quickly.
Our presentation is inspired by the papers~\cite{GPW12:Sketched-SVD,PN23:Fast-Randomized}.

\begin{problem}[Approximate null space]
Consider a tall matrix $\mtx{A} \in \R^{n \times d}$.  The task is to
find an orthonormal matrix $\mtx{W} \in \R^{d \times k}$ whose range
aligns with the $k$ trailing right singular vectors of $\mtx{A}$.
\end{problem}

The approximate null space problem supports a number of important applications.
It is the core computation required for total least-squares~\cite{GVL80:Analysis-Total},
which is also known as errors-in-variables regression in statistics.  In signal processing, some
spectrum estimation algorithms, such as MUSIC~\cite{Sch86:Multiple-Emitter}, require an approximate null space.
In approximation theory, the AAA algorithm~\cite{NST18:AAA-Algorithm} for rational approximation also
requires minimum singular vectors of tall matrices.

\subsubsection{Null space from approximate right singular subspaces}

Following a well-oiled track, we apply a subspace embedding to
the input matrix and compute its SVD.  Then we extract the
trailing right singular vectors to solve the approximate null space problem.
The reason this approach works is that the SVD of the sketched
matrix retains a lot of geometric information from the input matrix,
much like the \textsf{QR} factorization of the sketched matrix.
This idea is due to Gilbert et al.~\cite{GPW12:Sketched-SVD}.

\begin{listbox}
\begin{enumerate}
\item	Construct a (random, fast) subspace embedding $\mtx{\Phi} \in \R^{s \times n}$
for $\range(\mtx{A})$.

\item	Sketch the problem data: $\mtx{\Phi} \mtx{A} \in \R^{s \times d}$.

\item	Compute a full SVD of the sketched matrix: $\mtx{\Phi} \mtx{A} = \mtx{U} \mtx{\Sigma} \mtx{V}^*$.

\item	Report the trailing right singular vectors $\mtx{W} = \mtx{V}( : , (d-k+1):d ) \in \R^{d \times k}$.
\end{enumerate}
\end{listbox}

\noindent
Pseudocode for this procedure appears in \Cref{alg:sketch-svd}.  As usual, we choose the embedding dimension
$s = \mathcal{O}(d/\eps^2)$ to achieve distortion $\eps$.  The sketching step typically takes
$\mathcal{O}(nd \log d)$ operations, and the SVD step takes $\mathcal{O}(sd^2)$ operations.
The total is $\mathcal{O}(nd \log d + d^3 / \eps^2)$.
For comparison, a full SVD of the input matrix requires $\mathcal{O}(d^2 n)$ arithmetic.

\begin{algorithm}[t]
\begin{algbox}[1]
\caption{\textit{Approximate null space.}}
\label{alg:sketch-svd}

\Require	Matrix $\mtx{A} \in \R^{n \times d}$, dimension $k$ for approximate null space
\Ensure		Orthonormal matrix $\mtx{W} \in \R^{d \times k}$ whose range spans approximate null space 

\vspace{5pt}
\hrule
\vspace{5pt}

\Function{ApproxNullSpace}{$\mtx{A}$, $k$}

\State	Sample subspace embedding $\mtx{\Phi} \in \R^{s \times n}$ for $\range(\mtx{A})$
\Comment See~\Cref{sec:random-embed}

\State 	$\mtx{C} = \mtx{\Phi} \mtx{A}$ %
\Comment Embed problem data

\State	$(\mtx{U}, \mtx{\Sigma}, \mtx{V}^*) = \texttt{svd}(\mtx{C})$
\Comment Economy SVD

\State	$\mtx{W} = \mtx{V}(:, (d-k+1):d)$
\Comment Approximate $k$-dimensional null space
\EndFunction
\end{algbox}
\end{algorithm}

\subsubsection{Approximate null space: Analysis}

As usual, let $\mtx{A} \in \R^{n \times d}$ be a tall matrix,
perhaps rank deficient.
Here is a variational formulation of the null space problem: %
\begin{equation} \label{eqn:approx-null}
\minimize_{\mtx{V} \in \R^{d \times k}}\quad \fnormsq{ \mtx{A} \mtx{V} } \quad\subjto\quad \mtx{V}^* \mtx{V} = \Id_k.
\end{equation}
The solution to this problem is a subspace spanned by $k$ trailing
right singular vectors, and the optimal value is the energy in the
trailing $k$ singular values:
\begin{equation*} \label{eqn:approx-null-value}
\min\nolimits_{\mtx{V}^* \mtx{V} = \Id_k}\ \fnormsq{\mtx{A}\mtx{V}} = \sum_{j > d - k} \sigma_j^2(\mtx{A}).
\end{equation*}
If the $k$ smallest singular values are very small (or zero), then a subspace that
carries a similar amount of energy may serve as an approximate null space.

Using the subspace embedding property, we can easily verify that the sketching
method leads to a good solution to the approximate null space problem~\cite{PN23:Fast-Randomized}.

\begin{proposition}[Approximate null space; Park \& Nakatsukasa 2023]
Let $\mtx{A} \in \R^{n \times d}$ be a tall matrix, and let $\mtx{\Phi} \in \R^{s \times d}$
be a subspace embedding for $\range(\mtx{A})$ with distortion $\eps$.  \Cref{alg:sketch-svd} produces
an approximate null space $\mtx{W} \in \R^{d \times k}$ that is competitive with
the best $k$-dimensional null space:
\[
\fnormsq{ \mtx{A} \mtx{W} } \leq \frac{1+\eps}{1-\eps} \cdot
\min\nolimits_{\mtx{V}^* \mtx{V} = \Id_k}\ \fnormsq{ \mtx{A} \mtx{V} }.
\]
In particular, if $\mtx{A}\mtx{V} = \mtx{0}$ for some $k$-dimensional subspace $\mtx{V}$,
then $\mtx{A}\mtx{W} = \mtx{0}$.
\end{proposition}

\begin{proof}
Fix an orthonormal matrix $\mtx{V}_{\star} \in \R^{d \times k}$ that solves
the null space problem~\eqref{eqn:approx-null}.
Since $\mtx{\Phi}$ is a subspace embedding for $\range(\mtx{A})$,
\[
\fnormsq{ \mtx{AW} } \leq \frac{1}{1-\eps} \cdot \fnormsq{ \mtx{\Phi}\mtx{A}\mtx{W} }
	\leq \frac{1}{1-\eps} \cdot \fnormsq{ \mtx{\Phi}\mtx{A}\mtx{V}_{\star} }
	\leq \frac{1+\eps}{1-\eps} \cdot \fnormsq{ \mtx{A} \mtx{V}_{\star} }.
\]
The first inequality is the lower subspace embedding property~\eqref{eqn:subspace-embed},
applied to each column of the matrix $\mtx{W}$.  The second inequality holds because $\mtx{W}$ solves
the null space problem~\eqref{eqn:approx-null} for the sketched matrix $\mtx{\Phi}\mtx{A}$,
so the matrix $\mtx{V}_{\star}$ must increase the objective.  Last, we apply
the upper bound from the subspace embedding property.
\end{proof}

\subsubsection{Approximate SVD: Analysis}

In fact, we can say quite a lot more about the SVD of the sketched matrix
that arises in the approximate null space algorithm.  This result is
essentially due to Gilbert et al.~\cite{GPW12:Sketched-SVD}.  The proof
is adapted from Park \& Nakatsukasa~\cite{PN23:Fast-Randomized}.

\begin{proposition}[Approximate singular values; Gilbert et al.~2012]
Let $\mtx{A} \in \R^{n \times d}$ be a tall matrix, and let $\mtx{\Phi} \in \R^{s \times d}$
be a subspace embedding for $\range(\mtx{A})$ with distortion $\eps$.
The singular values of the sketched matrix $\mtx{\Phi}\mtx{A}$ satisfy
\[
(1-\eps) \cdot \sigma_i(\mtx{A})
	\leq \sigma_i(\mtx{\Phi}\mtx{A}) \leq (1+\eps) \cdot \sigma_i(\mtx{A})
	\quad\text{for $i = 1, \dots, d$.}
\]
\end{proposition}

\begin{proof}
Let $\mtx{A} = \mtx{U} \mtx{\Sigma} \mtx{V}^*$ be an SVD of the input matrix.
Then $\mtx{\Phi}$ is a subspace embedding for $\range(\mtx{U})$.
For each index $i = 1, \dots, d$, by rotational invariance of singular values,
\[
\sigma_i(\mtx{\Phi}\mtx{A}) = \sigma_i(\mtx{\Phi} (\mtx{U} \mtx{\Sigma} \mtx{V}^*))
	= \sigma_i((\mtx{\Phi}\mtx{U}) \mtx{\Sigma})
\]
By Ostrowski's relative perturbation theorem~\cite{Ost59:Quantitative-Formulation},
\[
\sigma_{d}(\mtx{\Phi} \mtx{U}) \cdot \sigma_i(\mtx{\Sigma})
	\leq \sigma_i(\mtx{\Phi} \mtx{A}) \leq \sigma_{1}(\mtx{\Phi} \mtx{U}) \cdot \sigma_i(\mtx{\Sigma}).
\]
By the subspace embedding property~\eqref{eqn:subspace-embed},
\[
(1 - \eps) \cdot \sigma_{i}(\mtx{A}) \leq \sigma_i(\mtx{\Phi} \mtx{A})
	\leq (1 + \eps) \cdot \sigma_i(\mtx{A}).
\]
This is the advertised result.
\end{proof}

It is also possible to argue that the right singular subspaces of the sketched
matrix $\mtx{\Phi}\mtx{A}$ align with the right singular subspace of the original
matrix $\mtx{A}$, provided that there are gaps between the singular values.
These results are more technical, so we refer to~\cite{GPW12:Sketched-SVD,PN23:Fast-Randomized}
for the details.

\section{More themes}

We have studied several templates for designing
randomized algorithms for matrix computations.
Reviewing the literature in this field, it is striking how
many other strategies fall outside these categories.
In this section, we give a brief introduction
to some other algorithmic approaches that have
a probabilistic component:
randomized approximation for preconditioning,
placing problem instances in general position,
and smoothed analysis of problem perturbations.

\subsection{Randomized preconditioning}

Many linear algebra problems have associated \hilite{condition numbers} that
reflect the sensitivity of problem solutions and may impact the
difficulty of solving the problem numerically.  For example, the
linear system condition number describes the stability of the
solution to a linear system~\cite[Sec.~2.6.2]{GVL13:Matrix-Computations-4ed}.
For a positive-definite linear system, the cost of the conjugate gradient algorithm can be
bounded in terms of the linear system condition number~\cite[Sec.~10]{She94:CG-Agonizing}.

The \hilite{preconditioning} strategy replaces a badly conditioned
computational problem with an equivalent problem that has a smaller condition
number.  For linear systems and least-squares, preconditioning
schemes typically multiply the input matrix with another matrix
that is both easy to invert and that also reduces the condition number.

We have already encountered many algorithms for inexpensively
computing a structured approximation of a matrix.
Some of these approximations can be used as preconditioners
for solving linear algebra problems involving the matrix.

\begin{idea}[Randomized preconditioning]
Given an input matrix, use a randomized algorithm to form
a structured approximation that serves as a preconditioner.
\end{idea}

In this section, we summarize a few examples where
randomized methods lead to preconditioners for
least-squares problems and regularized linear systems.
The key benefit of randomized algorithms for preconditioning
is that they can provide excellent guarantees while
operating more efficiently than deterministic techniques.

\subsubsection{Example: Overdetermined least-squares}
\label{sec:rand-precond-rt}

Let $\mtx{A} \in \R^{n \times d}$ be a tall matrix with full rank,
and let $\vct{b} \in \R^n$ be a response vector.
Recall the overdetermined least-squares problem
\[
\minimize_{\vct{x} \in \R^d}\quad \tfrac{1}{2} \normsq{ \mtx{A}\vct{x} - \vct{b} }_2.
\]
Rokhlin \& Tygert~\cite{RT08:Fast-Randomized} observed that we can
use sketching to design a preconditioner. %

As usual, let $\mtx{\Phi} \in \R^{s \times n}$ be a (random, fast)
subspace embedding for $\range(\mtx{A})$.  We can form
a \textsf{QR} factorization of the sketched matrix:
\[
\mtx{\Phi}\mtx{A} = \mtx{QR}.
\]
The cost of these steps can be much smaller than solving the original least-squares problem.
Proposition~\ref{prop:whiten} states that $\mtx{AR}^{-1}$
is a very well-conditioned matrix with the same range as $\mtx{A}$.
We can solve a triangular linear systems with $\mtx{R}$ efficiently,
so it serves as an excellent preconditioner.

Thus, we pass to the equivalent least-squares problem
\[
\minimize_{\vct{y} \in \R^d}\ \tfrac{1}{2} \normsq{ (\mtx{A}\mtx{R}^{-1}) \mtx{y} - \vct{b} }_2
\quad\text{where}\quad
\vct{x} = \mtx{R}^{-1} \vct{y}.
\]
We can solve this problem efficiently using the LSQR algorithm~\cite[Sec.~11.4.2]{GVL13:Matrix-Computations-4ed}.
Assuming exact arithmetic, the Rokhlin \& Tygert algorithm can provably solve a least-squares problem
to high accuracy faster than any previously known algorithm.
On the other hand, there remain unresolved questions about
the numerical stability of this approach~\cite{meier2023sketch}.

\subsubsection{Example: Randomly pivoted Cholesky}

In a similar vein, we can use randomized Nystr{\"o}m approximations
to design a preconditioner for a linear system with a psd
matrix~\cite{FTU23:Randomized-Nystrom}.
This section outlines a related method that uses randomly
pivoted Cholesky to accelerate kernel computations~\cite{DEF+23:Robust-Randomized}.

Let $\mtx{A} \in \Sym_n^+(\R)$ be a psd matrix where we pay for entry evaluations,
such as a kernel matrix.  Suppose that we wish to solve the regularized linear system
\begin{equation} \label{eq:regularized-linear-system}
(\mtx{A} + \mu \Id) \vct{x} = \vct{b}
\quad\text{for $\mu > 0$.}
\end{equation}
This problem often arises in kernel methods, %
where it is known as \hilite{kernel ridge regression}.

In \Cref{sec:rpc}, we saw that the randomly pivoted Cholesky algorithm
can be used to compute a near-optimal psd approximation of the matrix $\mtx{A}$.
More precisely, we can obtain a rank-$k$ psd matrix~\smash{$\widehat{\mtx{A}}_k$},
represented as an eigenvalue decomposition.
The cost of producing the approximation is $\mathcal{O}(kn)$ entry
evaluations and $\mathcal{O}(k^2 n)$ arithmetic.

This approximation yields a preconditioner~\cite{DEF+23:Robust-Randomized}
for solving the regularized system:
\[
(\widehat{\mtx{A}}_k + \mu \Id)^{-1} (\mtx{A} + \mu \Id) \vct{x} = (\widehat{\mtx{A}}_k + \mu \Id)^{-1} \vct{b}.
\]
We can solve linear systems with the preconditioner efficiently,
after just $\mathcal{O}(kn)$ operations.
Therefore, we can apply preconditioned conjugate
gradient~\cite[Sec.~11.5.2]{GVL13:Matrix-Computations-4ed}
to quickly solve the regularized linear system\eedit{}{\ \eqref{eq:regularized-linear-system}}. 
See \cite[Thm.~2.2]{DEF+23:Robust-Randomized} for
an analysis and details on the runtime.

Other authors~\cite{ACW17:Faster-Kernel} have proposed preconditioners
based on random features approximations.
The concept is similar but this approach is less competitive in practice.

\subsubsection{Example: Sparse Cholesky for graph Laplacians}
\label{sec:kyng}

We conclude with a spectacular example of the randomized preconditioning paradigm,
due to Kyng \& Sachdeva~\cite{KS16:Approximate-Gaussian}. 
See~\cite{Tro19:Matrix-Concentration-LN} for another treatment
of this result.  The recent paper~\cite{GKS23:Robust-Laplacian}
describes an implementation and empirical evaluation of a closely
related algorithm.

Suppose that $\mtx{A} \in \Sym_n^+(\R)$ is the graph Laplacian
of a combinatorial graph with $n$ vertices and $m$ edges.
Teng~\cite{Ten10:Laplacian-Paradigm} points out that
many applications in data analysis and computer algorithms
require the solution to linear systems in this Laplacian matrix: 
\[
\mtx{A} \vct{x} = \vct{b}
\quad\text{where}\quad
\vct{1}^* \vct{b} = 0
\quad\text{and}\quad
\vct{1}^* \vct{x} = 0.
\]
We write $\vct{1} \in \R^n$ for the all-ones vector.
Nominally, this problem costs $\mathcal{O}(n^3)$ operations
to solve with a direct method.  Classical iterative methods, such as algebraic multigrid,
may not be reliable for complicated graphs.

Kyng \& Sachdeva devised a randomized algorithm that produces
an approximate Cholesky factor $\mtx{L}$ of the Laplacian matrix:
\[
0.5 \, \mtx{A} \psdle \mtx{LL}^* \psdle 1.5 \, \mtx{A},
\]
where $\psdle$ denotes the semidefinite order and the matrix $\mtx{L}$
is lower triangular and sparse.  Their algorithm is related to the classic full Cholesky decomposition algorithm,
but it repeatedly uses matrix Monte Carlo to sparsify the rank-one
correction.  Their method produces a lower triangular factor $\mtx{L}$
with just $\mathcal{O}( m \log n )$ nonzero entries
after $\mathcal{O}(m \log^3 n)$ arithmetic operations.

The approximate Cholesky factor can be used as a preconditioner
for conjugate gradient to solve the linear system:
\[
(\mtx{L}^{\dagger} \mtx{A} \mtx{L}^{*\dagger}) \vct{y}
	= \mtx{L}^{\dagger} \vct{b}
	\quad\text{where}\quad
	\mtx{x} = \mtx{L}^{* \dagger} \vct{y}.
\]
The dagger ${}^\dagger$ denotes the pseudoinverse.
Indeed, we can solve the triangular system in $\mtx{L}$ with just
$\mathcal{O}(m \log n)$ operations, and it reduces the condition
number of the problem to a constant.  Therefore, we can find a
solution with $\eps$ relative error after about $\log(1/\eps)$
iterations.

The total operation count for the approximate factorization and solving
the linear system is
\[
\mathcal{O}(m \log^2(n) (\log(n) + \log(1/\eps))).
\]
Recall that the direct method may cost fully $\mathcal{O}(n^3)$ operations.
For a sparse graph, $m \approx n$, and we achieve
a factor $n^2$ improvement over the direct method.
For a dense graph, $m \approx n^2$, and we still achieve
a factor $n$ improvement over the direct method.
This is a remarkable outcome.

\subsection{General position} 

Many deterministic algorithms exhibit failure modes that can arise
from unfortunate choices of the problem instance.
For example, Gaussian elimination can encounter zero pivots or
very small pivots that render it numerically unstable,
but experience suggests that ``most'' input matrices
do not cause problems.

One way to deal with the possibility of a challenging problem
instance is to randomly transform the instance into an equivalent
instance that avoids the failure modes with high probability.
In other words, we can try to place the instance in
\hilite{general position}.
We are using this term loosely, as the meaning depends
on the computational problem and the algorithm.

\begin{idea}[General position]
Randomly transform a problem instance to an equivalent instance
where the computational method is likely to succeed.
\end{idea}

We will summarize two situations where this is possible.
First, we discuss how random rotations of a linear system
can ensure that Gaussian elimination succeeds.
Second, we show how to simultaneously diagonalize
a pair of matrices by combining them randomly.

\subsubsection{Example: Random rotations for Gaussian elimination}
\label{sec:rand-rot}

In this section, we present one of the earliest randomized linear algebra methods,
following the original report~\cite{Par95:Randomizing-Butterfly} by D.~Stott Parker.

Suppose we wish to solve the (square, nonsingular) linear system $\mtx{A} \vct{x} = \vct{b}$
where $\mtx{A} \in \R^{n \times n}$.  To address this problem, the classic algorithm starts by
applying pivoted Gaussian elimination to the matrix~$\mtx{A}$ to obtain a pivoted \textsf{LU}
factorization.  Pivoting is the process of selecting which row to eliminate at each step.
For a general matrix, pivoting may be necessary because we cannot eliminate a row that contains a zero in the diagonal
position.  %
Unfortunately, the need to pivot conflicts with other goals, such as blocking computations
or minimizing communication costs.

Parker observed that the \hilite{degenerate} matrices where zero pivots arise
compose a set with Lebesgue measure zero.
He proposed that we might avoid these degenerate
examples by transforming the input matrix.  Let $\mtx{U}, \mtx{V} \in \M_n(\R)$
be orthogonal matrices, drawn uniformly at random.  We can pass from the original
linear system to the equivalent problem
\[
(\mtx{UAV}) \vct{y} = \mtx{U} \vct{b}
\quad\text{and}\quad
\vct{x} = \mtx{V}\vct{y}.
\]
Parker proved that the transformed matrix ${\mtx{A}}_{\rm rot} = \mtx{UAV}$
is more favorable for Gaussian elimination.  \hilite{With probability one,
we can perform Gaussian elimination without pivoting
on ${\mtx{A}}_{\rm rot}$ and never encounter a zero diagonal entry.}
We have placed the problem in general position.

What is the cost of this approach?
With random orthogonal matrices, the transformation requires
$\mathcal{O}(n^3)$ operations for two matrix--matrix multiplications.
Recall that Gaussian elimination (with or without pivoting)
demands $\mathcal{O}(n^3)$ operations, so the preprocessing
does not increase the asymptotic cost of the algorithm.
Meanwhile, by avoiding pivoting in the Gaussian elimination step,
we can design implementations that exploit blocking
and parallelism.
Variants of Parker's proposal~\cite{BDHT13:Accelerating-Linear-System}
have proven effective in practice.

Related ideas appear in work of Demmel et al.~\cite{DDH07:Fast-Linear-Algebra}.
They propose the ``randomized \textsf{URV} algorithm'' for rank-revealing factorization,
and they establish more precise bounds on the stability of the resulting
numerical algorithms.

\subsubsection{Example: Randomized joint diagonalization}

Consider two symmetric matrices $\mtx{A}, \mtx{B} \in \Sym_n(\R)$ that commute with each other:
$\mtx{AB} = \mtx{BA}$.  As you know, commuting symmetric matrices are simultaneously diagonalizable.
In other words, there always exists an orthogonal matrix $\mtx{Q} \in \M_n(\R)$
with the property that $\mtx{Q}^* \mtx{A} \mtx{Q}$ and $\mtx{Q}^* \mtx{B} \mtx{Q}$ are both diagonal matrices.
Our goal is to compute such an orthogonal matrix.

This task is delicate.  An orthogonal matrix that diagonalizes $\mtx{A}$ does not always diagonalize~$\mtx{B}$,
so we have to understand something about how the matrices interact.
Moreover, in applications, the matrices may not commute exactly
because they contain numerical errors or derive from noisy data.

He \& Kressner~\cite{HK22:Randomized-Joint} have studied an elegant %
strategy that can address both these challenges; see their paper for citations
to the original sources.
They form a \hilite{random linear combination} of the two input matrices:
\[
\mtx{C} = \gamma_1 \mtx{A} + \gamma_2 \mtx{B}
\quad\text{where $\gamma_1, \gamma_2 \sim \normal(0,1)$ i.i.d.}
\]
In other words, $\mtx{C}$ is a \hilite{generic point} in the plane spanned
by the two matrices $\mtx{A}$ and $\mtx{B}$. 
If the matrices $\mtx{A}$ and $\mtx{B}$ commute with each other,
then $\mtx{C}$ clearly commutes with both $\mtx{A}$ and $\mtx{B}$.
To solve the simultaneous diagonalization problem, we can simply
extract the orthogonal matrix $\mtx{Q}$ from an eigenvalue decomposition
$\mtx{C} = \mtx{Q} \mtx{\Lambda} \mtx{Q}^*$.

He \& Kressner analyzed the case where the matrices commute exactly,
as well as the situation where they only commute approximately.
\hilite{If $\mtx{A}$ and $\mtx{B}$ commute, then $\mtx{Q}$ simultaneously diagonalizes the pair $(\mtx{A}, \mtx{B})$
with probability one.}
Moreover, if $\mtx{A}$ and $\mtx{B}$ nearly commute, then $\mtx{C}$ still produces
an orthogonal matrix $\mtx{Q}$ that approximately diagonalizes the pair $(\mtx{A}, \mtx{B})$.

\subsection{Smoothed analysis}
\label{sec:smoothed-analysis}

As we have discussed, deterministic algorithms can exhibit failure modes
for malign problem instances.  For some algorithms, the challenging
examples are unusual in the sense that most nearby problem instances
are easier to solve.  Spielman \& Teng~\cite{ST02:Smoothed-Analysis}
developed this perspective, called \hilite{smoothed analysis},
and they showed that it sheds light on the performance of several basic
algorithms in linear algebra and optimization.
Their insight immediately suggests a computational strategy.
By randomly perturbing a hard problem instance,
we can find a nearby instance where the algorithm is likely to succeed.

\begin{idea}[Smoothed analysis]
Apply a random perturbation to a problem instance to obtain a nearby instance
that is more computationally favorable.
\end{idea}

\noindent
Conceptually, this method is distinct from the preconditioning strategy
or the general position strategy because it results in an inequivalent
problem instance that has a different solution.  While algorithms based
on smoothed analysis often produce a high-quality solution to a nearby
problem instance, the sleight of hand will discomfit many numerical
analysts.

In this section, we treat two dramatically different applications of this idea:
Gaussian elimination and nonsymmetric eigenvalue problems.

\subsubsection{Example: Gaussian elimination}

As before, we wish to apply Gaussian elimination to the matrix~$\mtx{A} \in \M_n(\R)$.
Parker's method of random rotation (\Cref{sec:rand-rot}) avoids zero diagonal entries,
but it does not necessarily control the \hilite{growth} of the diagonal
entries---which is the main source of instability in the algorithm.

Sankar et al.~\cite{SST06:Smoothed-Analysis} observed that Gaussian elimination
admits a smoothed anlaysis.  For a parameter $\sigma > 0$, 
suppose that we \hilite{perturb} the input matrix:
\[
\mtx{A}_{\sigma} \coloneqq \mtx{A} + \sigma \mtx{G}
\quad\text{where $\mtx{G} \in \M_n(\R)$ has i.i.d.~$\normal(0,1)$ entries}.
\]
With high probability, Gaussian elimination on the matrix
${\mtx{A}}_{\sigma}$ never encounters a pivot that is larger than a fixed
polynomial in the dimension $n$ and the reciprocal $\sigma^{-1}$ of the
scale for the perturbation.
This property ensures that it is safe to apply Gaussian elimination
to the instance ${\mtx{A}}_{\sigma}$.

As a consequence, we could simply run Gaussian elimination on the
perturbed matrix $\mtx{A}_{\sigma}$ to obtain an approximate
\textsf{LU} factorization of the original matrix $\mtx{A}$.
This result demonstrates that we can sometimes avoid numerical problems
if we are willing to make a modest change in the problem data.

\subsubsection{Example: Pseudospectral shattering}

Consider a \hilite{complex} matrix $\mtx{A} \in \M_n(\C)$ that is diagonalizable,
but not necessarily normal.  The eigenvalue problem asks us to compute a full eigenvalue decomposition:
\[
\mtx{A} = \mtx{VDV}^{-1}
\quad\text{where $\mtx{D} \in \M_n(\C)$ is diagonal and $\mtx{V} \in \M_n(\C)$ is nonsingular.}
\]
This problem is challenging when the eigenvalues of the matrix are close together
or when the matrix $\mtx{V}$ of eigenvectors is badly conditioned,
because the computation is very sensitive to small perturbations of the data.
The matrix computation literature describes reliable algorithms
(based on \textsf{QR} iteration with shifts), but there was no
rigorous analysis to justify the use of these methods---until recently.

Banks et al.~\cite{BGKS22:Pseudospectral-Shattering} showed that smoothed analysis
offers a strategy for addressing these difficulties.  They considered a \hilite{complex} Gaussian
perturbation of the matrix on the scale $\sigma > 0$:
\[
\mtx{A}_{\sigma} \coloneqq \mtx{A} + \sigma \mtx{G}
\quad\text{where $\mtx{G} \in \M_n(\R)$ has i.i.d.~$\normal_{\C}(0,1)$ entries}.
\]
They proved that the matrix $\mtx{A}_{\sigma}$ remains diagonalizable,
the eigenvalues are well-separated, and the eigenvectors become well-conditioned.
More precisely, they demonstrate that the pseudospectrum of $\mtx{A}_{\sigma}$
has $n$ isolated components on a scale that is polynomial in $n^{-1}$ and $\sigma$.
See \Cref{fig:pseudoshattering} for a schematic.

As a consequence, Banks et al.~\cite{BGKS22:Pseudospectral-Shattering} were able to prove that the nonsymmetric eigenvalue problem
submits to algorithms based on spectral divide-and-conquer techniques,
which operate by recursively dividing the set of eigenvalues into two comparably sized groups.
Gaps between eigenvalues, introduced by the perturbation, %
are essential for the success of the procedure.
Later, also using smoothed analysis, Banks et al.~\cite{BGS22:Global-Convergence-I}
proved that a variant of shifted \textsf{QR} iteration also succeeds.
This work finally confirms that there are backward-stable algorithms that
can solve the nonsymmetric eigenvalue problem in $\mathcal{O}(n^3)$ time,
modulo polynomial factors in $\log(n)$.  We can quickly and reliably produce a very
accurate solution to an eigenvalue problem near the target instance.

\begin{figure}
    \centering
    \includegraphics{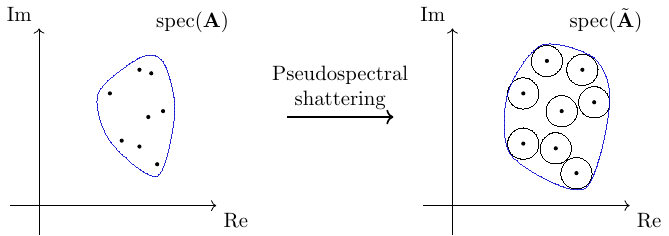}
    \caption{\textbf{Pseudospectral shattering.} After perturbation, the eigenvalues of the matrix are well-separated with high probability.}
    \label{fig:pseudoshattering}
\end{figure}

\subsection{Conclusions}

Our goal has been to isolate the probabilistic insights that play a role
in the design of randomized algorithms for matrix computations.
We have organized the methods thematically to find similarities
among algorithms that might seem different in concept:

\begin{enumerate}
\item	Monte Carlo approximation

\item	Random initialization

\item	Progress on average

\item	Randomized dimension reduction

\item	Randomized preconditioning

\item	General position

\item	Smoothed analysis
\end{enumerate}

\noindent
This list does not pretend to be comprehensive,
and the literature contains other strategies that fall outside of this taxonomy.
Undoubtedly, there are more techniques waiting to be discovered!

\section*{Acknowledgments}

JAT is grateful to Massimo Fornasier and the other organizers of the
2023 CIME Summer School on Machine Learning for the invitation to
give these lectures and to spend a week in Cetraro, Calabria.
(\textit{``Vi ravviso, o luoghi ameni!''})
JAT recognizes generous grant support from ONR Award N00014-18-1-2363,
NSF FRG Award 1952777, and the Caltech Carver Mead New Adventures
Fund that made possible the research, writing, and travel.
Additional financial and administrative support were provided
by CIME and by the Mathematics Department of TU-Munich.

AK expresses gratitude to the organizers of the 2023 CIME Summer School 
on Machine Learning and CIME for their support, which made her participation 
in the school and writing these notes possible.

We would like to thank Ethan Epperly for a close reading
and suggestions for improvements.

\printbibliography[title=References]

\end{document}